\documentclass[a4paper,11pt]{article}
\usepackage{aligned-overset}
\usepackage{cancel}
\usepackage{graphicx}        
\usepackage{multicol}        
\usepackage[bottom]{footmisc} 
\usepackage{amsmath}
\usepackage{empheq}
\usepackage{epsfig}
\usepackage{authblk}
\usepackage{float}
\usepackage{diagbox}
\usepackage{nicefrac}
\usepackage[version=4]{mhchem}
\usepackage[T1]{fontenc}
\usepackage{blkarray}
\usepackage{mathtools,amsthm,amsfonts}   
\usepackage{citesort}
\usepackage[style=numeric-comp,maxbibnames=99,bibencoding=utf8,giveninits=true]{biblatex}
\bibliography{elasticity_robust_nodal.bib}
\usepackage{helvet}  

\usepackage{newtxtext}
\usepackage[upint,varg]{newtxmath}

\usepackage{hhline}
\usepackage{tikz}

\usepackage{color}

\usepackage{bm}
\usepackage{nicefrac}
\usepackage[hmargin={28mm,28mm},vmargin={30mm,35mm}]{geometry}
\usepackage{array}
\usepackage{pdfpages}
\usepackage{enumitem}
\usepackage{bbold}
\usepackage{xparse}								
\usepackage{stmaryrd}							
\usepackage{sidecap}								
\usepackage{upgreek}
\usepackage{accents}
\allowdisplaybreaks
\usepackage{tcolorbox}
\usepackage{algorithmic}
\usepackage[colorlinks,allcolors={blue}]{hyperref}
\usetikzlibrary{shapes,arrows,positioning,calc}
\usepackage{alltt}
\usepackage[bbgreekl]{mathbbol}
\DeclareSymbolFontAlphabet{\mathbb}{AMSb}
\DeclareSymbolFontAlphabet{\mathbbl}{bbold}

\def\R{\mathbb R}

\def\P{\mathbb P}

\def\T{{\mathbf{T}}}

\def\<{\langle}
\def\>{\rangle}

\def\Chi{\raise .3ex \hbox{\large $\chi$}} 

\def\ov{\overline}

\def\normal{\mathbf{n}}

\def\tang{{\mathbf{\tau}}}
\def\trescacof{\text{g}}
\def\lan{\lambda_{\normal}}
\def\laf{{\boldsymbol{\lambda}}}
\def\lat{{\boldsymbol{\lambda}_{\tang}}}
\def\la{\lambda}
\def\lanh{\la_{h,n}}
\def\lath{{\boldsymbol{\lambda}_{h,\tang}}}

\def\x{{\bf x}}

\def\F{{\bf F}}

\def\cells{{\mathcal{M}}}
\def\faces{{\mathcal{F}}}
\def\nodes{{\mathcal{V}}}
\def\edges{{\mathcal{E}}}
\def\cK{\mathcal{K}}

\def\cF{\mathcal{F}}

\def\x{{\bf x}}

\def\a{\alpha}

\def\bu{\mathbf{u}}
\def\bq{\mathbf{q}}
\def\bv{\mathbf{v}}

\newcommand{\Kz}{{\mathcal{K}z}}
\newcommand{\Lz}{{\mathcal{L}z}}

\def\cE{\mathcal{E}}
\def\cF{\mathcal{F}}

\def\cK{\mathcal{K}}

\def\cM{\mathcal{M}}

\def\cV{\mathcal{V}}

\def\cW{\mathcal{W}}
\def\frakr{\mathfrak{r}}

\def\fC{\mathfrak{C}}

\def\blambda{\boldsymbol{\lambda}}

\def\Up{\Upsilon}
\def\bv{{\bf{v}}}

\def\dgsk{\uline{\boldsymbol{X}}_{{\bf grad},h}^{k}}
\def\ddgsk{\uline{\boldsymbol{X}}_{{\bf grad},h,0}^{k}}

\def\dgskS{\uline{\boldsymbol{X}}^{k,\ell,\star}_{\mathbf{grad},h}}
\def\ddgskS{\uline{\boldsymbol{X}}^{k,\ell,\star}_{\mathbf{grad},h,0}}
\def\dlskS{\uline{\boldsymbol{X}}^{k,\ell,\star}_{{\bf grad},K}}
\newcommand{\uS}[1]{\underline{\mathbf{#1}}^{\star}_{h}} 
\newcommand{\uSe}[2]{\underline{\mathbf{#1}}^{\star}_{#2,h}} 
\newcommand{\ukS}[1]{\underline{\mathbf{#1}}^{\star}_{K}} 
\newcommand{\ub}[1]{\underline{\mathbf{#1}}^{\bub}_{h}}  

\newcommand{\GS}[1]{\mathbb{G}^{k,\star}_{#1}}
\newcommand{\UpS}[1]{\Upsilon^{k+1,\star}_{#1}}
\newcommand{\IhS}{\underline{\mathbb{I}}^{\star}_{h}}
\newcommand{\Jsig}[1]{\jump{#1}_{\sigma}}
\newcommand{\Jn}[1]{\jump{#1}_{\sigma,\normal}}
\newcommand{\RhS}{\underline{\mathbf{R}}^{\star}_{h}}
\newcommand{\Pgp}[2]{P_{\gamma,+}\big(#1,#2\big)}
\newcommand{\Pt}[2]{\boldsymbol{P}_{\gamma,g}\big(#1,#2\big)}

\newcommand{\bPg}[2]{\mathbf{P}_{\gamma_{h}}\big(#1,#2\big)}
\newcommand{\continuousbPg}[2]{\mathbf{P}_{\gamma}\big(#1,#2\big)}

\def\krond{\delta_{k,\ell}}

\newcommand{\stabc}[2]{\mathbb{S}_{E,\nu,h}(#1,#2)}
\newcommand{\stabin}[2]{\mathbb{S}_{h}(#1,#2)}
\def\dsml{\mathcal{U}_{h}^{k,\ell}}
\def\Ms{\boldsymbol{M}_{h}^{\ell}}

\def\[{\partial}
\def\O{\Omega}
\def\tr{{\rm {tr}}}
\def\bc{\bbtau}

\def\djumph#1{\jump{#1}_{h}}
\def\djumphn#1{\jump{#1}_{h,\normal}}

\newcommand{\bspace}{\O\setminus\overline{\Gamma}}
\newcommand{\Ks}{{\mathcal{K}s}}
\newcommand{\Ke}{{\mathcal{K}e}}
\newcommand{\Le}{{\mathcal{L}e}}
\newcommand{\Ls}{{\mathcal{L}s}}
\def\Ksig{{K\!\sigma}}
\def\Lsig{{L\!\sigma}}

\newcommand{\uvec}[1]{\underline{\mathbf{#1}}}
\newcommand{\bvec}[1]{\boldsymbol{#1}}
\newcommand{\bo}[1]{\mathbf{#1}}
\newcommand{\wbo}[1]{\mathbf{#1}^{\star}}
\newcommand{\bx}{\bvec{x}}
\newcommand{\bos}[1]{\boldsymbol{#1}}

\newcommand{\bub}{{\mathrm{b}}}
\def\ext{\text{ext}}
\def\inte{\text{int}}
\def\R{\mathbb{R}}

\def\P{\mathbb{P}}

\def\I{\mathbb{I}}

\newcommand{\DIV}{\operatorname{\mathrm{div}}}

\newcommand{\Poly}[1]{\mathbb{P}^{#1}}

\newcommand{\Polybub}[2]{\Poly{#1,#2,\star}}

\newcommand{\Sigmaspace}{\mathbf{\Sigma}}    

\newcommand{\NORM}[2]{\|#2\|_{#1}}
\newcommand{\SEMINORM}[2]{|#2|_{#1}}

\def\O{\Omega}

\def\bbeps{\bbespilon}
\newcommand{\Id}{\mathrm{I}_3}

\newcommand{\jump}[1]{\llbracket #1 \rrbracket}

\newcommand{\email}[1]{\href{mailto:#1}{#1}}

\begingroup
    \makeatletter
    \@for\theoremstyle:=definition,remark,plain\do{%
        \expandafter\g@addto@macro\csname th@\theoremstyle\endcsname{%
            \addtolength\thm@preskip\parskip
            }%
        }
\endgroup

\newtheorem{theorem}{Theorem}
\numberwithin{theorem}{section}
\newtheorem{proposition}[theorem]{Proposition}
\newtheorem{lemma}[theorem]{Lemma}

\theoremstyle{remark}
\newtheorem{remark}[theorem]{Remark}
\theoremstyle{definition}

\def\div{{\rm \textbf{div}}}
\def\divScal{{\rm div}}

\definecolor{labelkey}{rgb}{0.6,0,1}

\usepackage{marginnote}

\usepackage{multirow}
\usepackage{parskip}
\usepackage{algorithmic}
\usepackage{pgfplots,pgfplotstable}
\pgfplotsset{compat=1.18}
\usepgfplotslibrary{groupplots}
\usepackage{subcaption}
\newcommand{\logLogSlopeTriangle}[5]
{
    \pgfplotsextra
    {
        \pgfkeysgetvalue{/pgfplots/xmin}{\xmin}
        \pgfkeysgetvalue{/pgfplots/xmax}{\xmax}
        \pgfkeysgetvalue{/pgfplots/ymin}{\ymin}
        \pgfkeysgetvalue{/pgfplots/ymax}{\ymax}

        \pgfmathsetmacro{\xArel}{#1}
        \pgfmathsetmacro{\yArel}{#3}
        \pgfmathsetmacro{\xBrel}{#1-#2}
        \pgfmathsetmacro{\yBrel}{\yArel}
        \pgfmathsetmacro{\xCrel}{\xBrel}

        \pgfmathsetmacro{\lnxB}{\xmin*(1-(#1-#2))+\xmax*(#1-#2)} 
        \pgfmathsetmacro{\lnxA}{\xmin*(1-#1)+\xmax*#1} 
        \pgfmathsetmacro{\lnyA}{\ymin*(1-#3)+\ymax*#3} 
        \pgfmathsetmacro{\lnyC}{\lnyA+#4*(\lnxA-\lnxB)}
        \pgfmathsetmacro{\yCrel}{\lnyC-\ymin)/(\ymax-\ymin)}

        \coordinate (A) at (rel axis cs:\xArel,\yArel);
        \coordinate (B) at (rel axis cs:\xBrel,\yBrel);
        \coordinate (C) at (rel axis cs:\xCrel,\yCrel);

        \draw[#5]   (A)-- node[pos=0.5,anchor=north] {\scriptsize{1}}
                    (B)-- 
                    (C)-- node[pos=-0.15,anchor=west] {\scriptsize{#4}} 
                    cycle;
    }
}

\usepackage{diagbox}
\usepackage{booktabs}

\makeatletter
\def\thm@space@setup{%
  \thm@preskip=\parskip \thm@postskip=0pt
}
\makeatother

\usepackage[normalem]{ulem}
\newcounter{corr}
\definecolor{violet}{rgb}{0.580,0.,0.827}
\newcommand{\corr}[3]{\typeout{Warning : a correction remains in page \thepage}
  \stepcounter{corr}
 	      {\color{blue}\ifmmode\text{\,\sout{\ensuremath{#1}}\,}\else\sout{#1}\fi}
              {\color{red}#2}
              {\color{violet} #3}
}

\newcommand{\klpair}[4]{
  \addplot [#2, line width=1.3pt, mark=#3, mark size=2.3pt] table[x=Ncbrt, y=H1rel] {data/#1_A.dat};
  \addplot [#2, line width=1.3pt, dashed, mark=#4, mark size=2.3pt, mark options={solid,fill=#2}]
    table[x=Ncbrt, y=LambdaErr] {data/#1_B.dat};
}

\newcommand{\klpairleg}[5]{
  \addplot [#2, line width=1.3pt, mark=#3, mark size=2.3pt] table[x=Ncbrt, y=H1rel] {data/#1_A.dat};
  \addlegendentry{$k{=}#5$, energy$\quad$}
  \addplot [#2, line width=1.3pt, dashed, mark=#4, mark size=2.3pt, mark options={solid,fill=#2}]
    table[x=Ncbrt, y=LambdaErr] {data/#1_B.dat};
  \addlegendentry{$k{=}#5$, multiplier$\quad$}
}

\pgfplotsset{
  everypanel/.style={
    width=7.6cm, height=6.6cm,
    xlabel={$N_{\mathrm{cell}}^{1/3}$}, ylabel={Relative error},
    xlabel style={font=\small}, ylabel style={font=\small},
    tick label style={font=\scriptsize},
  }
}

\newcommand{\panelBubbleLeg}[2]{
  \begin{loglogaxis}[everypanel, legend columns=3, legend to name=#2]
    \klpairleg{#1_k0l0}{black}{star}{*}{0}
    \klpairleg{#1_k1l1}{red}{square*}{triangle*}{1}
    \klpairleg{#1_k2l2}{blue}{x}{o}{2}
    \logLogSlopeTriangle{0.5}{0.4}{0.1}{1}{black};
    \logLogSlopeTriangle{0.5}{0.4}{0.1}{2}{black};
    \logLogSlopeTriangle{0.5}{0.4}{0.1}{3}{black};
  \end{loglogaxis}
}
\newcommand{\panelNoBubble}[1]{
  \begin{loglogaxis}[everypanel]
    \klpair{#1_k1l0}{red}{square*}{triangle*}
    \klpair{#1_k2l1}{blue}{x}{o}
    \logLogSlopeTriangle{0.5}{0.4}{0.1}{1}{black};
    \logLogSlopeTriangle{0.5}{0.4}{0.1}{2}{black};
    \logLogSlopeTriangle{0.5}{0.4}{0.1}{3}{black};
  \end{loglogaxis}
}

\begin{document}
\title{An Arbitrary-Order Augmented Lagrangian Discrete De Rham Scheme for Contact Mechanics with Tresca Friction}
\author[1,2]{{J\'er\^ome Droniou}\footnote{\email{jerome.droniou@cnrs.fr}}}
\author[1]{{Raman Kumar}\footnote{ \email{raman.kumar@umontpellier.fr}}}
\author[3]{{Roland Masson}\footnote{\email{roland.masson@inria.fr}}}
\affil[1]{IMAG, Univ. Montpellier, CNRS, Montpellier, France.}%
\affil[2]{School of Mathematics, Monash University, Australia.}%
\affil[3]{Universit\'e C\^ote d’Azur, Inria, CNRS, Laboratoire J.A. Dieudonn\'e, Team Galets, Nice, France.}%

\date{}
\maketitle
\vspace{-.5cm}
{\footnotesize{
\centering
}}

\begin{abstract}
We design and analyse an arbitrary-order augmented Lagrangian Discrete de Rham
polytopal scheme for contact mechanics in fractured media. The fractures are represented
as objects of codimension one in the matrix, a linear elastic medium, and a Tresca friction model is adopted when contact 
between fracture walls occurs.
Both the normal Signorini-type/traction constraint and the Tresca friction law are recast as projection identities, and are combined into a single augmented Lagrangian bilinear form. The displacement is discretised by a DDR space of arbitrary polynomial degree $k\ge 0,$ enriched on the fracture faces by bubble degrees of freedom that ensure stability via the existence of a discrete lifting between the Lagrange multiplier space and the discrete jump of the displacement along the fracture. We establish an abstract error estimate yielding optimal order convergence for the displacement and the Lagrange multiplier variables. 
Extensive numerical experiments confirm the theoretical analysis and demonstrate the practical accuracy and robustness of the scheme.
\end{abstract}

  \textbf{keywords:} contact mechanics, Tresca friction, fracture networks,
discrete de Rham method, polytopal method, augmented Lagrangian,
bubble stabilisation, error estimates.

\section{Introduction}
The numerical simulation of mechanical deformation in fractured or faulted porous media is central to several subsurface applications, including the geological storage of \ce{CO2} and the development of deep geothermal systems, where the assessment of fault reactivation and induced seismicity relies on an accurate description of the contact behaviour at the fracture interfaces~\cite{JHA2014,thm-bergen,GH19}. In these models, the fractures are represented as a network of planar surfaces of co-dimension one embedded
in the surrounding matrix, leading to the mixed-dimensional (fracture--matrix) formulations that have been the object of many recent works~\cite{MJE05,BoonNordbotten22,contact-norvegiens}; the rock deformation is then coupled, across the interfaces, with non-penetration and frictional conditions. Two features make the discretisation of such problems demanding. First, the contact conditions are intrinsically nonsmooth: they take the form of a variational inequality combining a unilateral (Signorini-type)
constraint in the normal direction with a Tresca friction law in the
tangential directions~\cite{oden-contact,HASLINGER1996313,book-contact-2005}. Second, realistic geological geometries involve fracture networks with corners, tips and intersections, together with strongly heterogeneous and distorted cells, which calls for discretisation methods supporting general polytopal meshes.
 
Several strategies have been developed to enforce contact conditions at the discrete level; we refer to~\cite{wriggers2006,Wohlmuth11} for general overviews. Penalty methods are simple and robust but inconsistent, and the large penalty parameter required to control penetration degrades the conditioning of the linear systems. Mixed formulations instead introduce a Lagrange multiplier representing the surface traction along the interface; they are consistent and physically natural, but the displacement and multiplier spaces must satisfy a discrete inf--sup (Babu\v{s}ka--Brezzi) compatibility condition, which constrains the choice of approximation spaces and, for equal-order or low-order pairs, must be restored through a dedicated stabilisation~\cite{Renard03,HR2007,Lleras-2009,burman2010interior}.
Nitsche's method avoids the inf--sup condition by eliminating the
multiplier and weakly imposing the contact conditions; optimal a priori estimates have been obtained in this framework for the Signorini problem and extended to Tresca and Coulomb friction in
the mixed-dimensional setting~\cite{Chouly2014,Chouly2017,CHLR2020,CEP20,beaude2023mixed}.
A unifying viewpoint is provided by the augmented Lagrangian formulation. Following the reformulation of the contact conditions introduced by Alart and Curnier~\cite{alart1991mixed}, the multiplier is updated through a projection of an augmented quantity built from the displacement jump and the multiplier itself. Eliminating the multiplier from this formulation recovers Nitsche's method, while retaining it yields a stabilised mixed method: this is the framework in which well-posedness and optimal error estimates for the obstacle and Signorini problems were established. More broadly, the augmented Lagrangian approach has been introduced as a unifying scheme for stabilised methods in computational mechanics~\cite{Burman.Hansbo.Larson:17,BHL2023}.
The augmented Lagrangian formulations have also been used to model frictional contact within fracture networks at the finite element level~\cite{NEJATI2016123}. The augmented Lagrangian approach keeps an explicit multiplier, which is of direct interest as a surface traction and convenient for coupling with a flow model, while improving the conditioning compared with a pure multiplier method and remaining amenable to efficient semismooth Newton solvers.

On the discretisation side, a variety of polytopal methods supporting general meshes have been designed, such as Hybrid High-Order~\cite{hho-book,dipietro-ern}, Virtual
Element~\cite{beirao-brezzi-marini,da2015virtual}, MultiPoint Stress Approximation~\cite{keilegavlen-nordbotten} and Discontinuous
Galerkin~\cite{hansbo-larson}, whose common motivation
is the geometric flexibility needed to mesh complex geological structures. Several of these have been adapted to contact mechanics through Virtual Element discretisations~\cite{WRR2016}, Nitsche-type Hybrid High-Order schemes~\cite{CEP20}, and fully discrete polytopal schemes for fracture networks~\cite{droniou2023bubble,BDMP:21}. The closest precursors of the present work are the fully discrete polytopal/Virtual Element schemes of~\cite{droniou2025analysis}, which treats Tresca frictional contact on fracture networks using a face-wise approximation of the Lagrange multiplier and a displacement space enriched by bubble degrees of freedom on the fractures to ensure inf--sup stability. These schemes handle the full geometric complexity of fracture networks, but they are built on a mixed formulation rather than an augmented Lagrangian one, and they are restricted to orders 1 or 2 approximations.
 
The discrete de Rham (DDR) framework~\cite{ddr-variant} offers a systematic way to design arbitrary-order, polytopal-compatible discretisations through discrete differential operators and potential reconstructions, supported by a well-developed theory of consistency, Poincar\'e inequalities andstability~\cite{di2025uniform,Di-Pietro.Droniou.ea:24}. It therefore provides a natural setting in which to obtain a high-order, analysis-backed scheme for contact mechanics on general meshes. An order two DDR scheme is designed for the Tresca friction recently in \cite{DKMR25}. To the best of our knowledge, however, no arbitrary-order augmented Lagrangian polytopal scheme for frictional
contact on fracture networks has yet been proposed and analysed. Filling in that gap is one of the novel contributions of this work.
Another important contribution is the first complete analysis of an augmented Lagrangian scheme for Tresca friction, which does not seem to be covered by the literature, even in classical finite element methods \cite[Section~8.4.3]{chouly2023finite}.

We propose and analyse an arbitrary-order augmented Lagrangian DDR scheme for a static, isotropic, linearly elastic medium subject to Tresca frictional contact at matrix--fracture interfaces. The displacement is discretised by a DDR space of arbitrary degree~$k\ge 0$, enriched by face-bubble degrees of freedom on the fracture faces. This enrichment distinguishes the resulting spaces and reconstruction operators from their standard counterparts, and it guarantees the discrete lifting between the Lagrange multiplier space of degree $\ell\in\{\max\{0,k-1\},k\}$ and the discrete jump of the displacement across the fracture network. The contact conditions are imposed through the Alart--Curnier projection, handling the normal Signorini constraint and the Tresca friction law within a single augmented Lagrangian bilinear form. Contrary to most finite element analyses of augmented Lagrangian methods, our analysis does not require the mesh to be quasi-uniform in the entire domain, only along the fracture network; this is particularly relevant in the context of arbitrary-order polytopal schemes as local refinements can be seamlessly done near complex geometries (e.g., fracture networks), while simpler regions (e.g., matrix far from fractures) can be properly captured using coarser meshes. Our main contributions are the
following:
\begin{itemize}
\item Design and complete analysis of an arbitrary-order polytopal scheme for the augmented Lagrangian formulation of contact mechanics. 
\item A discrete lifting between the Lagrange multiplier space and the discrete displacement jump across the fracture faces, which leads to well-posedness of the scheme.
\item An abstract error estimate (Theorem~\ref{thm:abs.err.bound}),
formulated in the spirit of the Gradient Discretisation
Method~\cite{gdm} through primal-consistency and limit-conformity
measures, which separate the discretisation error into independently
bounded approximation terms.
\item A comprehensive error analysis with an energy-norm convergence rate of order $\mathcal{O}(h^{\frakr+\frac{1}{2}})$~(Theorem~\ref{thm:error.estimate} with discrete jump~\eqref{eq:jump.n}), where $h$ denotes the mesh size, and $\frakr=k+\frac{1}{2}$ when $\ell=k$, while $\frakr=k$ when $\ell=k-1$.

\item A variant that removes a stabilisation term involving jumps of Lagrange multipliers between fracture faces. We only prove a suboptimal $\mathcal{O}(h^{\ell+\frac12})$ estimate for this variant (Theorem~\ref{thm:error.estimate.sub} with projected jump~\eqref{eq:jump.modify}), but it has the important advantage to allow for the local elimination of Lagrange multipliers, and, for $\ell = 0$, coincides with the mixed schemes of \cite{droniou2023bubble,droniou2025analysis,DKMR25}.

\item Numerical tests on tetrahedral and generalised hexahedral (polytopal) meshes that confirm the theoretical rates,  and on a network of intersecting fractures that show the efficiency of the higher-order schemes and the robustness of the Newton method with respect to the parameter in the Alart--Curnier projection.
\end{itemize}
A key technical ingredient of the inf--sup analysis is a discrete lifting operator constructed directly from the DDR degrees of freedom with face enrichment. In contrast, the finite element lifting
of~\cite{Burman.Hansbo.Larson:17} relies on a nodal cut-off function
and on a quadrature rule with strictly positive interior weights, which imposes degree-dependent limitations that our approach does not have.

The rest of the paper is organised as follows.
Section~\ref{sec:model} introduces the contact mechanics model with Tresca friction and its augmented Lagrangian formulation. Section~\ref{sec:scheme} sets up the bubble-enriched DDR spaces and reconstruction operators, establishes the discrete Korn inequality, and defines the augmented Lagrangian DDR scheme. Section~\ref{Sec:main_results} states the abstract error estimate and deduces the optimal convergence rates. The theoretical results are validated by extensive numerical tests presented in Section \ref{sec:numerical.exp.}.
Section~\ref{sec:Existence.uniqueness} is devoted to the well-posedness of the scheme, including the discrete lifting~(Lemma~\ref{rem:lifting}), the existence of a discrete solution, and its uniqueness. Section~\ref{sec:proof_error} contains the proof of the abstract error estimate. Section~\ref{sec:nostab} presents a stabiliser-free variant, for which a suboptimal bound is obtained by projecting the discrete jump. Finally, the paper is concluded in Section~\ref{sec:conclusion}.

\section{Model}\label{sec:model}
Let $\O\subset\R^3$ be a polyhedral domain having a co-dimension-$1$ fracture network $\Gamma$. The network is the union of finitely many planar fractures indexed by $\bo{I}$:
\begin{equation*}
\overline{\Gamma}=\bigcup_{i\in\bo{I}}\overline{\Gamma}_i,
\end{equation*}
where each $\Gamma_i$ is an open, simply connected polygon lying in some plane of $\R^3$. Intersections between distinct fractures, if any, are confined to their boundaries; this assumption entails no loss of generality. Throughout, we restrict ourselves to fracture configurations satisfying the following geometric requirement: every connected component of $\O\setminus\Gamma$ must touch $\partial\O$ on a set of positive two-dimensional measure. In particular, $\Gamma$ cannot disconnect any subregion of $\O$ from the outer boundary $\partial \O$.
To each side of a fracture, viewed from the matrix $\O\setminus\overline{\Gamma}$, we attach a label $+$ or $-$, together with an outward unit normal $\bo{n}^\pm$; by construction, these normals are opposite:
\begin{equation*}
\bo{n}^++\bo{n}^-=\bo{0}.
\end{equation*}
In the sequel, boldface notation is used to denote vector quantities, vector-valued functions, matrix-valued functions, their associated function spaces, and corresponding discrete spaces. For instance, $\boldsymbol{H}^1(X)=H^1(X)^3$. Under homogeneous Dirichlet boundary conditions, we consider the displacement space
\begin{equation*}
\bo{U}_0\coloneqq \boldsymbol{H}^1_0(\O\setminus\overline{\Gamma}),
\end{equation*}
where functions may be discontinuous across the fracture set $\Gamma$. This space is equipped with the norm $\lVert\bo{v}\rVert_{H^1(\O\setminus\overline{\Gamma})}=\lVert \nabla\bo{v}\rVert_{L^2(\O\setminus\overline{\Gamma})}.$ For $\bo{u} \in \bo{U}_0$, the oriented jump across $\Gamma$ is defined by
\begin{equation*}
\jump{\bo{u}}\coloneqq\gamma^+\bo{u}-\gamma^-\bo{u}.
\end{equation*}
where $\gamma^\pm$ denote the trace operators taken on the $\pm$ side of $\Gamma$. Projecting this jump onto $\bo{n}^+$ yields its normal component $\jump{\bo{u}}_{\bo{n}}\coloneqq\jump{\bo{u}}\cdot\bo{n}^+$. Normal traces of $H_{\DIV}(\bspace)$ fields on either side of the fracture are denoted analogously by $\gamma_{\bo{n}}^\pm$. The operator $\bbeps$, referred to as the symmetric gradient, is defined on $\bo{U}_0$ by $\bbeps(\bo{v})\coloneqq\frac{1}{2}\left(\nabla \bo{v}+(\nabla \bo{v})^T\right)$.

In strong form, the contact mechanics model is expressed as 
\begin{subequations}\label{mode_tresca_1} 
\begin{empheq}[left=\empheqlbrace\,]{alignat=2}
   -\div \bbsigma(\bu)&= \mathbf{f}  &&\quad \mbox{ on } \Omega{\backslash}\ov\Gamma,\\
	 \bbsigma(\bu)&=\frac{E}{1+\nu}\left(\bbeps(\bu)+\frac{\nu}{1-2\nu} (\divScal\, \bu )\Id \right) &&\quad \mbox{ on }  \Omega{\backslash}\ov\Gamma,\\
			\gamma_\normal^+\bbsigma(\bu)+\gamma_\normal^-\bbsigma(\bu) &= \mathbf{0} &&\quad \mbox{ on }  \Gamma,\\
			T_{\normal}(\bu) \leqslant 0,~ \jump{\bu }_{\normal} \leqslant 0,~ &\jump{\bu }_{\normal}T_{\normal}(\bu) =0 && \quad \mbox{ on }  \Gamma,
			\label{eq:contact.conditions.strong}\\
              \left|\T_{\tang}(\bu)\right| &\leqslant \trescacof && \quad \mbox{ on }  \Gamma,\label{eq:tang.contact.condition_1}\\
			 \T_{\tang}(\bu)  \cdot  \jump{\bu }_{\tang} +\trescacof \left| \jump{\bu }_{\tang}\right|&=0 && \quad \mbox{ on }  \Gamma,\label{eq:tang.contact.condition_2}\\
			\bu&=\mathbf{0}  &&\quad \mbox{ on } \partial\Omega,
    \end{empheq}
	\end{subequations}    
with a Tresca threshold $g\geq 0,$ Poisson ratio $\nu$, and Young’s modulus $E$ satisfying $\frac{E}{1+\nu}\in [\mu_{1},\mu_{2}]$ with $0 < \mu_{1} \leq\mu_{2}$, the $3\times 3$ identity matrix $\Id$, and normal and tangential surface tractions $T_{\normal}(\bu)=\gamma_{\normal}^+\bbsigma(\bu)\cdot{\normal}^+$ and $\T_{\tang}(\bu) = \gamma_\normal^+\bbsigma(\bu) - T_\normal(\bu) \normal^+.$
We assume that the external force term $\mathbf{f}$ belongs to $\boldsymbol{L}^2(\O).$ 

Setting $\lan=-T_{\normal}(\bu)$, the contact conditions \eqref{eq:contact.conditions.strong} are equivalent to Alart and Curnier reformulation \cite{alart1991mixed} defined as
\begin{equation}\label{Alart_cur_def}
\lan:= \gamma^{-1}\big[\jump{\bu}_{\normal}+\gamma \lan\big]_{+}
\end{equation}
where $\gamma>0$ is a user-chosen parameter and $[\cdot]_{+}=\max\{0,\cdot\}.$
For brevity, we set
$$
\Pgp{\jump{\bo{u}}_{\normal}}{\lan} := \big[\jump{\bu}_{\normal}+\gamma \lan\big]_{+}.
$$
For any $r\in \mathbb{R}^{+},$ we introduce the notation $[\cdot]_{r}$ for the orthogonal projection onto $\mathcal{B}(0,r) \subset \mathbb{R}^3,$ where $\mathcal{B}(0,r)$ is the closed ball centered at origin $\mathbf{0}$ and a radius $r.$ This projection can be defined analytically as: For all $\mathbf{a} \in \mathbb{R}^3,$
$$[\mathbf{a}]_{r}:=\begin{cases}
    \mathbf{a} & \text{ if }\; \left|\mathbf{a}\right| \leq r,\\
    r\frac{\mathbf{a}}{\left|\mathbf{a}\right|} & \text{ otherwise}.
\end{cases}$$
Setting $\lat:=-\T_{\tang}(\bu),$ the Tresca friction conditions \eqref{eq:tang.contact.condition_1}-\eqref{eq:tang.contact.condition_2} are equivalent to the formulation~\cite[Eq.~(9)]{Chouly2014} defined as 
\begin{equation}\label{eq:tang.lagrange.multiplier}
\lat:=\gamma^{-1}\big[\jump{\bu}_{\tang}+\gamma\lat\big]_{\gamma g}. 
\end{equation}
Setting
$$\Pt{\jump{\bu}_{\tang}}{\lat}:=\big[\jump{\bu}_{\tang}+\gamma\lat\big]_{\gamma g}$$
and combining \eqref{Alart_cur_def} and \eqref{eq:tang.lagrange.multiplier} we introduce the notation
\begin{equation}\label{eq:combine.Pg.def}
\continuousbPg{\jump{\bu}}{\laf}:=\Pgp{\jump{\bu}_{\normal}}{\lan}\normal^{+}+\Pt{\jump{\bu}_{\tang}}{\lat}.
\end{equation}
The augmented Lagrangian formulation for the contact problem \eqref{mode_tresca_1} reads: Find $(\bu,\laf) \in \bo{U}_0\times \boldsymbol{L}^{2}(\Gamma)$ such that, for all $(\bo{v}, \boldsymbol{\mu})\in \bo{U}_0\times \boldsymbol{L}^{2}(\Gamma)$, 
\begin{subequations}
\begin{align}
 \int_{\Omega}\bbsigma(\bu):\bbeps(\bv)+\int_{\Gamma} \gamma^{-1}\continuousbPg{\jump{\bo{u}}}{\laf} \cdot \jump{\bv}&=\int_{\Omega}\mathbf{f}\cdot \bv,\label{Form1_nonlinearform_a}\\
 \int_{\Gamma} \Big[\gamma\laf-\continuousbPg{\jump{\bo{u}}}{\laf}\Big]\cdot \boldsymbol{\mu}&=0 \label{Form1_nonlinearform_c}.
\end{align}
\end{subequations}
Note that, from the variational formulation, the Lagrange multiplier satisfies $\laf= -\gamma_\normal^+\bbsigma(\bu) = \gamma_\normal^-\bbsigma(\bu).$ 
Using the standard decomposition into normal and tangential components, we write
$$ \laf= \lan \normal^{+} +\lat.$$
Adding together \eqref{Form1_nonlinearform_a}-\eqref{Form1_nonlinearform_c} leads to the following equivalent formulation:
\begin{align}
 \int_{\Omega}\bbsigma(\bu):\bbeps(\bv)+\mathbb{C}((\bu,\laf);(\bv,\boldsymbol{\mu})) &=\int_{\Omega}\mathbf{f}\cdot \bv \quad \forall (\bo{v}, \boldsymbol{\mu})\in \bo{U}_0\times \boldsymbol{L}^{2}(\Gamma), \label{Form1_nonlinearform} 
\end{align}
where the nonlinear contact operator is defined as 
\begin{equation*}
\mathbb{C}((\bu,\laf);(\bv,\boldsymbol{\mu}))=\int_{\Gamma}\gamma^{-1}\continuousbPg{\jump{\bo{u}}}{\laf}\cdot\Big[ \jump{\bv}-\gamma \boldsymbol{\mu} \Big] + \int_{\Gamma} \gamma \laf\cdot \boldsymbol{\mu}.
\end{equation*}
 
\section{Scheme}\label{sec:scheme}

\subsection{Mesh}\label{Subsec:Mesh}
We consider a polytopal mesh of the computational domain $\Omega$ conforming to the fracture network $\Gamma$ (see \cite{ddr-variant}), in the sense that no mesh cell is intersected by $\Gamma$ in its interior. The sets of cells, faces, edges, and vertices are denoted by $\cells$, $\faces$, $\edges$, and $\nodes$, respectively. For $z\in \cells\cup \faces\cup \edges,$ we denote by $h_{z}$~(resp. $|z|$) the diameter of $z$~(resp. measure in the corresponding dimension).

For a family of mesh entities ${\cal Y}\in\{\cells,\faces,\edges,\nodes\}$ and an entity $z\in\cells\cup\faces\cup\edges\cup\nodes$, we let ${\cal Y}z$ denote the subset of elements in ${\cal Y}$ that either contain $z$ in their boundary or belong to the boundary of $z$. For instance, $\faces_K$ denotes the set of faces of a cell $K\in\cells$, while $\cells_s$ denotes the set of cells having a vertex $s\in\nodes$ on their boundary. For ${\cal Y}\in\{\faces,\edges,\nodes\}$, we further denote by ${\cal Y}^{\mathrm{int}}$ and ${\cal Y}^{\mathrm{ext}}$ the sets of interior and boundary entities, respectively.

Let $h_{\Gamma}\coloneqq \max\limits_{\sigma\in\faces_{\Gamma}} h_{\sigma},$ where $\faces_{\Gamma}\subset\faces^{\mathrm{int}}$ denotes the set of faces belonging to the fracture network. We assume that the mesh is quasi-uniform on $\Gamma$, in the sense that there exists a constant $c>0,$ independent of $h_{\Gamma}$, such that
\begin{equation}\label{eq:quasi.uniform.fracture}
ch_{\Gamma}\le h_{\sigma}
\qquad \forall \sigma\in\faces_{\Gamma}.
\end{equation}
Note that this quasi-uniformity assumption on $\Gamma$ is much less stringent than a quasi-uniformity of the entire mesh, and is compatible with applications in which the fracture network is often meshed finely while fracture-free regions of the matrix can be meshed more coarsely.
By the conformity of the mesh with the fracture network, we can  write
$$\overline{\Gamma}= \bigcup\limits_{\sigma \in \faces_{\Gamma}}\overline{\sigma}.$$
For each cell $K\in\cells$ and for every $\sigma\in\faces_K$, we denote by $\normal_{\Ksig}$ the unit outward normal to $K$ on $\sigma.$ Furthermore, for any face $\sigma \in \faces$ and any edge $e\in \edges_{\sigma},$ the symbol $\normal_{\sigma e }$ denotes the unit outward normal to $\sigma$ on $e$ in the planed spanned by $\sigma$. 

Throughout the analysis, we assume that the Tresca threshold $g$ is a piecewise constant on $\cF_{\Gamma}.$

For any mesh entity $X\in\cells\cup\faces\cup\edges$ and any integer $r\ge0$, let $\Poly{r}(X)$ denote the space of scalar-valued polynomials of total degree at most $r$ on $X$. We denote by $\Pi_X^r$ the $L^2(X)$-orthogonal projector onto $\Poly{r}(X)$, and use the same notation componentwise for vector- and tensor-valued functions. Given a collection of mesh entities $\mathcal X\in\{\cells,\faces_\Gamma\}$, we write $\Poly{r}(\mathcal X)$ for the corresponding broken polynomial space of degree at most $r$, and $\Pi_{\mathcal X}^r$ for the associated broken $L^2$-orthogonal projector.

Throughout the paper, the mesh is assumed to satisfy the regularity requirements stated in \cite[Definition~1.9]{hho-book}. The notation $a\lesssim b$ (respectively, $a\gtrsim b$) indicates that $a\le Cb$ (respectively, $a\ge Cb$), where the constant $C>0$ is independent of the mesh size and depends only on $\Omega$, $\Gamma$, the mesh regularity parameter, and, where applicable, on the material parameters (Young's modulus and Poisson's ratio) and the load $\mathbf f$.

\subsection{Discrete spaces}\label{Subsec: discrete.space}

\subsubsection{Displacement space}

To accommodate displacement jumps across matrix--fracture interfaces, the displacement unknowns are allowed to be discontinuous across the fracture network while remaining continuous across mesh entities not intersected by fractures. To formalize this construction, for each mesh entity $z\in\cF\cup\cE\cup\cV$ and each cell $K\in\cM_z$ containing $z,$ we define ${\Kz}$ as the set of cells containing $z$ and on the same connected side of $K$ with respect to the fracture network. Thus, for any two cells $K,L\in\cM_z$, one has $\Kz=\Lz$ if and only if $K$ and $L$ lie on the same side of $\Gamma.$
We associate with each equivalence class $\Kz$ a displacement DOF, denoted by $\bv_{\Kz}$. Consequently, cells located on the same side of the fracture share the same unknown, whereas cells separated by $\Gamma$ possess independent DOFs, thereby allowing for displacement discontinuities across fractures. In the absence of a fracture through $z$, then $\Kz=\cells_{z}$ and 
the set $\cM_z$ consists of a single equivalence class, yielding a unique displacement unknown at $z$~(still denoted by $\bv_{\Kz}$) and recovering the standard continuity of the displacement field.

The numerical scheme \eqref{Form1_nonlinearform} is based on the vector-valued discrete gradient space of the Discrete De Rham complex \cite{ddr-variant}, enriched with face bubbles on the positive side of $\Gamma.$ These bubble corrections are included to ensure the existence of a discrete lifting between the Lagrange multiplier space and the discrete displacement jump across the fracture faces, see Lemma~\ref{rem:lifting}; this lifting will be used to establish the stability and  a priori estimates of the discrete scheme.

Let us fix $k\ge 0$ a polynomial degree, and $\ell\in\{\max\{0,k-1\},k\}$, that is, if $k=0$ then $\ell=0$ and, if $k\ge 1$, $\ell$ is either $k-1$ or $k$.  For any $\sigma\in\faces$, we define a space of bubble vector-valued polynomials of degree $\le \ell$ by
$$\Polybub{k-1}{\ell}(\sigma)^3:= \Poly{k-1}(\sigma)^{3} \oplus \krond \; \Poly{\ell \perp k-1}(\sigma)^3,
$$
where $\Poly{\ell \perp k-1}(\sigma)$ is the orthogonal complement of $\Poly{k-1}(\sigma)$ in $\Poly{\ell}(\sigma)$, and $\krond$ denotes the Kronecker delta. 
A polynomial $\wbo{w}\in\Polybub{k-1}{\ell}(\sigma)^{3}$ is thus uniquely decomposed as
\begin{equation*}
\text{$\wbo{w}=\bo{w}+\krond\; \bo{w}^\bub$ with $\bo{w}\in\Poly{k-1}(\sigma)^3$ and $\bo{w}^\bub\in \Poly{\ell \perp k-1}(\sigma)^3$}.
\end{equation*}
For each $K\in\cells$, we denote by
\begin{equation*}
	\faces_{\Gamma,K}^+ = 
	\Big\{ \sigma \in \faces_{\Gamma} \cap \faces_{K}\,:\, \normal_{\Ksig} \cdot \normal^+  > 0  \Big\}
\end{equation*}
the set of faces of $K$ that lie on the positive side of $\Gamma$ (that set is empty if $K$ does not touch $\Gamma$ or touches it from its negative side). 

The global discrete space with face bubble enrichment is described as 
\begin{align*}
\begin{split}
\dgskS\coloneqq\{\uS{v}=&\big((\bo{v}_{\Ks})_{K\in\cM,~s\in \mathcal{V}_K},(\bo{v}_{\Ke})_{K\in\cM,~e\in \mathcal{E}_K},(\bo{v}_{\cK\sigma})_{K\in\cM,~\sigma\in\cF_K},(\krond \, \bo{v}_{K\sigma }^{\bub})_{K\in\cM,~\sigma \in \mathcal F_{\Gamma,K}^+},\\ &~(\bo{v}_K)_{K\in\cM}\big):\bo{v}_{\Ks}\in \mathbb R^3,~\bo{v}_{\Ke}\in \Poly{k-1}(e)^{3},~\bo{v}_{\cK\sigma}\in\Poly{k-1}(\sigma)^{3},\\ &\qquad \quad \bo{v}^{\bub}_{\Ksig}\in \Poly{\ell \perp k-1}(\sigma)^3,\text{ and }\bo{v}_K\in\Poly{k-1}(K)^{3}\},
\end{split}
\end{align*}
and its subspace that incorporates homogeneous Dirichlet boundary conditions on $\partial\O$ is
\begin{align*}
\ddgskS\coloneqq\{\uS{v}\in\dgskS\colon \bo{v}_{\Ks}=0 ~\text{if }s\in\cV^{\ext},~\bo{v}_{\Ke}=0\text{ if }e\in\cE^{\ext},\bo{v}_{\cK\sigma}=0 \text{ if }\sigma\in \cF^{\ext}\}.
\end{align*}
Now, any $\uS{v}\in\ddgskS$ decomposes
into bubble-free and bubble DOFs:
\begin{align}\label{eq:global.decomp}
  \uS{v} &= \uvec{v}_h+\ub{v}\\
  \text{ with } \uvec{v}_h &:= \big((\bo{v}_{\Ks})_{K\in\cM,~s\in \mathcal{V}_K},(\bo{v}_{\Ke})_{K\in\cM,~e\in \mathcal{E}_K},(\bo{v}_{\cK\sigma})_{K\in\cM,~\sigma\in\cF_K},(0)_{K\in\cM,~\sigma \in \mathcal F_{\Gamma,K}^+},(\bo{v}_K)_{K\in\cM}\big)\nonumber\\
  \text{ and }\ub{v}&:= \bigl(0,\;0,\;(\krond \,\bo{v}^{\bub}_{\Ksig})_{K\in\cM,\,\sigma \in \mathcal F_{\Gamma,K}^+}, \;0\bigr).
  \nonumber
\end{align}
We denote by $\ddgsk$ the space spanned by bubble-free vectors, which corresponds to the original DDR space:
$$
\ddgsk\coloneq \{\uS{v}\in\ddgskS\,:\,\bo{v}^{\bub}_{\Ksig}=\bo{0}\quad\forall K\in\cells\,,\;\forall\sigma\in\faces_K\}.
$$

For each $K\in\cM$, the local discrete space contains the components attached to a cell, its faces, edges, and vertices:
\begin{align*}
\begin{split}
\dlskS\coloneqq\{\ukS{v}&=\big((\bo{v}_{\Ks})_{s\in \mathcal{V}_K},(\bo{v}_{\Ke})_{e\in \mathcal{E}_K},(\bo{v}_{\cK\sigma})_{\sigma\in\cF_K},(\krond\, \bo{v}_{K\sigma}^{\bub})_{\sigma \in \mathcal F_{\Gamma,K}^+}, \bo{v}_K\big):~\bo{v}_{\Ks}\in \mathbb R^3,\\&\qquad \bo{v}_{\Ke}\in\Poly{k-1}(e)^{3},~{\bo{v}}_{\cK\sigma}\in\Poly{k-1}(\sigma)^{3},~\bo{v}_{K\sigma }^{\bub}\in \Poly{\ell \perp k-1}(\sigma)^3 \text{ and }\bo{v}_K\in\Poly{k-1}(K)^{3}\}.
\end{split}
\end{align*}
At the local level, any $\ukS{v} \in \dlskS$ admits an analogous decomposition to \eqref{eq:global.decomp}
\begin{equation*} 
\ukS{v} := \uvec{v}_{K} + \uvec{v}_{K}^{\bub},
\end{equation*}
where $\uvec{v}_{K}$ is the local element without a face bubble.

\subsubsection{Lagrange multiplier space}
The discrete space for the Lagrange multiplier is made of piecewise polynomial functions on fracture faces:
\begin{equation*}
\Ms \coloneqq\{\laf_h\in \boldsymbol{L}^2(\Gamma):~\laf_\sigma\coloneqq(\laf_h)_{|_\sigma} \in \Poly{\ell}(\sigma)^3\;\forall\sigma\in\cF_{\Gamma}\}.
\end{equation*}
The normal and tangential components of $\laf_{h}\in\Ms$ are defined by $$\lanh= \laf_{h}\cdot\normal^{+},\qquad \lath= \laf_{h}-\lanh\normal^{+}.$$

\subsection{Interpolator}
 We denote by $\boldsymbol{C}^0_0(\overline{\O}\setminus\Gamma)$ the set of continuous functions $\overline{\O}\setminus\Gamma\to\R^3$ that vanish on $\partial\O$.
The interpolator $\IhS:\boldsymbol{C}^0_0(\overline{\O}\setminus\Gamma)\rightarrow \ddgskS$ is defined as follows: For all $\bo{v}\in \boldsymbol{C}^0_0(\overline{\O}\setminus\Gamma)$,
\begin{align*}
\IhS\bo{v}\coloneqq &\big((\I_{\Ks}\bo{v})_{K\in\cM,~s\in \mathcal{V}_K},(\I_{\Ke}\bo{v})_{K\in\cM,~e\in \mathcal{E}_K},(\I_{\cK\sigma}\bo{v})_{K\in\cM,~\sigma \in \cF_{K}},\\ &\qquad
(\krond \I_{\Ksig}^{\bub}\bo{v})_{K\in\cM,~\sigma \in \mathcal F_{\Gamma,K}^+},~
(\I_K\bo{v})_{K\in\cM}\big),
\end{align*}
where
\begin{subequations}\label{eq:DOF.Ih}
\begin{alignat}{2}
\I_{\Ks}\bo{v}&=\bo{v}|_K(\bx_s)&&\qquad\forall s\in\cV_K,~K\in\cM,\notag\\
\I_{\cK e}\bo{v}&=\Pi^{k-1}_e(\bo{v}|_K)&&\qquad\forall e\in\cE_K,~K\in\cM,\notag\\
\I_{\cK\sigma}\bo{v}&=\Pi^{k-1}_\sigma(\bo{v}|_K)&&\qquad\forall \sigma\in\cF_K,~K\in\cM,\notag\\
\I_{\Ksig}^{\bub}\bo{v}&=\Pi^{\ell}_{\sigma}(\bo{v}|_K-\Pi^{k-1}_\sigma(\bo{v}|_K))&&\qquad\forall \sigma \in \mathcal F_{\Gamma,K}^+,~K\in\cM,\notag\\
\I_{K}\bo{v}&=\Pi^{k-1}_K\bo{v} &&\qquad\forall K\in\cM.\notag 
\end{alignat}
\end{subequations}
Next, any $\IhS\bo{v}\in \ddgskS$ admits an analogous decomposition to \eqref{eq:global.decomp}
\begin{align}\label{eq:IhS.decomposition}
\IhS\bo{v}:=\underline{\mathbb I}_h\bo{v} + \underline{\mathbb I}_h^{\bub}\bo{v}\qquad\forall \bo{v} \in \boldsymbol{C}^0_0(\overline{\O}\setminus\Gamma).
\end{align}

\subsection{Discrete operators}\label{discrete_opeators}
For each edge $e\in\cE_K$, define the local edge reconstruction as 
$\Upsilon_{Ke}^{k+1}:\ddgsk\rightarrow \Poly{k+1}(e)^3$ such that
\begin{equation}\label{eq:def.UpS.Ke}
\Upsilon_{Ke}^{k+1}\underline{\bv}_h(\bx_s) = \bo{v}_{\Ks} \quad \forall s\in\cV_e\,,\quad
\Pi^{k-1}_e\Upsilon_{Ke}^{k+1}\underline{\bv}_h=\bo{v}_{\Ke}.
\end{equation}
Next, for any face $\sigma\in\cF_K$, we introduce the local face gradient reconstruction
$\mathbb{G}^{k}_{\Ksig}:\ddgsk\rightarrow \Poly{k}(\sigma)^{3\times 2}$ and face displacement 
reconstruction $\Upsilon^{k+1}_{\Ksig}:\ddgsk\rightarrow \Poly{k+1}(\sigma)^3$ in the following 
manner: 
\begin{equation*}
\int_\sigma \mathbb{G}^{k}_{\Ksig}\uvec{v}_h : \bos{\xi} = -\int_\sigma \bo{v}_{\cK\sigma} \cdot \DIV(\bos{\xi}) 
+ \sum_{e\in\cE_\sigma} \int_e \Upsilon^{k+1}_{Ke}\uvec{v}_h \cdot \bos{\xi}{\bf n}_{\sigma e} \qquad \forall \bos{\xi}\in\Poly{k}(\sigma)^{3\times 2},
\end{equation*}
and
\begin{equation}\label{eq:Up.face}
\int_\sigma \Upsilon^{k+1}_{\Ksig}\uvec{v}_h \cdot \DIV(\bos{\eta}) = -\int_\sigma \mathbb{G}^{k}_{\Ksig}\uvec{v}_h : \bos{\eta} + \sum_{e\in\cE_\sigma} \int_e \Upsilon^{k+1}_{Ke}\uvec{v}_h\cdot \bos{\eta}{\bf n}_{\sigma e}\quad \forall \bos{\eta}\in \big[(\bx-{\bx}_\sigma)\Poly{k+1}(\sigma)\big]^3,
\end{equation}
where $\bx_\sigma$ is such that $\sigma$ contains a disc centered at $\x_\sigma$ and of radius $\gtrsim h_\sigma$, $\big[(\bx-{\bx}_\sigma)\Poly{k+1}(\sigma)\big]^3$ is considered as a space of $3\times 2$ matrices, and the divergences of matrices are taken row-wise.

For all elements $K\in\cells,$ we construct a second face potential $\UpS{\Ksig}$ by adding, for faces $\sigma\in\faces_{\Gamma,K}^+$ of $K$ on the positive side of the fracture, a  correction to $\Upsilon_{\Ksig}^{k+1}$:
\begin{align}\label{eq:UpS.face}
\UpS{\Ksig}\uS{v}= \left\{\begin{array}{ll}
  \Upsilon^{k+1}_{\Ksig}\uvec{v}_h + \krond\Big[\bo{v}_{\cK\sigma} + \bo{v}^{\bub}_{\Ksig} -\Pi^\ell_{\sigma}\Upsilon^{k+1}_{\Ksig}\uvec{v}_h\Big]&\text{ if $\sigma\in\faces_{\Gamma,K}^+$},\\
  \Upsilon^{k+1}_{\Ksig}\uvec{v}_h &\text{ if $\sigma\notin\faces_{\Gamma,K}^+$.}
  \end{array}\right.
\end{align}
Note that if $\ell\not=k$, then we have 
 \begin{equation}\label{eq:UpS.face_1}
 \UpS{\Ksig}\uS{v}=\Upsilon^{k+1}_{\Ksig}\uvec{v}_h.
\end{equation}
Then, for any element $K\in\cM$, we introduce the enriched local element  gradient reconstruction $\GS{K} : \ddgskS\rightarrow \Poly{k}(K)^{3\times 3}$ and the enriched element displacement reconstruction $\UpS{K} :\ddgskS\rightarrow \Poly{k+1}(K)^3$ in the following manner:
\begin{equation}\label{eq:G1K_enhance}
\int_K \GS{K}\uS{v} : \bos{\xi} = -\int_K \bo{v}_K \cdot \DIV(\bos{\xi}) + \sum_{\sigma \in \cF_{K}} \int_\sigma \UpS{\Ksig} \uS{v} \cdot \bos{\xi}{\bf n}_{\Ksig}\quad \forall \bos{\xi}\in\Poly{k}(K)^{3\times 3},
\end{equation}
and
\begin{equation*}
\int_K \UpS{K}\uS{v} \cdot \DIV(\bos{\eta}) = -\int_K \GS{K}\uS{v}: \bos{\eta} + \sum_{\sigma\in\cF_K} \int_\sigma \UpS{\Ksig}\uS{v} \cdot \bos{\eta}{\bf n}_{\Ksig} \quad \forall \bos{\eta}\in\big[(\bx-\bx_K)\Poly{k+1}(K)\big]^3,
\end{equation*}
where $\bx_K$ is a point such that $K$ contains a ball of center $\bx_K$ and radius $\gtrsim h_K$.

Let $\sigma\in\cF_{\Gamma}$ be a fracture face, and $K$~(resp.~$L$) is the cell on the positive~(resp.~negative) side of $\Gamma$. We  define the discrete displacement jump operator on $\sigma$ as $\Jsig{\cdot}: \ddgskS \to \Poly{k+1}(\sigma)^3$ such that, for all $\uS{v}\in \ddgskS$,
\begin{align} 
\Jsig{\uS{v}}:=\UpS{\Ksig}\uS{v} - \UpS{\Lsig}\uS{v}=\Upsilon_{\Ksig}^{k+1}\uvec{v}_{h} - \Upsilon_{\Lsig}^{k+1}\uvec{v}_{h}+\krond\Big[\bo{v}_{\cK\sigma}+\bo{v}^{\bub}_{\Ksig}-\Pi_{\sigma}^{\ell}\Upsilon_{\Ksig}^{k+1}\uvec{v}_{h}\Big].\label{eq:jump.n}
\end{align}
The normal and tangential jump of this operator are denoted by $\Jn{\uS{v}}=\Jsig{\uS{v}}\cdot \normal^+$ and $\jump{\uS{v}}_{\sigma,\tang}=\Jsig{\uS{v}}-\Jn{\uS{v}}\normal^{+}$, respectively.

These local gradients, displacement reconstructions, and jumps are patched together to create 
their global piecewise polynomial counterparts
\begin{equation*}
\GS{h} \colon \ddgskS \rightarrow \Poly{k}(\cM)^{3\times 3}, \quad
\UpS{h} \colon \ddgskS \rightarrow \Poly{k+1}(\cM)^3,
\quad \djumph{\cdot} \colon \ddgskS \rightarrow \Poly{k+1}(\cF_\Gamma)^3
\end{equation*}
by setting, for all $\uS{v}\in\ddgskS$,
 \begin{alignat*}{2}
\big(\GS{h}\uS{v}\big)_{|_K} & = \GS{K}\uS{v}
&&\qquad \forall\,K\in\mathcal{M}, \\
\big(\UpS{h}\uS{v}\big)_{|_K} &= \UpS{K}\uS{v}
&&\qquad \forall\,K\in\mathcal{M},\\
\big(\djumph{\uS{v}}\big)_{|_\sigma} 
& = \Jsig{\uS{v}}
&&\qquad \forall\,\sigma\in\mathcal{F}_\Gamma.
\end{alignat*}
Moreover, the discrete symmetric gradient $\bbeps_{h}$, divergence $\divScal_{h}$, and stress tensor $\bbsigma_{h}$ are deduced from the previous operators: with $\tr$ the matrix trace operator,
\begin{equation}
\bbeps_{h}:=\frac{1}{2}\left(\GS{h} + (\GS{h})^{T}\right),\quad\divScal_{h}:= \tr(\bbeps_{h}),\quad
\bbsigma_{h}(\cdot):= \frac{E}{1+\nu}\left(\bbeps_{h}(\cdot)+\frac{\nu}{1-2\nu}\divScal_{h}(\cdot)\Id \right).\label{eq:def.sigma.h}
\end{equation}

\subsection{Relation to the standard DDR framework}\label{subsec:rela.stand.ddr}
If $\ell\neq k$, then $\krond=0$ and no bubble component is present. The space $\dgskS$ then coincides with the vector version of the standard discrete gradient space $\dgsk$ of \cite[Section~3.1]{ddr-variant}. When $\ell=k$, bubble DOFs are active, and $\dgskS$ is a strict enrichment of $\dgsk$; see the decomposition in \eqref{eq:global.decomp} (note that, in both cases, $\uvec{v}_h \in \dgsk$). 

Regarding the reconstruction operators, when $\ell\neq k$ or $\sigma \notin \cF_{\Gamma, K}^{+}$, the bubble correction is absent, i.e., $\krond=0$ or the face does not lie on the positive side of $\Gamma$, so that $\UpS{\Ksig}\uS{v}=\Upsilon_{\Ksig}^{k+1}\uvec{v}_{h}$; see \eqref{eq:UpS.face_1}. Consequently, if $K$ is such that $\faces_{\Gamma,K}^+=\emptyset$ the local enriched cell operators reduce to their standard DDR counterparts: $\GS{K}\uS{v}=\mathbb{G}^k_K\uvec{v}_h$ and $\UpS{K}\uS{v}=\Upsilon^{k+1}_K\uvec{v}_h$ for all $\uS{v}\in\ddgskS$, as introduced in~\cite[Section~3.3.1]{ddr-variant}. When $\ell=k$ and 
$\faces_{\Gamma,K}^+\neq\emptyset$ (i.e., bubble DOFs are active on $K$), the operators $\GS{K}$ and $\UpS{K}$ are enriched versions of $\mathbb{G}^{k}_{K}$ and $\Upsilon^{k+1}_{K}$, respectively; the enrichment is reflected in the additional correction term added to $\Upsilon_{\Ksig}^{k+1}$ yielding $\UpS{\Ksig}$; see \eqref{eq:UpS.face}.

\subsection{Augmented Lagrangian DDR scheme}\label{Dis_bfm}
Throughout the analysis, we assume that $\gamma_{h} =\gamma_{0}h_{\Gamma},$  where $\gamma_{0}\in \mathbb{R}^{+}.$ We now introduce the numerical scheme for the augmented Lagrangian DDR formulation of the contact mechanics problem \eqref{Form1_nonlinearform}: Find $(\uS{u},\laf_{h})\in \ddgskS\times \Ms$ such that, for all $(\uS{v},\bos{\mu}_h)\in \ddgskS\times \Ms$, 
\begin{equation}\label{eq:scheme}
\mathbb{A}_{h}((\uS{u},\laf_{h}); (\uS{v},\bos{\mu}_h))=\sum_{K\in\cM}\int_K {\bf f}\cdot\bo{v}_K,
\end{equation}
where the bilinear form $\mathbb{A}_{h}:\left(\ddgskS\times \Ms\right) \times \left(\ddgskS\times \Ms\right) \to \mathbb{R}$ is defined as
\begin{equation}\label{eq:def.bilnear.form}
\mathbb{A}_{h}((\uS{u},\laf_{h}); (\uS{v},\bos{\mu}_h)):=\int_{\Omega}\bbsigma_h(\uS{u})\colon\bbeps_h(\uS{v})+  \stabc{\uS{u}}{\uS{v}} + \mathbb{C}_{h}((\uS{u},\laf_{h}); (\uS{v},\bos{\mu}_h)).
\end{equation}
The non-linear contact operator $\mathbb{C}_{h}$ 
is defined as
\begin{equation}\label{form1_ch}
\begin{aligned}
\mathbb{C}_{h}((\uS{u},\laf_{h}); (\uS{v},\bos{\mu}_h)):={}& 
\int_{\Gamma}\gamma^{-1}_h\bPg{\djumph{\uS{u}}}{\laf_h}\cdot\Big[ \djumph{\uS{v}}-\gamma_{h} \bos{\mu}_{h}\Big] + \int_{\Gamma} \gamma_{h} \laf_{h}\cdot \bos{\mu}_{h} \\
&+ \mathbb{S}_{\Gamma}(\laf_{h},\bos{\mu}_{h}).
\end{aligned}
\end{equation}
The stabilisation term $\mathbb{S}_{\Gamma}$ is defined by penalizing the jumps of the multiplier variable across the interior edges of the contact faces, following the approach introduced in \cite{burman2010fictitious,burman2010interior}:
\begin{equation}\label{def.S.Gamma}
\mathbb{S}_{\Gamma}(\laf_{h},\bos{\mu}_{h}):= \sum_{e\in \edges_{\Gamma}^{\text{int}}}\delta \gamma_{h} \int_{e}h_{e} \jump{{\laf}_{h}}_{e}\jump{\bos{\mu}_{h}}_{e},
\end{equation}
where $\delta >0$ is a stabilisation parameter, $\jump{\mathbf{x}}_{e}$ denotes the jump of the quantity $\mathbf{x}$ across the edge $e$ (the orientation of the jump is irrelevant) and $\edges_{\Gamma}^{\text{int}}$ is the set of interior edges associated with the faces in  $\faces_{\Gamma}.$ The stabilisation term $\mathbb{S}_{\Gamma}$ is introduced to ensure the optimal convergence rate of the discrete scheme \eqref{eq:scheme}.

The stabilisation bilinear form $\stabc{\cdot}{\cdot}$ is defined by
\begin{equation*}
\stabc{\uS{u}}{\uS{v}}\coloneqq \sum_{K\in\cM}\zeta(E_K,\nu_K)\mathbb S_K(\uS{u},\uS{v})
\end{equation*}
with local stabilisation bilinear form $\mathbb S_K:\ddgskS\times\ddgskS\rightarrow \mathbb R$  given by
\begin{equation}\label{eq:stabK}
\begin{aligned}
\mathbb S_{K}{}&(\uS{u},\uS{v}):=h^{-1}_K\sum_{\sigma\in\cF_K}
\int_\sigma \left(\UpS{K}\uS{u}-{\Upsilon}^{k+1}_{\Ksig}\uvec{u}_{h}\right)\cdot\left(\UpS{K}\uS{v}-{\Upsilon}^{k+1}_{\Ksig}\bo{v}_h\right)\\&+\sum_{e\in\cE_K}\int_e\left(\UpS{K}\uS{u}-\Upsilon^{k+1}_{K e}\uvec{u}_{h}\right)\cdot\left(\UpS{K}\uS{v}-\Upsilon^{k+1}_{Ke}\uvec{v}_{h}\right)\\
&+h_K\sum_{s\in\mathcal V_K}\left(\UpS{K}\uS{u}(\bx_s)-\bo{u}_\Ks\right)\cdot\left(\UpS{K}\uS{v}(\bx_s)-\bo{v}_\Ks\right)\\&
+ h^{-1}_{K}\krond\sum_{\sigma\in \mathcal F_{\Gamma,K}^+} \int_\sigma \left(\bo{u}_{\cK\sigma} + \bo{u}^{\bub}_{\Ksig} -\Pi^\ell_{\sigma}\Upsilon^{k+1}_{\Ksig}\uvec{u}_h\right)\cdot \left(\bo{v}_{\cK\sigma} + \bo{v}^{\bub}_{\Ksig} -\Pi^\ell_{\sigma}\Upsilon^{k+1}_{\Ksig}\uvec{v}_h\right).
\end{aligned}
\end{equation}
Above, for each $K\in\cells$, $E_K$ and $\nu_K$ respectively denote the average of $E$ and $\nu$ on $K$, and $\zeta$ is such that $\zeta(E_{K},\nu_{K})>\mu_{1}> 0$ for all $K\in \cells$ and scales as $\frac{E_{K}}{1-2\nu_{K}}$ (or, equivalently, as the sum of the cell-averaged Lam\'e coefficients).

Let us also introduce the unscaled stabilisation bilinear form
\begin{equation*}
		 \stabin{\uS{u}}{\uS{v}}:= \sum_{K\in\cells} S_K(\uS{u},\uS{v}). 
\end{equation*}

\begin{proposition}[Existence and uniqueness of the solution to the scheme]\label{wellposedness}
There exists a unique solution $(\uS{u},\laf_{h}) \in\ddgskS\times \Ms$ \
to \eqref{eq:scheme}. 
\end{proposition}
\begin{proof} See Section \ref{sec:proof.existence.uniqueness}.
\end{proof}

\subsection{Discrete norms}

We now define the following
$H^{1}$-like norm on $\ddgskS$ as: For all $\uS{v}\in \ddgskS$,
\begin{equation*}
\lVert \uS{v}\rVert_{1,h}\coloneqq \left(\sum_{K\in\cM}\lVert \uS{v}\rVert^2_{1,K}\right)^{\frac{1}{2}} \text{ with }\lVert\uS{v}\rVert^2_{1,K}\coloneqq \lVert\GS{K}\uS{v} \rVert^2_{L^2(K)}+\mathbb S_K(\uS{v},\uS{v}).
\end{equation*}

\begin{theorem}[Discrete Korn Inequality]\label{thm:korn}
It holds
\begin{equation}\label{eq:dis_korn}
\lVert \uS{v}\rVert_{1,h}^2\lesssim \lVert \bbeps_h(\uS{v})\rVert^2_{L^2(\O\setminus\overline{\Gamma})}+\mathbb S_h(\uS{v},\uS{v})\qquad\forall\uS{v}\in\ddgskS.
\end{equation}
\end{theorem}
\begin{proof}
In the lowest-order case~(i.e., $k=0$), the result follows by arguments
analogous to those in \cite[Theorem~7]{droniou2025analysis}. For
higher-order approximations~(i.e., $k\geq 1$), the proof proceeds along the
same lines as in \cite[Theorem~3.1]{DKMR25}, with the face unknown
$\bo{v}_{\cK\sigma}$ used therein replaced, on $\sigma\in\F^+_{\Gamma,K}$, by
$\bo{v}_{\cK\sigma}+\krond\bo{v}^\bub_{K\sigma}$.
\end{proof}

We now introduce the discrete energy inner product $\langle \cdot,\cdot\rangle_{\mathrm{en},h}$ such that, for all $\uS{u},\uS{v}\in \ddgskS$,
 \begin{equation} \label{eq:innerproduct}
 \langle \uS{v} , \uS{w} \rangle_{\mathrm{en},h}:= \int_{\Omega}\bbsigma_h(\uS{v})\colon\bbeps_h(\uS{w})+\mu_{1}\mathbb S_{h}(\uS{v},\uS{w}) \quad \forall \uS{v} , \uS{w} \in \ddgskS.
 \end{equation}
The energy norm associated to this inner-product is denoted by $\NORM{\mathrm{en},h}{\uS{v}}$.
 We note that, by  the discrete Korn inequality \eqref{eq:dis_korn} and the definition \eqref{eq:def.sigma.h} of $\bbsigma_h$,
\begin{equation}\label{eq:norm_energy}
\lVert \uS{v}\rVert_{1,h}^2 \simeq \lVert \uS{v}\rVert_{\mathrm{en},h}^2,
\end{equation}
which shows that $\langle \cdot , \cdot \rangle_{\mathrm{en},h}$ is, indeed, an inner product.
Here the hidden constant in $\simeq$ depends only on $\Omega$, $\Gamma$, on the mesh regularity parameter, and on the Young modulus $E$ and Poisson ratio $\nu.$

\section{Main results}\label{Sec:main_results}
The following two quantities, inspired by the Gradient Discretisation Method \cite{gdm}, will be instrumental to state the error estimates.

\begin{itemize}
\item The primal consistency error is: For $\uS{v}\in\ddgskS$,
\begin{equation}\label{eq:primal.consist}
\bos{\fC}_h(\bo{u},\uS{v})\coloneqq \left(\lVert \nabla \bo{u}-\GS{h}\uS{v}\rVert_{L^2(\O\setminus\overline{\Gamma})}^2+\mathbb S_h(\uS{v},\uS{v})\right)^{\frac{1}{2}}.
\end{equation}
 \item Letting
\begin{equation}\label{sigma.space}
\Sigmaspace\coloneqq\{\bbtau\in H_{\DIV}(\O\setminus\overline{\Gamma};\mathcal{S}^3(\mathbb R))\colon~ \gamma^+_{\bo{n}}\bbtau+\gamma^-_{\bo{n}}\bbtau=\bo{0},
~\gamma^+_{\bo{n}}\bbtau\in L^2(\Gamma)\},
\end{equation}
(where $\mathcal S^3(\mathbb R)$ is the space of symmetric $3\times 3$ matrices with real coefficients, and $\gamma^+_{\bo{n}}\bbtau\coloneq (\gamma^+\bbtau) \normal^{+}$), the limit-conformity measure is defined as: For $\bbtau\in\Sigmaspace$,
\begin{equation}\label{eq:adjoint.consist}
  \begin{aligned}
&  \bos{\cW}_h(\bbtau)\coloneqq{} \sup\limits_{\uvec{v}_h\in\ddgskS}\frac{w_h(\bbtau,\uS{v})}{\lVert\uS{v}\rVert_{1,h}}, \\
\text{ with }  w_h(\bbtau,\uS{v})\coloneqq{}& -\int_\O \bbtau\colon \bbeps_h(\uS{v})+\int_\Gamma \gamma^+_{\bo{n}}\bbtau \cdot \djumph{\uS{v}}-\sum_{K\in\cM}\int_{K}\bo{v}_K\cdot\DIV\bbtau,
  \end{aligned}
\end{equation}
where, here and in the following, suprema of quantities involving a division by a norm are implicitly taken by removing the zero vector.
\end{itemize}

\begin{theorem}[Abstract error bounds]\label{thm:abs.err.bound}
 Let $(\bo{u},\laf)$ be the solution of \eqref{Form1_nonlinearform} and assume that $\bu\in \bo{U}_0 \cap \boldsymbol{C}_0^0(\Omega\backslash\overline{\Gamma}).$ 
 Then, the discrete solution $(\uS{u},\laf_{h})$ of \eqref{eq:scheme} satisfies the following error estimate: recalling that $\gamma_{h} =\gamma_{0}h_{\Gamma}$, for sufficiently small $\gamma_{0}\in \mathbb{R}^{+}$ and sufficiently large $\delta\in \mathbb{R}^{+},$  
\begin{multline*}
\NORM{\mathrm{en},h}{\IhS\bu-\uS{u}} 
+\NORM{L^2(\Gamma)}{
     \gamma^{-\frac{1}{2}}_{h}\bPg{\djumph{\uS{u}}}{\laf_h}-\gamma^{\frac{1}{2}}_{h}\laf
   } +\NORM{L^2(\Gamma)}{\gamma^{\frac{1}{2}}_{h}\big(\Pi_{\cF_{\Gamma}}^{\ell}\laf-\laf_h\big)} \\ + \mathbb{S}_{\Gamma}(\Pi_{\cF_{\Gamma}}^{\ell}\laf-\laf_h,\Pi_{\cF_{\Gamma}}^{\ell}\laf-\laf_h)^{\frac{1}{2}}
   \lesssim \bos{\fC}_h(\bu,\IhS\bu)+ \bos{\cW}_h(\bbsigma(\bu)) 
    +\NORM{L^2(\Gamma)}{\gamma^{-\frac{1}{2}}_{h}\big(\jump{\bo{u}}
    -\djumph{\IhS\bo{u}}\big)}
  \\  +\NORM{L^2(\Gamma)}{\gamma^{\frac{1}{2}}_{h}\big(\laf-\Pi_{\cF_{\Gamma}}^{\ell}\laf\big)} 
     +\left(\sum_{K\in \cells}\sum_{\sigma\in\cF_{\Gamma,K}^+}\NORM{L^{2}(\sigma)}{\gamma^{\frac{1}{2}}_{h}\big(\bbsigma(\bo{u})-\Pi_{K}^{k}\bbsigma(\bo{u})\big)}^2\right)^{\frac{1}{2}}+\mathbb{S}_{\Gamma}(\Pi_{\cF_{\Gamma}}^{\ell}\laf,\Pi_{\cF_{\Gamma}}^{\ell}\laf)^{\frac{1}{2}}.
\end{multline*}
\end{theorem}
\begin{proof}
   See Section~\ref{sec:proof_error}. 
\end{proof}

In the following theorem, $H^m(\cM)$ denotes the standard space of broken $H^m$ functions $\cM$, endowed with its usual seminorm. As in the rest of the paper, the boldface $\boldsymbol{H}^m$ indicates the space of vector-valued functions whose components belong to $H^m$.

\begin{theorem}\label{thm:error.estimate}
 Let $(\bo{u},\laf)$ be the solution of \eqref{Form1_nonlinearform}. Assume that $\bu\in \bo{U}_0 \cap \boldsymbol{H}^{k+2}(\cM)$ and $\laf \in H^{\frac{1}{2}+\epsilon}(\Gamma)\cap H^{\frakr-\frac12}(\faces_\Gamma)$ for some $\epsilon>0$, and $\frakr$ as below.
 Then, the discrete solution $(\uS{u},\laf_{h})$ of \eqref{eq:scheme} satisfies the following error estimate: recalling that $\gamma_{h} =\gamma_{0}h_{\Gamma}$, for sufficiently small $\gamma_{0}\in \mathbb{R}^{+}$ and sufficiently large $\delta\in \mathbb{R}^{+},$
\begin{multline*}
    \NORM{\mathrm{en},h}{\IhS\bu-\uS{u}} 
   + \NORM{L^2(\Gamma)}{
     \gamma^{-\frac{1}{2}}_{h}\bPg{\djumph{\uS{u}}}{\laf_h}-\gamma^{\frac{1}{2}}_{h}\laf
   } +\NORM{L^2(\Gamma)}{\gamma^{\frac{1}{2}}_{h}\big(\Pi_{\cF_{\Gamma}}^{\ell}\laf-\laf_h\big)} \\ +  \mathbb{S}_{\Gamma}(\Pi_{\cF_{\Gamma}}^{\ell}\laf-\laf_h,\Pi_{\cF_{\Gamma}}^{\ell}\laf-\laf_h)^{\frac{1}{2}} \lesssim h^{\frakr}\left(|\bbsigma(\bu)|_{H^{k+1}(\cM)}+ |\bu|_{H^{k+2}(\cM)}+ |\laf|_{H^{\frakr-\frac12}(\faces_\Gamma)}\right),
\end{multline*}
where $\frakr$ depends on the choices of $\ell$:
\begin{itemize}
\item if $\ell = k$ (active bubble): $\frakr = k+1$.
\item if $\ell = k - 1$ (without bubble): $\frakr = k+\frac{1}{2}$.
\end{itemize}
\end{theorem}
\begin{proof}
   See Section~\ref{subsec:error.bound}. 
\end{proof}

\begin{theorem}[Suboptimal error estimate without $\mathbb{S}_{\Gamma}$ and the projected jump reconstruction~\eqref{eq:jump.modify}]\label{thm:error.estimate.sub}
 Let $(\bo{u},\laf)$ be the solution of \eqref{Form1_nonlinearform} . Assume that $\bu\in \bo{U}_0 \cap \boldsymbol{H}^{k+2}(\cM)$ and $\laf \in H^{\frac{1}{2}+\epsilon}(\Gamma)\cap H^{\ell}(\faces_\Gamma).$
 Then, the discrete solution $(\uS{u},\laf_{h})$ of the scheme \eqref{eq:scheme} modified without stabilizer~$\mathbb{S}_{\Gamma}$ and using the projected jump reconstruction operator~\eqref{eq:jump.modify}, satisfies the following error estimate: recalling that $\gamma_{h} =\gamma_{0}h_{\Gamma}$, for sufficiently small $\gamma_{0}\in \mathbb{R}^{+}$ and sufficiently large $\delta\in \mathbb{R}^{+},$
\begin{multline*}
    \NORM{\mathrm{en},h}{\IhS\bu-\uS{u}} 
   + \NORM{L^2(\Gamma)}{
     \gamma^{-\frac{1}{2}}_{h}\bPg{\djumph{\uS{u}}}{\laf_h}-\gamma^{\frac{1}{2}}_h\laf
   } +\NORM{L^2(\Gamma)}{\gamma^{\frac{1}{2}}_h\big(\Pi_{\cF_{\Gamma}}^{\ell}\laf-\laf_h\big)} \\  \lesssim h^{\ell+\frac{1}{2}}\left(|\bbsigma(\bu)|_{H^{k+1}(\cM)}+ |\bu|_{H^{k+2}(\cM)}+ |\laf|_{H^{\ell}(\faces_\Gamma)} +\SEMINORM{H^{\ell+1}(\faces_\Gamma)}{\jump{\bu}}\right).
\end{multline*}
\end{theorem}
\begin{proof}
See Section~\ref{sec:nostab}. 
\end{proof}

\begin{remark}[Locking-free scheme]
Thanks to the bubble enrichment for $k=0$ or to the already existing face displacement unknown for $k\ge 1$, following the analysis in \cite{droniou2026locking,DKMR25} it can be seen that the hidden constants in the error estimates of Theorems \ref{thm:abs.err.bound}, \ref{thm:error.estimate} and \ref{thm:error.estimate.sub} do not depend on the Poisson ratio $\nu$.
As a consequence, the schemes presented here locking-free as $\nu\to\frac12$.
\end{remark}

\section{Numerical experiments}\label{sec:numerical.exp.}
In this section, we assess the accuracy and robustness of the augmented Lagrangian DDR scheme~\eqref{eq:scheme} for contact mechanics with Tresca friction. The first and second Lamé
coefficients are related to Young's modulus $E$ and Poisson's ratio $\nu$ by $L := \frac{\nu E}{(1+\nu)(1-2\nu)}$ and $G := \frac{E}{2(1+\nu)}$, respectively. Section~\ref{exm:tresca} considers a manufactured solution on the tetrahedral and hexa-cut mesh families of \cite{DKMR25}, for $k \in \{0,1,2\}$ and both admissible choices of $\ell$. It compares the scheme of Theorem~\ref{thm:error.estimate} (jump~\eqref{eq:jump.n} with $\mathbb{S}_\Gamma$) with the stabiliser-free, projected-jump~\eqref{eq:jump.modify} variant of Theorem~\ref{thm:error.estimate.sub}. Section~\ref{exm:intersecting.fracture} addresses a geometrically realistic discrete
fracture--matrix configuration with intersecting fractures, for which no exact solution is available. There, the schemes are compared through the mean normal jump and the mean normal
multiplier on each fracture, and through the number of Newton iterations for several values of $\gamma_0$. The choices of $\gamma_0$ and $\delta$ are discussed in Remark~\ref{rem:gamma0.delta}.

\subsection{Tresca friction model with manufactured solution}\label{exm:tresca}
Following \cite{DKMR25}, we consider the 3D domain $\Omega = (-1,1)^3$ with a single non-immersed fracture $\Gamma = \{0\} \times (-1,1)^2$.  The friction coefficient $g$ is set to $1$, corresponding to frictionless contact,  and the Lam\'e coefficients are set to be $G = L = 1$. The exact solution is defined by
\begin{align*}
\bo{u}(x,y,z) =
\begin{cases}
\begin{pmatrix}
h(x,y)P(z) - g y \\
P(z) \\
x^2 P(z)
\end{pmatrix} & \text{if } z \ge 0, \\[1.2em]
\begin{pmatrix}
h(x,y)Q(z) - g y \\
2Q(z) \\
x^2 Q(z)
\end{pmatrix} & \text{if } z < 0,~ x < 0, \\[1.2em]
\begin{pmatrix}
h(x,y)Q(z) - g y \\
Q(z) \\
x^2 Q(z)
\end{pmatrix} & \text{if } z < 0,~ x \ge 0,
\end{cases}
\end{align*}
with $h(x,y) = -\sin(x)\cos(y)$, $P(z) = z^2$, and $Q(z) = z^2/4$, 
and satisfies the Tresca frictional-contact conditions at the matrix–fracture interface $\Gamma$. The right-hand side $\mathbf{f} = -\operatorname{div}\boldsymbol{\sigma}(\mathbf{u})$ and the Dirichlet boundary conditions on $\partial\Omega$ are obtained from $\bo{u}$. The fracture $\Gamma$ is in slippy-contact for $z<0$ and sticky-contact for $z>0.$ The convergence of the scheme is assessed on the same families of uniform tetrahedral and hexa-cut meshes as in~\cite{DKMR25}~(the latter meshes are genuine polyhedral meshes, with each non-planar face of a random hexahedral mesh cut in two); the polynomial degrees are $k=0,1,2$ and, for each $k\ge 1$, both admissible choices $\ell=k$ (with bubble) and $\ell=k-1$ (without bubble). Figures~\ref{fig:tresca.tet} and \ref{fig:tresca.hex} report the relative energy-norm error  and the relative $\gamma_h^{1/2}$-weighted $L^2(\Gamma)$ multiplier error against $N_{\mathrm{cell}}^{1/3}$, together with together with reference slopes. Panels (a) -- (b) of each figure use the scheme to validate Theorem~\ref{thm:error.estimate}, i.e., the jump reconstruction~\eqref{eq:jump.n} with stabiliser $\mathbb{S}_\Gamma$, for which the theorem predicts an energy-norm rate of $h^{k+1}$ when $\ell=k$ and $h^{k+1/2}$ when $\ell=k-1$. Panels (c) -- (d) use the stabiliser-free (no $\mathbb{S}_{\Gamma}$), projected-jump variant~\eqref{eq:jump.modify} of Theorem~\ref{thm:error.estimate.sub}, for which the (suboptimal) predicted rate is $h^{\ell+1/2}$ irrespective of whether $\ell=k$ or $\ell=k-1$.

On tetrahedral meshes, the scheme of Theorem~\ref{thm:error.estimate} reproduces the predicted
rates closely once $k\ge 1$: for $\ell=k$, the energy-norm rate settles at
$2.0$--$2.1$ for $k=1$ and exactly $3.0$ for $k=2$, against the predicted
$k+1=2$ and $3$; for $\ell=k-1$, it settles at $2.6$ for $k=2$ against the
predicted $k+1/2=2.5$, while the $k=1$ case ($1.1$--$1.4$ against a
predicted $1.5$) is still approaching its asymptotic rate over the range of
meshes tested. The $k=0$ case (only $\ell=0$ is admissible) similarly rises
from $1.3$ to $1.5$ against a predicted rate of $1$, i.e.\ mildly
super-convergent rather than under-converging. The multiplier error tracks
the same predictions, generally matching or mildly exceeding them (e.g.\
$2.6$--$2.7$ for $k=1$, $\ell=k$, against a predicted $2$).

The stabiliser-free (no $S_{\Gamma}$) variant corresponding to Theorem~\ref{thm:error.estimate.sub} (panels (c)--(d)) is, on both mesh families, markedly \emph{more} accurate in practice than its own suboptimal bound: for $k=0$ the observed energy-norm rate is $1.7$ against a predicted $0.5$, and for $\ell=k-1,\,k=1$ it is $2.0$--$2.1$ against a predicted $0.5$ -- more than a full order better in both cases. The multiplier error shows the same pattern (e.g.\ $2.4$--$2.6$ against a predicted $0.5$ for $k=1$, $\ell=k-1$). This is the analogous
observation already made for the Tresca case in \cite[Section~6.2]{DKMR25}, which is related to higher-order mixed scheme.

The hexa-cut family reproduces the same qualitative trends, with one caveat worth reporting: for $k=0$ and for the $\ell=k-1$ cases, the observed multiplier rate is non-monotonic at the coarsest refinement levels, e.g.,$-0.1$ then $0.6,1.0,1.6$ for $k=\ell=0$ with the
stabilised jump, and $-0.3$ then $0.6,1.0,1.6$ for $k=1,\ell=0,$ before settling into a regime consistent with the tetrahedral results.



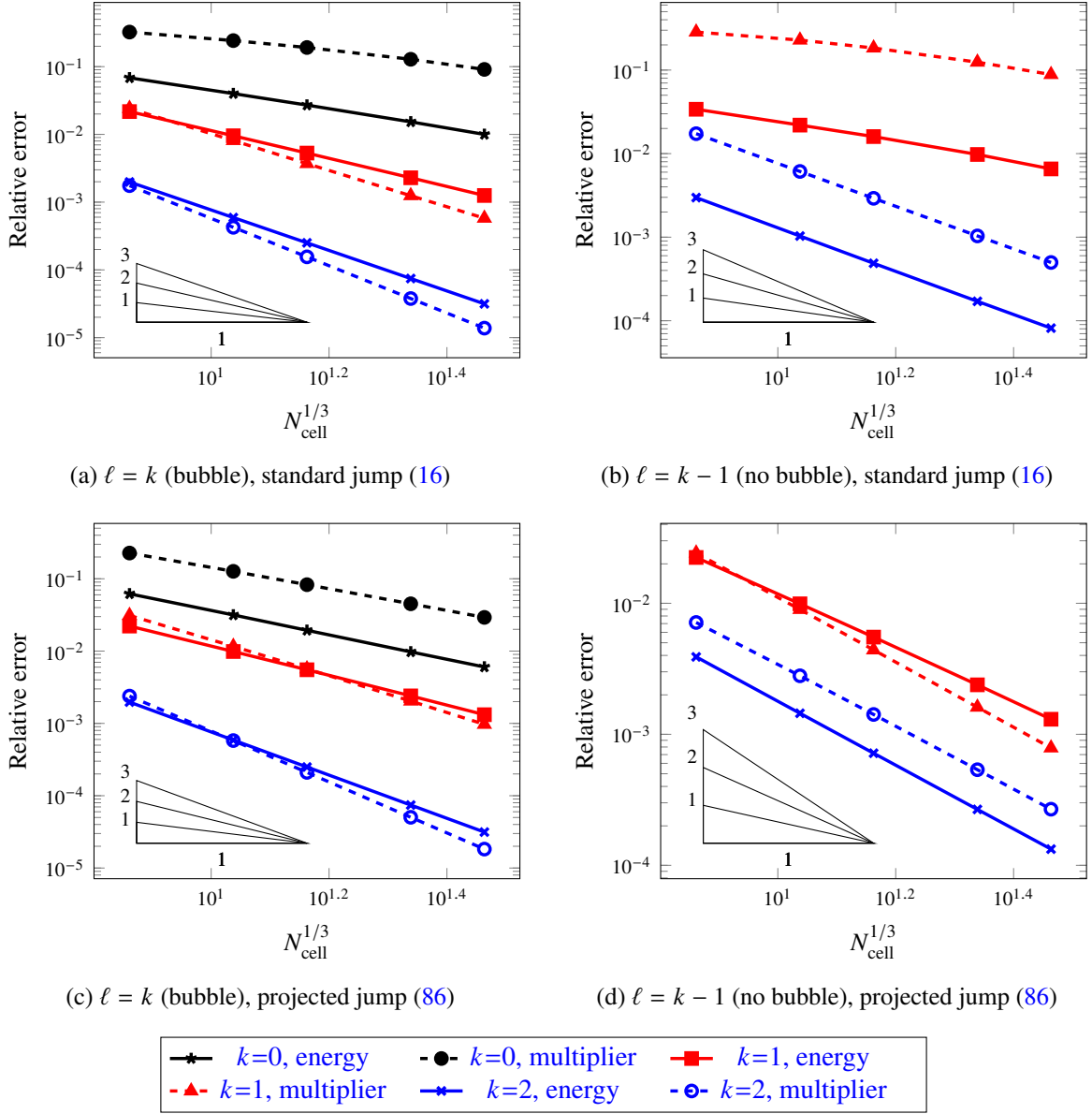
\begin{figure}[H]\centering
  \begin{subfigure}[b]{0.48\textwidth}
    \centering
    \begin{tikzpicture}\panelBubbleLeg{tet_std}{leg.tet}
    \end{tikzpicture}
    \subcaption{$\ell=k$ (bubble), standard jump~\eqref{eq:jump.n}}
    \label{fig:tet:a}
  \end{subfigure}\hfill
  \begin{subfigure}[b]{0.48\textwidth}
    \centering
    \begin{tikzpicture}\panelNoBubble{tet_std}\end{tikzpicture}
    \subcaption{$\ell=k-1$ (no bubble), standard jump~\eqref{eq:jump.n}}
    \label{fig:tet:b}
  \end{subfigure}\\[0.4cm]
  \begin{subfigure}[b]{0.48\textwidth}
    \centering
    \begin{tikzpicture}\panelBubbleLeg{tet_proj}{leg2.tet}\end{tikzpicture}
    \subcaption{$\ell=k$ (bubble), projected jump~\eqref{eq:jump.modify}}
    \label{fig:tet:c}
  \end{subfigure}\hfill
  \begin{subfigure}[b]{0.48\textwidth}
    \centering
    \begin{tikzpicture}\panelNoBubble{tet_proj}\end{tikzpicture}
    \subcaption{$\ell=k-1$ (no bubble), projected jump~\eqref{eq:jump.modify}}
    \label{fig:tet:d}
  \end{subfigure}\\[0.3cm]
  \ref{leg.tet}
  \caption{Tetrahedral mesh, Tresca friction test case Section~\ref{exm:tresca}: relative energy-norm error (solid curve) and relative $\gamma_h^{1/2}$-weighted $L^2$ multiplier error (dashed curve) versus $N_{\mathrm{cell}}^{1/3}$, for $k=0,1,2$. Panels (a) -- (b) use the scheme of Theorem~\ref{thm:error.estimate} (jump~\eqref{eq:jump.n},
  stabiliser $\mathbb{S}_\Gamma$); panels (c) -- (d) use the stabiliser-free $\mathbb{S}_\Gamma$, projected-jump variant of Theorem~\ref{thm:error.estimate.sub}. For $k=0$, only $\ell=0$ is valid, so panels (b) and (d) show $k=1,2$ only.}
  \label{fig:tresca.tet}
\end{figure}

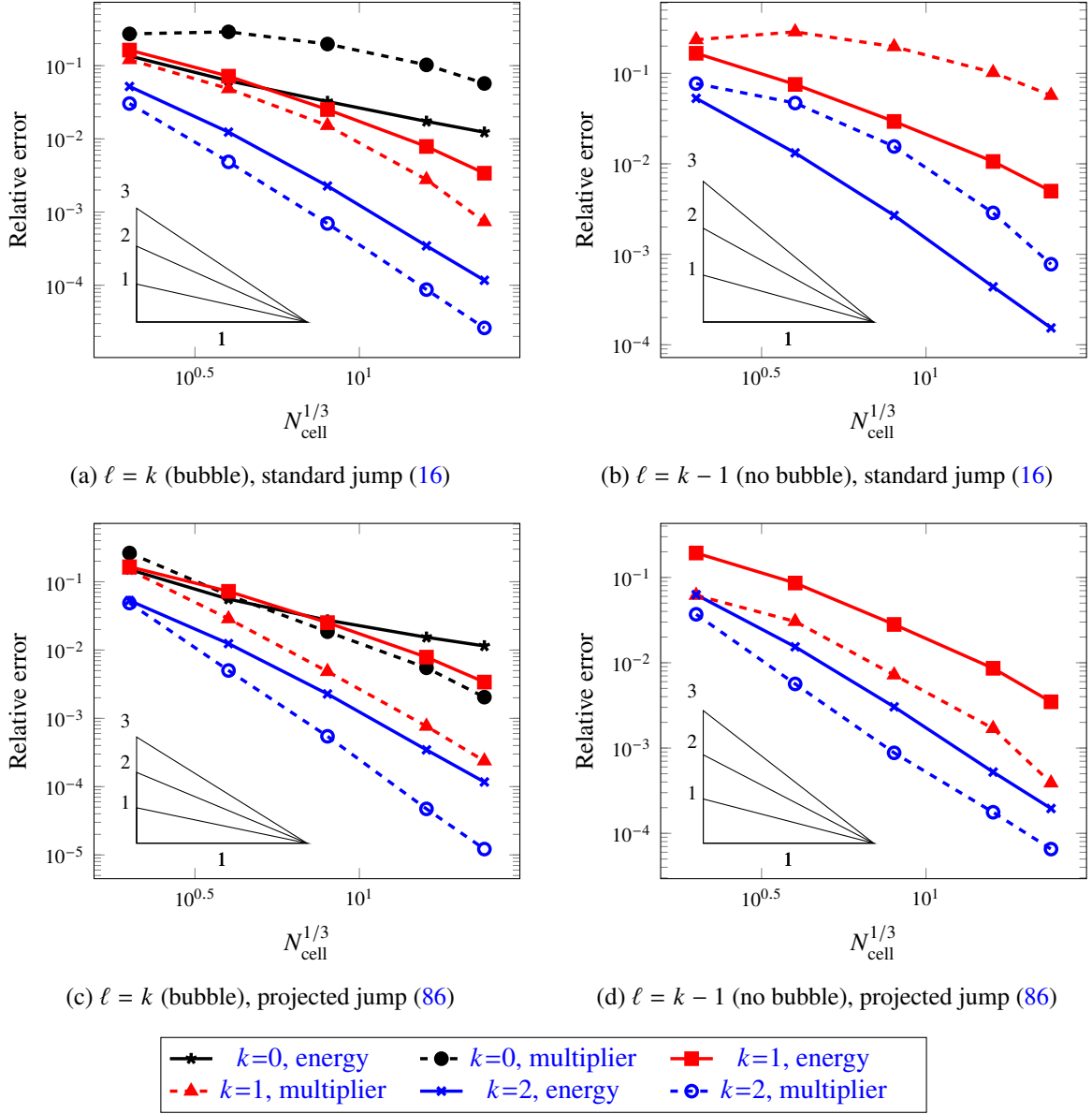
\begin{figure}[H]\centering
  \begin{subfigure}[b]{0.48\textwidth}
    \centering
    \begin{tikzpicture}\panelBubbleLeg{hex_std}{leg.hex}\end{tikzpicture}
    \subcaption{$\ell=k$ (bubble), standard jump~\eqref{eq:jump.n}}
    \label{fig:hex:a}
  \end{subfigure}\hfill
  \begin{subfigure}[b]{0.48\textwidth}
    \centering
    \begin{tikzpicture}\panelNoBubble{hex_std}\end{tikzpicture}
    \subcaption{$\ell=k-1$ (no bubble), standard jump~\eqref{eq:jump.n}}
    \label{fig:hex:b}
  \end{subfigure}\\[0.4cm]
  \begin{subfigure}[b]{0.48\textwidth}
    \centering
    \begin{tikzpicture}\panelBubbleLeg{hex_proj}{leg2.hex}\end{tikzpicture}
    \subcaption{$\ell=k$ (bubble), projected jump~\eqref{eq:jump.modify}}
    \label{fig:hex:c}
  \end{subfigure}\hfill
  \begin{subfigure}[b]{0.48\textwidth}
    \centering
    \begin{tikzpicture}\panelNoBubble{hex_proj}\end{tikzpicture}
    \subcaption{$\ell=k-1$ (no bubble), projected jump~\eqref{eq:jump.modify}}
    \label{fig:hex:d}
  \end{subfigure}\\[0.3cm]
  \ref{leg.hex}
  \caption{Hexa-cut mesh,  Tresca friction test case Section~\ref{exm:tresca}: relative energy-norm error (solid curve) and relative $\gamma_h^{1/2}$-weighted $L^2$ multiplier error (dashed curve) versus $N_{\mathrm{cell}}^{1/3}$, for $k=0,1,2$. Panels (a) -- (b) use the scheme of Theorem~\ref{thm:error.estimate} (jump~\eqref{eq:jump.n},
  stabiliser $\mathbb{S}_\Gamma$); panels (c) -- (d) use the stabiliser-free $\mathbb{S}_\Gamma$, projected-jump variant of Theorem~\ref{thm:error.estimate.sub} (jump~). $k=0$ has only
  $\ell=0$, so panels (b) and (d) show $k=1,2$ only .}
  \label{fig:tresca.hex}
\end{figure}

\subsection{Discrete fracture matrix~(DFM) model with intersecting fractures}\label{exm:intersecting.fracture}
We consider the domain $\Omega=(0,1)^3$ containing the fracture network shown in Figure~\ref{fig:mesh_6k}. The tetrahedral meshes, which conform to the fracture network, are generated with \texttt{Gmsh}~\cite{gmsh}. The material is assumed to be homogeneous and isotropic, with Young's modulus $E=4\times 10^{9}$ and Poisson's ratio $\nu=0.2$, which gives a shear modulus $L=E/(2(1+\nu))=\frac{5}{3}\times 10^{9}$. Frictional contact along the fractures is described by the Tresca model with a constant friction threshold $g=4\times 10^{6}$. Dirichlet conditions are imposed on the bottom and top faces, namely $\bo{u}=\bo{0}$ at $z=0$ and $\bo{u}=(0.0015,-0.0015,-0.002)$ at $z=1$, while homogeneous Neumann conditions are prescribed on the lateral faces. 

Figures~\ref{fig:F1.normal.comp.Lag.Multi}-\ref{fig:lambda.fractures234.multi}  and Table~\ref{tab:dfm.newton.comparison} together give a qualitative and quantitative assessment of the scheme
on this more demanding, geometrically realistic setting, in the same spirit as the study carried out in~\cite{droniou2023bubble,DKMR25} for intersecting fractures.

Figure~\ref{fig:F1.normal.comp.Lag.Multi} plots the normal component of the discrete Lagrange multiplier on fracture~F1 for three choices of $(k,\ell),$ using the projected jump~\eqref{eq:jump.modify} on a tetrahedral mesh with 123k cells and 8.8k fracture faces. The lowest-order case $(k,\ell)=(0,0)$ resolves the transition
between the open and sticking regions only coarsely, producing the visibly faceted, piecewise-constant profile of panel~(a); increasing the polynomial degree to $(k,\ell)=(1,0)$ already smooths this transition considerably (panel~(b)), and $(k,\ell)=(1,1)$ sharpens the resolution of the
contact boundary further still (panel~(c)), without introducing any spurious oscillations. This qualitative gain in smoothness is consistent with the higher-order scheme results reported in~\cite{DKMR25}. Figure~\ref{fig:contact_normal} shows, for $(k,\ell)=(1,0)$ and with the projected jump~\eqref{eq:jump.modify}, the classification of every fracture face into one of the four contact states~(see Figure~\ref{fig:contact_normal}(a)) together with the corresponding normal displacement jump $\djumphn{\uS{u}}$~(see Figure~\ref{fig:contact_normal}(b)). The results are similar for $(k,\ell)=(1,1)$.

Figure~\ref{fig:F1.normal.comp.Lag.Multi.smooth} illustrates the smoothing effect due to the standard contact stabiliser on the normal Lagrange multiplier F1, for three different standard-jump penalty parameters $\delta$. With $\delta=0,$ the multiplier has pronounced local oscillations. Increasing the parameter $\delta$ to $0.5$ visibly reduces these oscillations, with a maximum of $7.3\times10^5$, and produces a smoother field. At $\delta=1,$ the field remains smooth, with smaller additional changes, with a maximum of $7\times10^5$. This illustrates the smoothing effect of the contact stabiliser. But for $\delta=10$, the maximum of the multiplier decreases from $7.5\times 10^5$ at~($\delta=0$) to $5.4\times 10^5,$ indicating that stronger stabilisation smooths the field but also reduces its peak.

As the mesh is refined, the mean normal jump and mean multiplier on~F1 are expected to converge to their (unknown) exact values; Figure~\ref{fig:normal.fracture1.multi} shows that all four discretisations $(k,\ell) \in \{(0,0),(1,0),(1,1),(2,1)\}$ with projected jump~\eqref{eq:jump.modify} converge to the same limiting values, but at very different rates in terms of degrees of freedom. On the coarsest meshes, the $(k,\ell)=(1,0)$ results are sometimes closer to the values approached under refinement than the $(k,\ell)=(1,1)$ results. Thus, these plots do not establish a uniform advantage for $(k,\ell)=(1,1)$. As the mesh is refined, the higher-order schemes approach similar mean values, whereas the $(k,\ell)=(0,0)$ results change more substantially over the meshes considered. The normal displacement jump and normal Lagrange multiplier are shown in Figures~\ref{fig:normal.fracture1.multi}(a) and~\ref{fig:normal.fracture1.multi}(b), respectively. Figure~\ref{fig:lambda.fractures234.multi} repeats this comparison for the mean of the normal Lagrange multiplier on fractures F2, F3, and F4. The same pattern recurs: all four schemes are converging toward common limiting values, with $(k,\ell)=(0,0)$ approaching them only slowly, while the higher-order schemes are already within the limit on the coarsest meshes shown. For the projected jump scheme with $\ell=0$, the converged discrete solution is independent of $\gamma_0$; the choice of $\gamma_0$ affects the convergence of the semismooth Newton method~(see Remark~\ref{remark:equivalencemixed}).

\begin{figure}[H]
\centering
\includegraphics[width=0.40\textwidth]{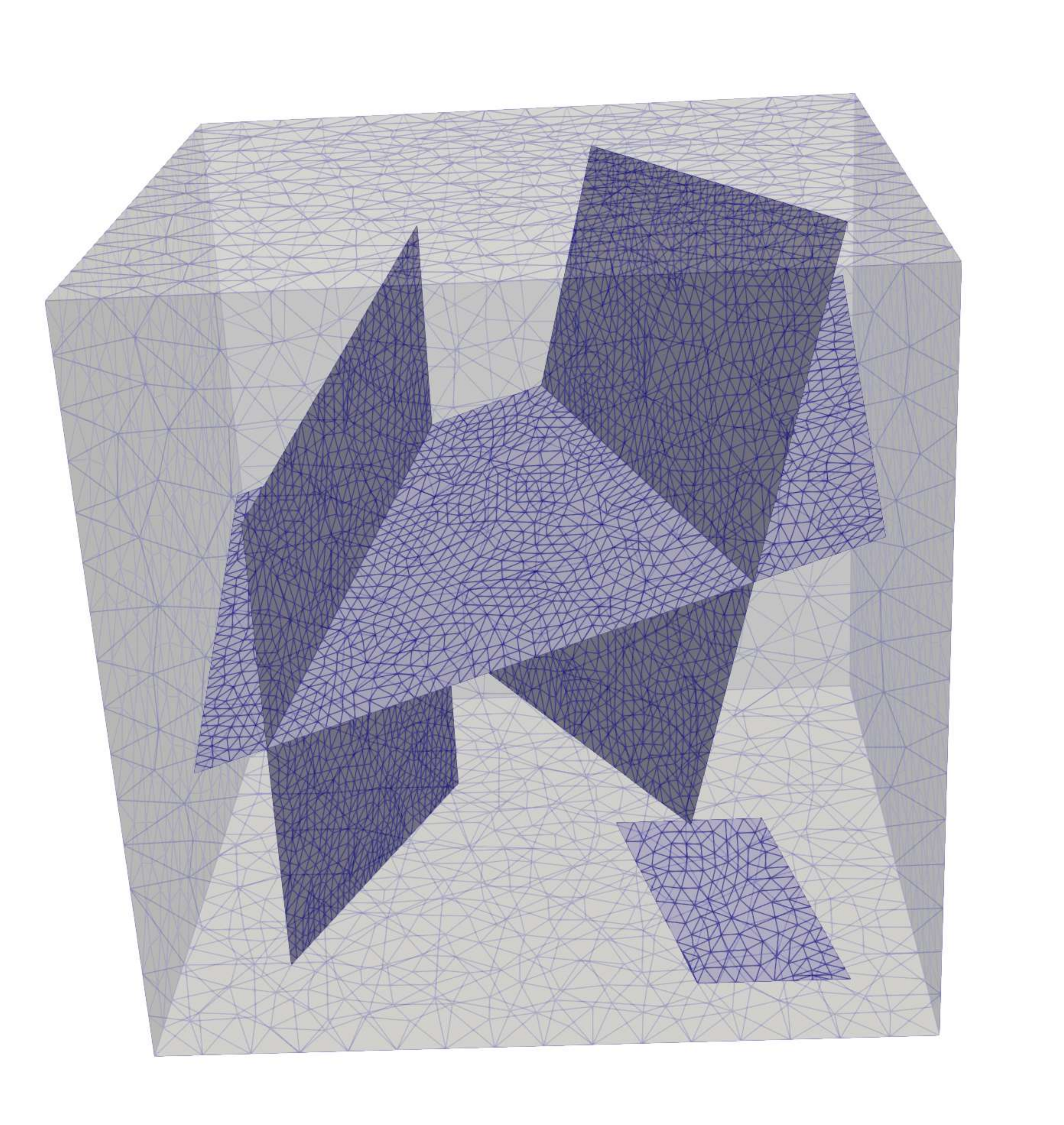}
\caption{Fracture network for the test case in Section~\ref{exm:intersecting.fracture}}
\label{fig:mesh_6k}  
\end{figure}
\begin{figure}[H]
\begin{subfigure}{0.33\textwidth}
  \includegraphics[width=\textwidth]{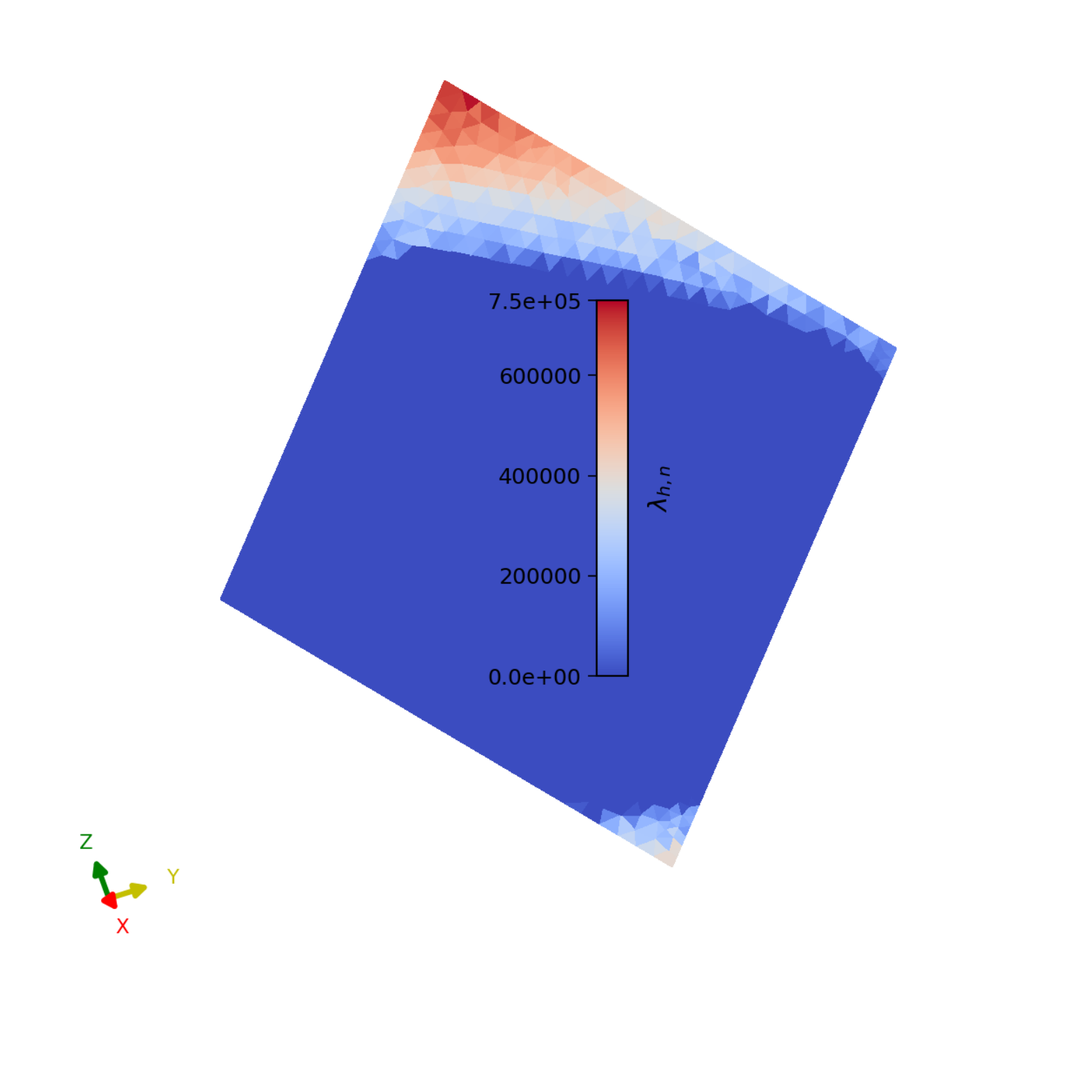}
  \caption{}
\end{subfigure}
\begin{subfigure}{0.33\textwidth}
  \includegraphics[width=\textwidth]{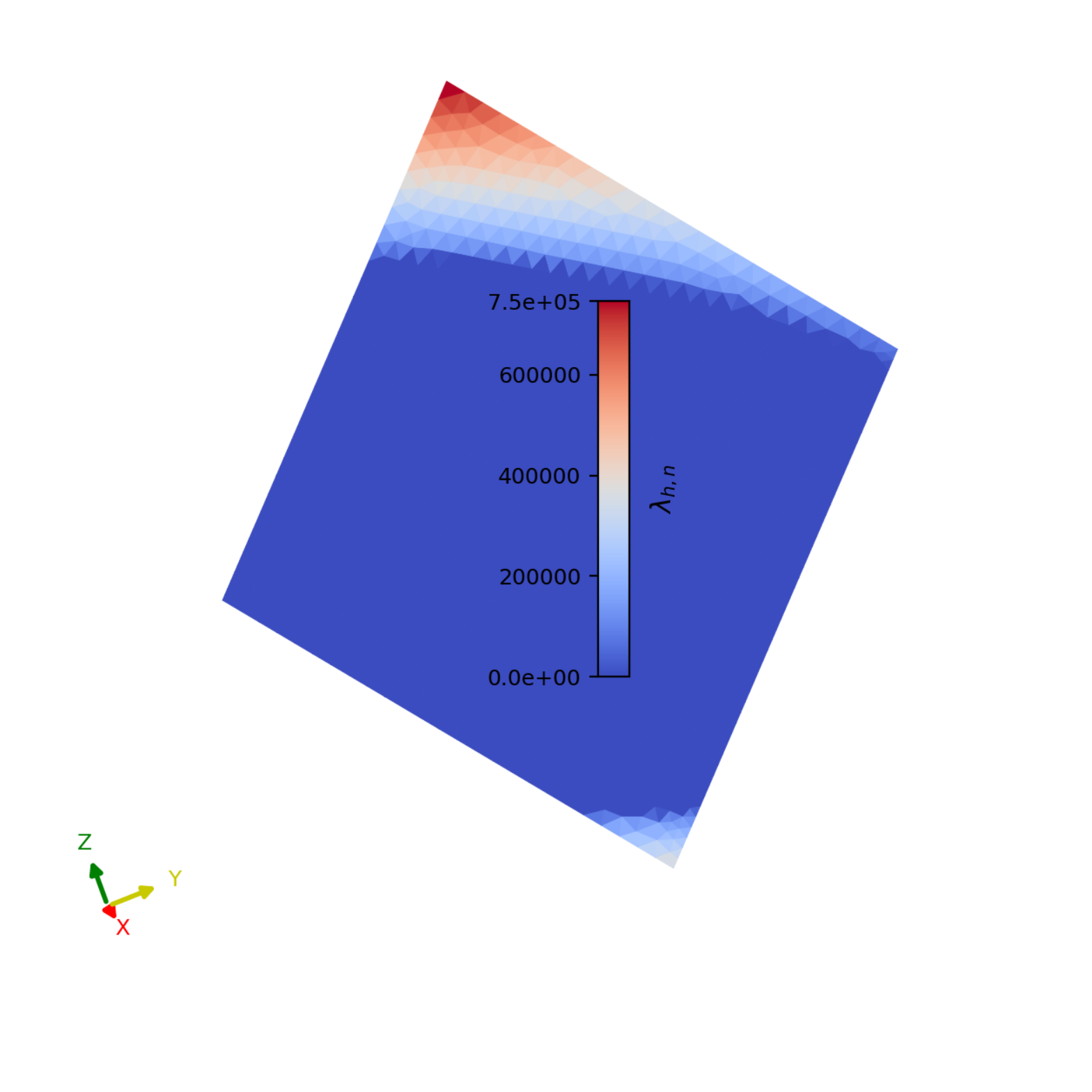}
  \caption{}
\end{subfigure}
\begin{subfigure}{0.33\textwidth}
  \includegraphics[width=\textwidth]{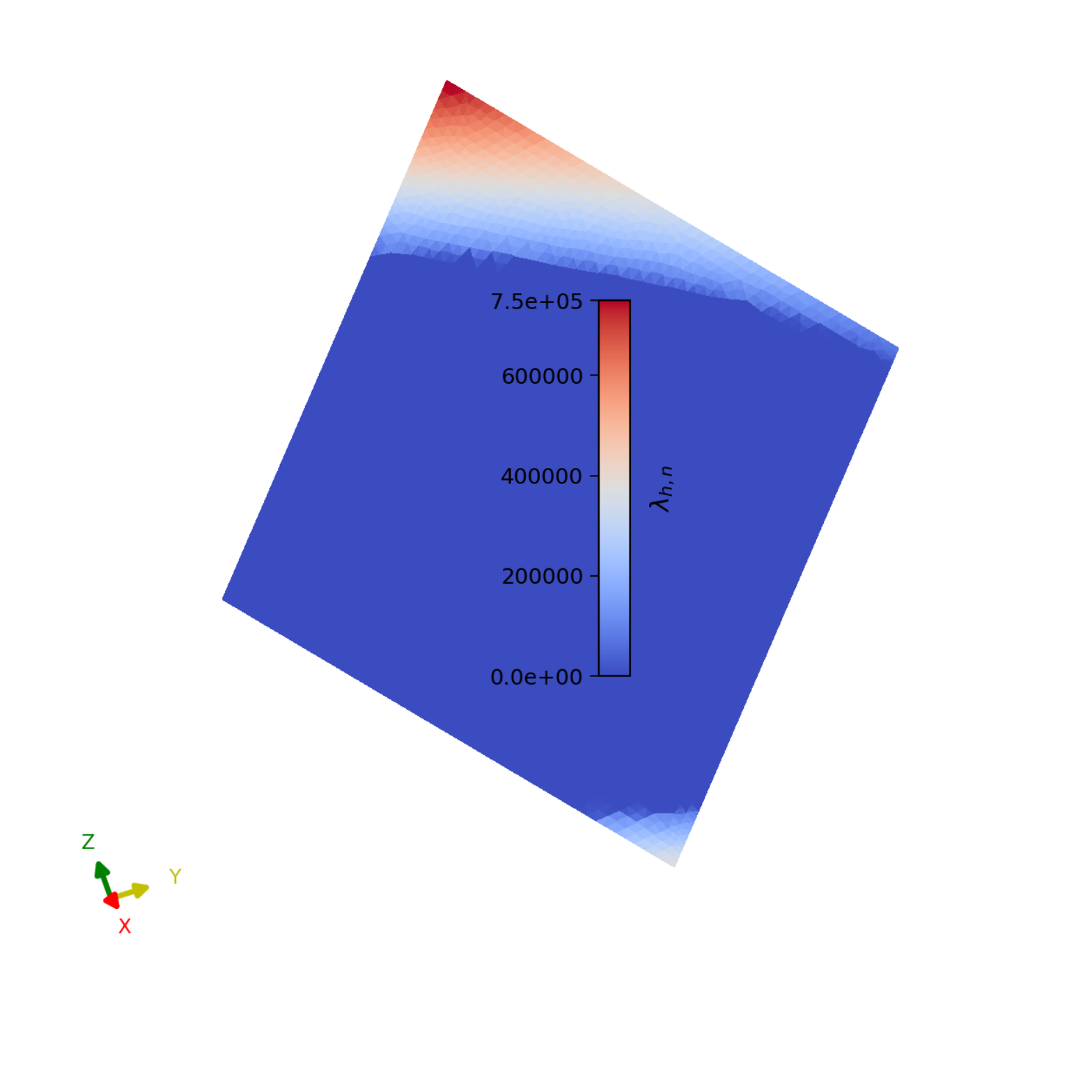}
  \caption{}
\end{subfigure}
\caption{Results of the test case in Section~\ref{exm:intersecting.fracture}: normal component of the discrete Lagrange multiplier on fracture F1 for tetrahedral meshes with the projected jump~\eqref{eq:jump.modify}. The panels correspond to (a) $(k,\ell)=(0,0)$, (b) $(k,\ell)=(1,0)$, and (c) $(k,\ell)=(1,1)$.}  
\label{fig:F1.normal.comp.Lag.Multi}  
\end{figure}

\begin{figure}[H]
\begin{subfigure}{0.33\textwidth}
  \includegraphics[width=\textwidth]{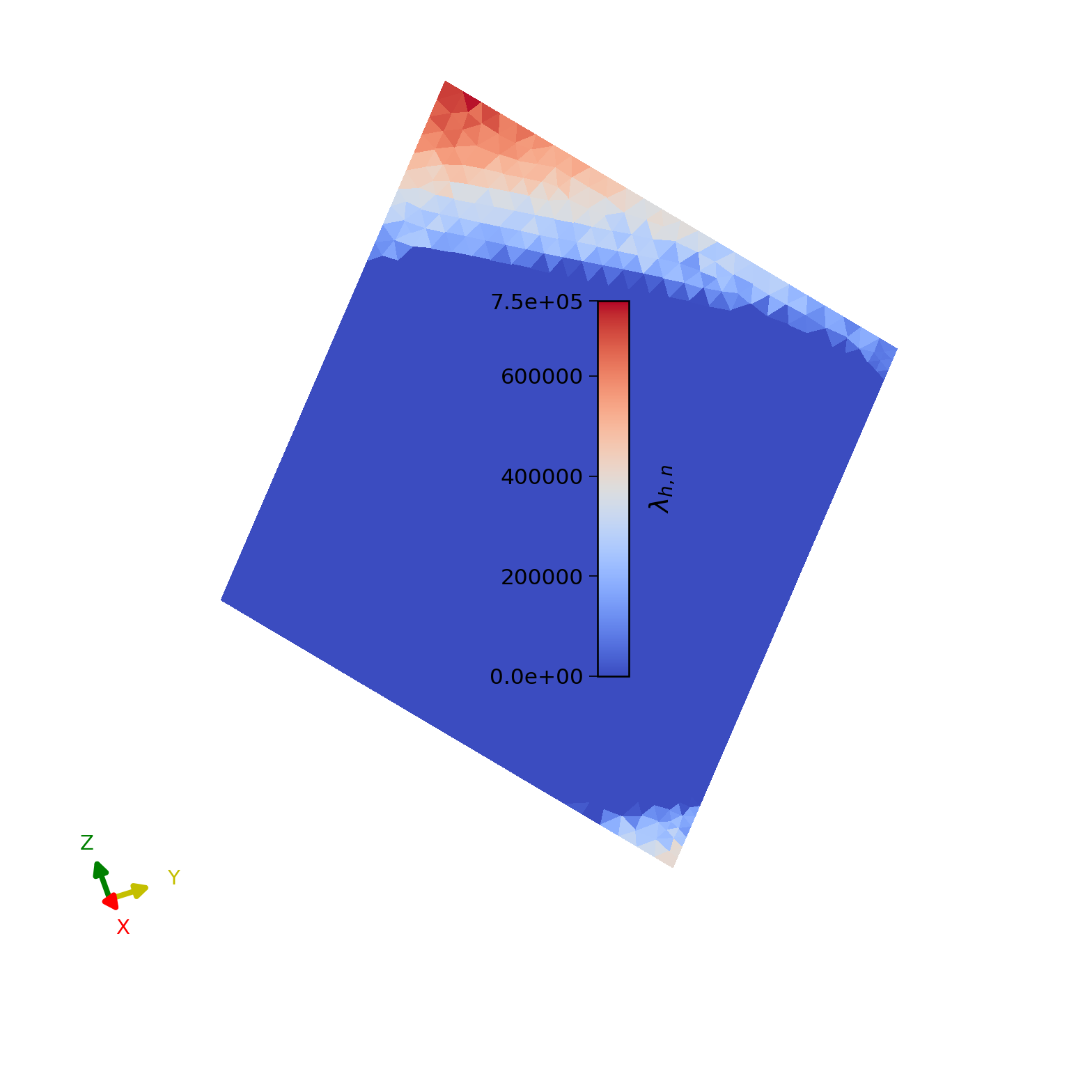}
  \caption{}
\end{subfigure}
\begin{subfigure}{0.33\textwidth}
  \includegraphics[width=\textwidth]{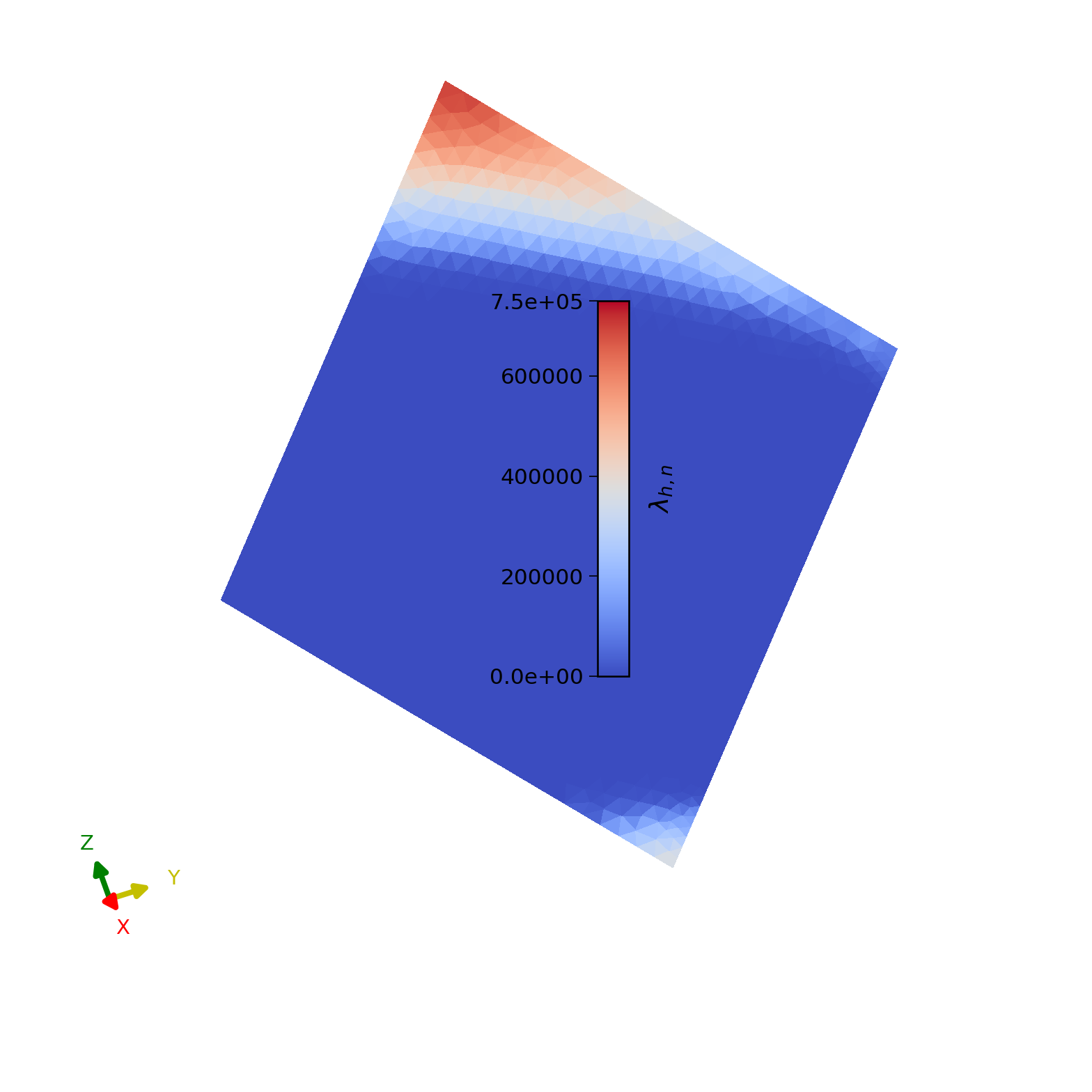}
  \caption{}
\end{subfigure}
\begin{subfigure}{0.33\textwidth}
  \includegraphics[width=\textwidth]{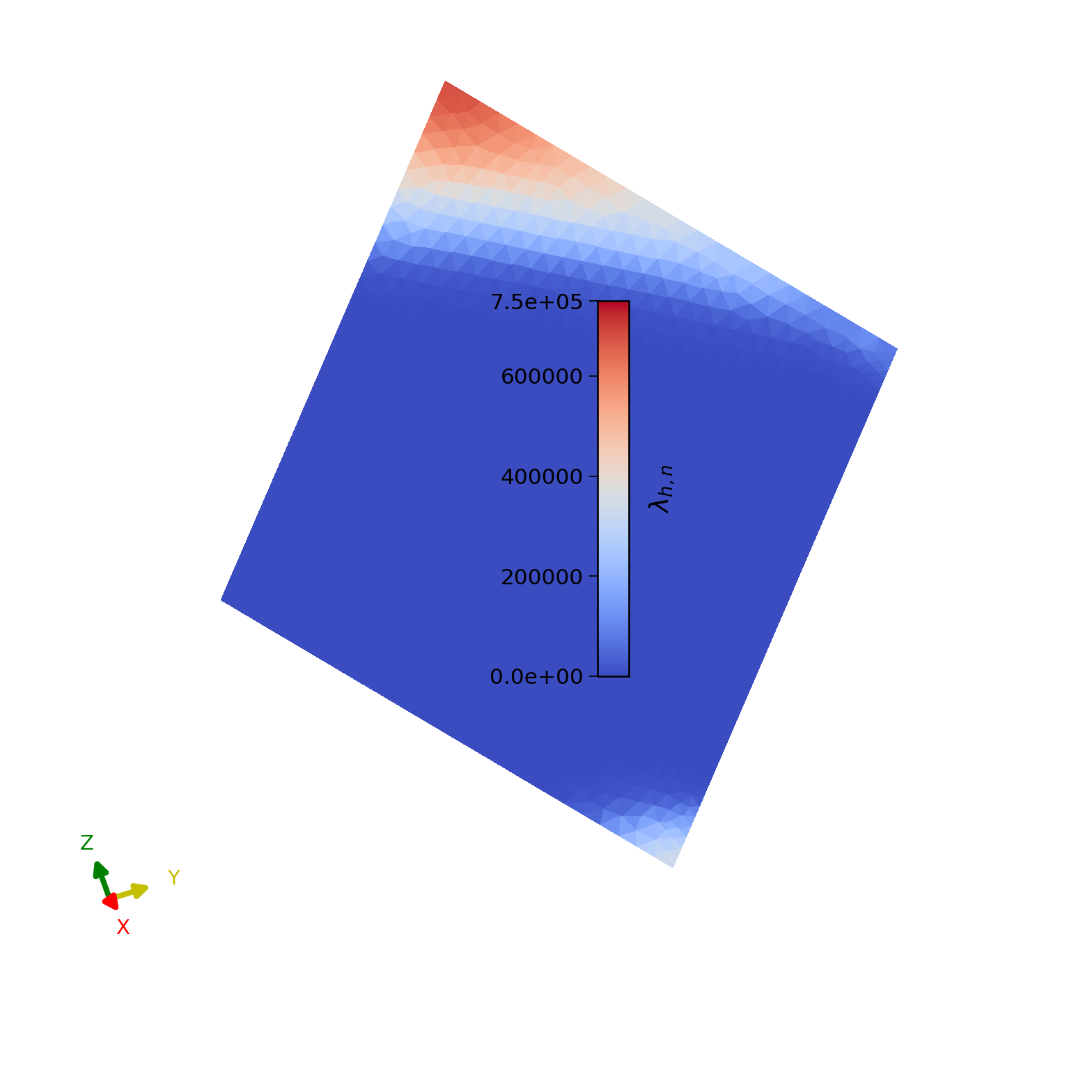}
  \caption{}
\end{subfigure}
\caption{Results of the test case in Section~\ref{exm:intersecting.fracture}: normal component of the discrete Lagrange multiplier on fracture F1 for $k=\ell=0$ with the standard jump~\eqref{eq:jump.n}. Panels (a), (b), and (c) correspond to $\delta=0$, $0.5$, and $1$, respectively.}
\label{fig:F1.normal.comp.Lag.Multi.smooth}  
\end{figure}


\begin{figure}[H]
\begin{subfigure}{0.50\textwidth}
  \includegraphics[width=\textwidth]{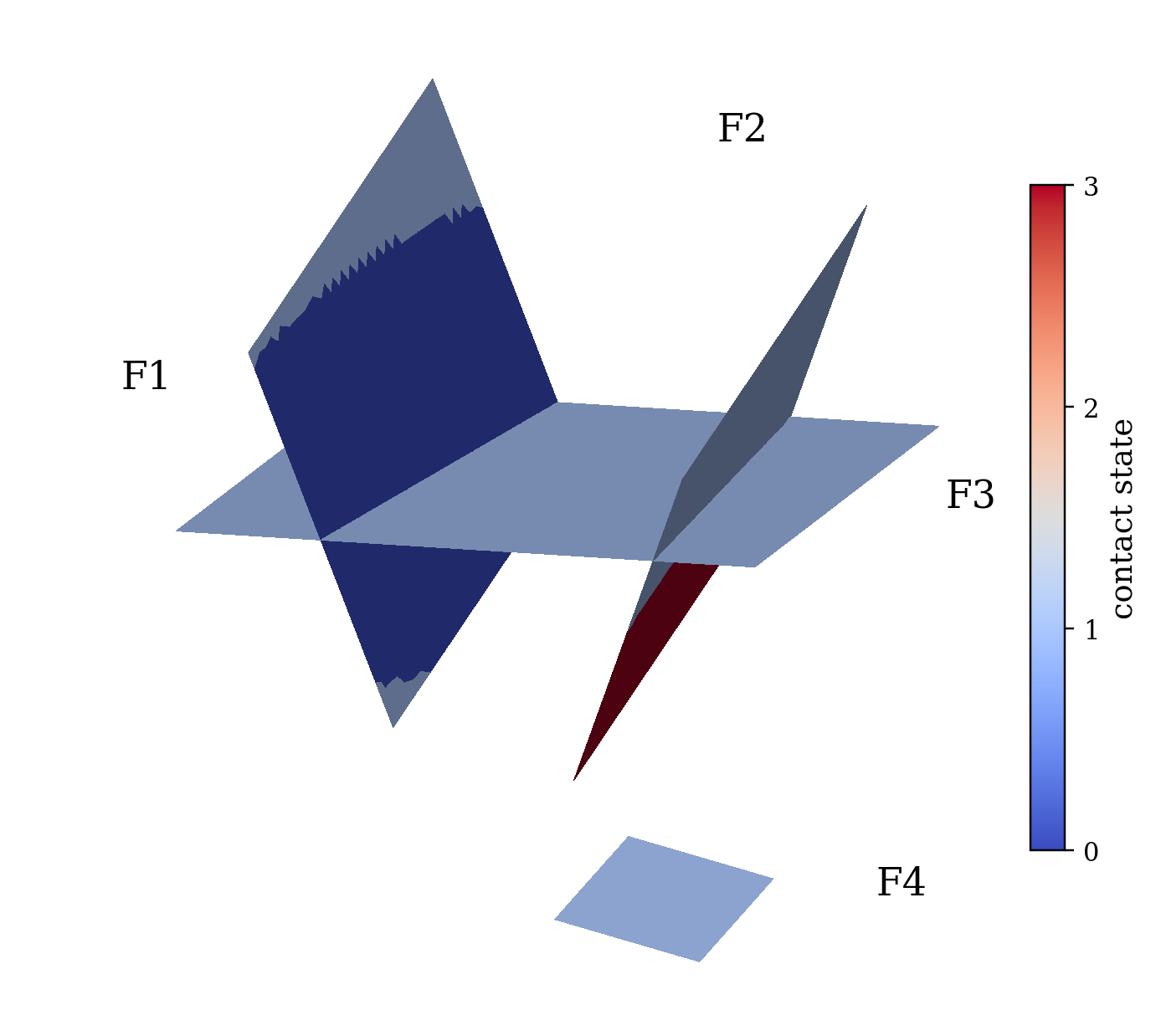}
  \caption{}
\end{subfigure}
\begin{subfigure}{0.50\textwidth}
  \includegraphics[width=\textwidth]{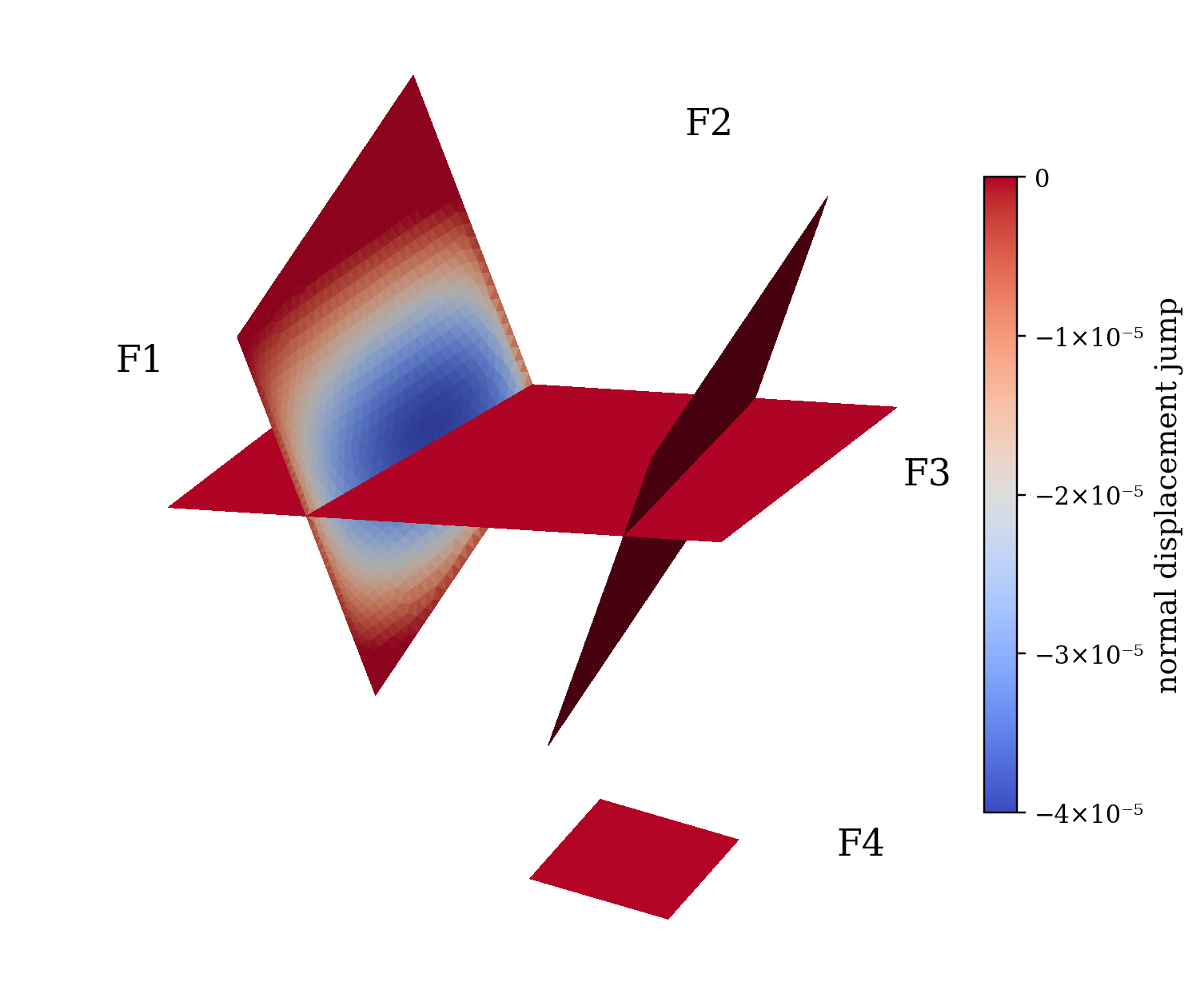}
  \caption{}
\end{subfigure}
\caption{Results for the test case in Section~\ref{exm:intersecting.fracture} for $k=1$ and $\ell=0$ with projected jump~\eqref{eq:jump.modify}: (a) contact state classification, where the values indicate: 0 if $\left|\blambda_{h,\tau}\right| < g$ and $\djumphn{\uS{u}} <0$; 1 if $\left|\blambda_{h,\tau}\right| < g$ and $\djumphn{\uS{u}} =0$; 2 if $\left|\blambda_{h,\tau}\right| = g$ and $\djumphn{\uS{u}} <0$; and 3 if $\left|\blambda_{h,\tau}\right| =g$ and $\djumphn{\uS{u}} =0$; and (b) normal displacement jump obtained using a mesh with 123k cells and 8.8k fracture faces.}
\label{fig:contact_normal} 
\end{figure}
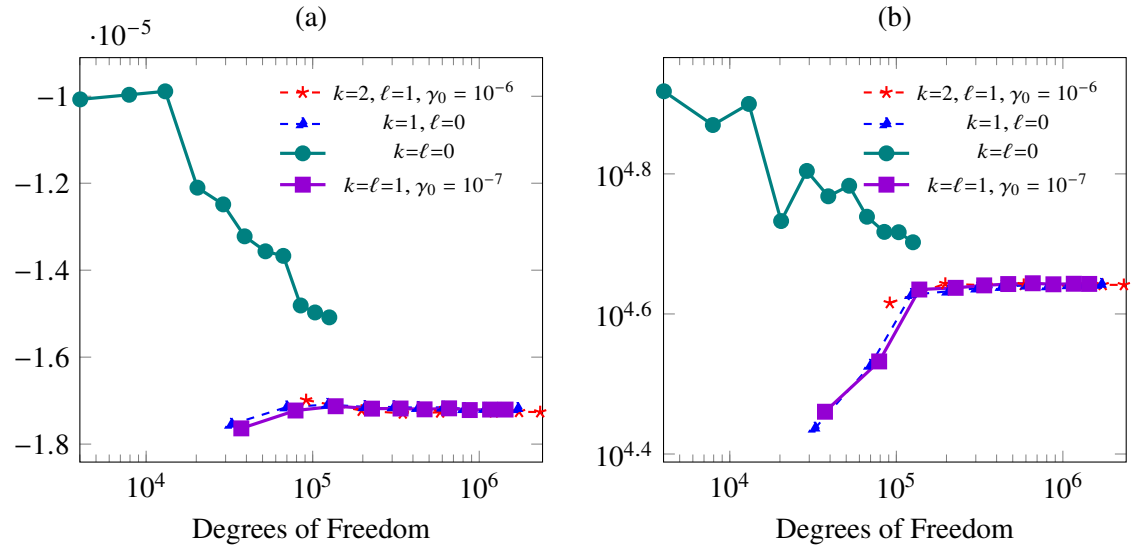
\begin{figure}[H]
\centering
\begin{tikzpicture}
\begin{groupplot}[
  group style={group size=2 by 1, horizontal sep=1.6cm},
  width=0.50\textwidth, height=0.45\textwidth,
  xlabel={Degrees of Freedom}, xmode=log,
  legend style={font=\scriptsize, draw=none, fill=none},
  every axis plot/.append style={mark size=2.4pt}
]
\nextgroupplot[xmin=4000, xmax=2400000, title={(a)}, legend pos=north east]

\addplot[red, thick, dashed, mark=star, mark size=2.4pt] coordinates { (91320,-1.698630000000e-05) (196893,-1.722910000000e-05) (347757,-1.729270000000e-05) (582330,-1.726930000000e-05) (871029,-1.725600000000e-05) (1222893,-1.726700000000e-05) (1734255,-1.725680000000e-05) (2323335,-1.726320000000e-05) };
\addlegendentry{$k{=}2,\ell{=}1,\gamma_{0}=10^{-6}$}
\addplot[blue, thick, dashed, mark=triangle*, mark size=2.4pt] coordinates { (32463,-1.754820000000e-05) (69777,-1.714960000000e-05) (122997,-1.709110000000e-05) (205515,-1.714390000000e-05) (306981,-1.714670000000e-05) (430479,-1.716270000000e-05) (609507,-1.715400000000e-05) (815547,-1.718850000000e-05) (1087011,-1.718370000000e-05) (1348707,-1.718680000000e-05) (1703745,-1.718870000000e-05) };
\addlegendentry{$k{=}1,\ell{=}0$}
\addplot[teal, very thick, solid, mark=*, mark size=2.4pt] coordinates { (4020,-1.007190000000e-05) (7911,-9.965860000000e-06) (13023,-9.886850000000e-06) (20328,-1.210210000000e-05) (29019,-1.248690000000e-05) (39033,-1.322150000000e-05) (52077,-1.356630000000e-05) (66621,-1.367110000000e-05) (84714,-1.481210000000e-05) (103302,-1.497250000000e-05) (125859,-1.508510000000e-05) };
\addlegendentry{$k{=}\ell{=}0$}
\addplot[violet, very thick, solid, mark=square*, mark size=2.4pt] coordinates { (37287,-1.763840000000e-05) (78873,-1.722770000000e-05) (137373,-1.712920000000e-05) (227091,-1.718450000000e-05) (336909,-1.718160000000e-05) (469599,-1.720040000000e-05) (659475,-1.717760000000e-05) (877275,-1.721720000000e-05) (1162347,-1.720570000000e-05) (1439019,-1.720570000000e-05) };
\addlegendentry{$k{=}\ell{=}1,\gamma_{0}=10^{-7}$}

\nextgroupplot[xmin=4000, xmax=2400000, ymode=log,
  title={(b)}, legend pos=north east]

\addplot[red, thick, dashed, mark=star, mark size=2.4pt] coordinates { (91320,4.127220000000e+04) (196893,4.396780000000e+04) (347757,4.373990000000e+04) (582330,4.380560000000e+04) (871029,4.364210000000e+04) (1222893,4.376820000000e+04) (1734255,4.380290000000e+04) (2323335,4.379410000000e+04) };
\addlegendentry{$k{=}2,\ell{=}1,\gamma_{0}=10^{-6}$}
\addplot[blue, thick, dashed, mark=triangle*, mark size=2.4pt] coordinates { (32463,2.732660000000e+04) (69777,3.360660000000e+04) (122997,4.239900000000e+04) (205515,4.287900000000e+04) (306981,4.331840000000e+04) (430479,4.354270000000e+04) (609507,4.372960000000e+04) (815547,4.363110000000e+04) (1087011,4.378700000000e+04) (1348707,4.379700000000e+04) (1703745,4.383820000000e+04) };
\addlegendentry{$k{=}1,\ell{=}0$}
\addplot[teal, very thick, solid, mark=*, mark size=2.4pt] coordinates { (4020,8.276750000000e+04) (7911,7.408180000000e+04) (13023,7.941020000000e+04) (20328,5.402230000000e+04) (29019,6.369990000000e+04) (39033,5.860360000000e+04) (52077,6.067360000000e+04) (66621,5.479410000000e+04) (84714,5.213160000000e+04) (103302,5.207690000000e+04) (125859,5.038070000000e+04) };
\addlegendentry{$k{=}\ell{=}0$}
\addplot[violet, very thick, solid, mark=square*, mark size=2.4pt] coordinates { (37287,2.886050000000e+04) (78873,3.405260000000e+04) (137373,4.313220000000e+04) (227091,4.337210000000e+04) (336909,4.372430000000e+04) (469599,4.391640000000e+04) (659475,4.403030000000e+04) (877275,4.388150000000e+04) (1162347,4.395760000000e+04) (1439019,4.393420000000e+04) };
\addlegendentry{$k{=}\ell{=}1,\gamma_{0}=10^{-7}$}
\end{groupplot}
\end{tikzpicture}
\caption{Results of the test case in  Section~\ref{exm:intersecting.fracture}: comparison of $(k,\ell)=(0,0),(1,0),(1,1),(2,1)$ on fracture F1 with projected jump~\eqref{eq:jump.modify}: (a) mean normal displacement
jump and (b) mean normal component of the discrete Lagrange multiplier. Solid thick lines denote $k=\ell$; dashed lines denote $\ell=k-1$.}
\label{fig:normal.fracture1.multi}
\end{figure}
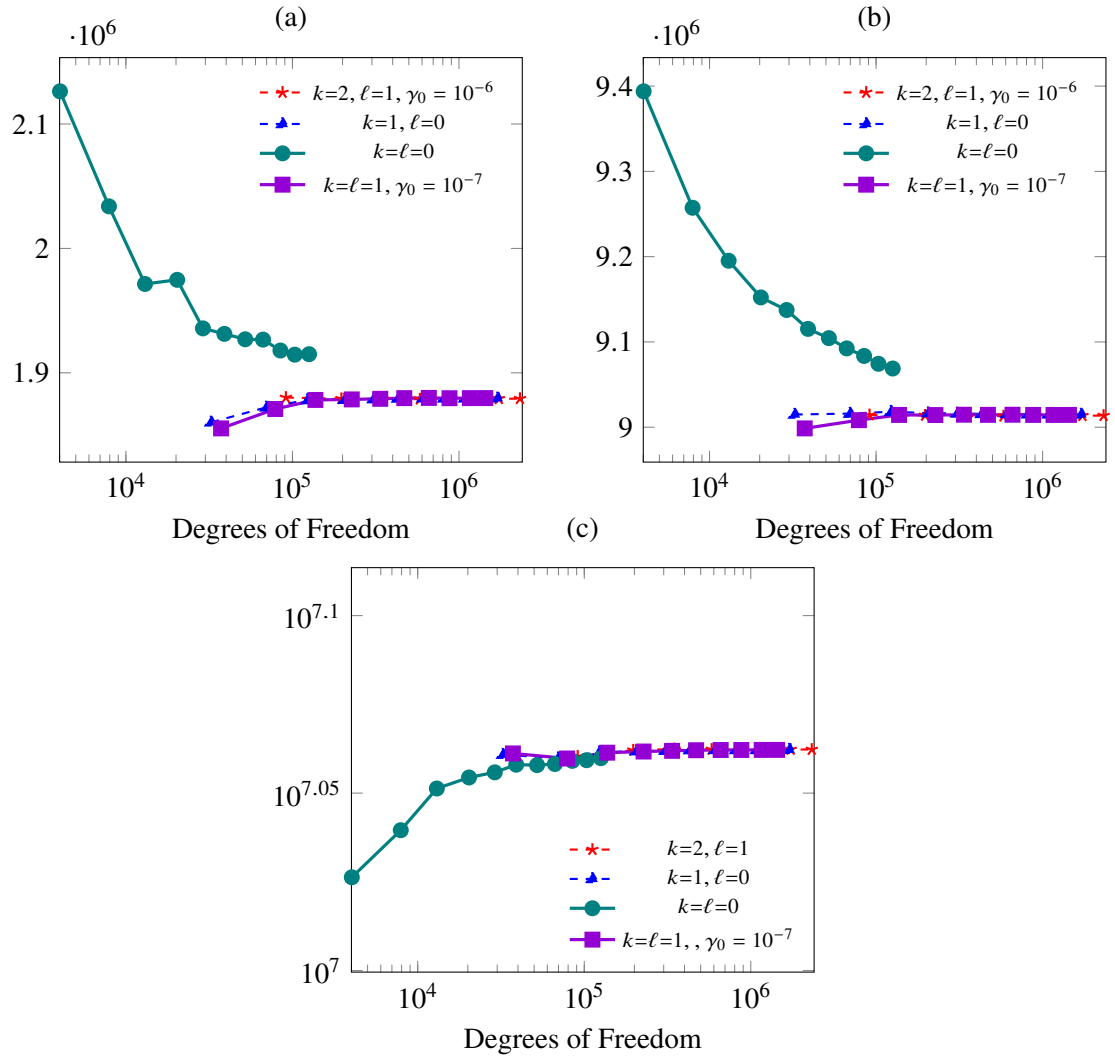
\begin{figure}[H]
\centering
\begin{tikzpicture}
\begin{groupplot}[
  group style={group size=2 by 1, horizontal sep=1.6cm},
  width=0.50\textwidth, height=0.45\textwidth,
  xlabel={Degrees of Freedom}, xmode=log,
  legend style={font=\scriptsize, draw=none, fill=none},
  every axis plot/.append style={mark size=2.4pt}
]
\nextgroupplot[xmin=4000, xmax=2400000, title={(a)}, legend pos=north east]

\addplot[red, thick, dashed, mark=star, mark size=2.4pt] coordinates { (91320,1.880140000000e+06) (196893,1.879170000000e+06) (347757,1.878980000000e+06) (582330,1.878840000000e+06) (871029,1.879030000000e+06) (1222893,1.879020000000e+06) (1734255,1.879180000000e+06) (2323335,1.879220000000e+06) };
\addlegendentry{$k{=}2,\ell{=}1,\gamma_{0}=10^{-6}$}
\addplot[blue, thick, dashed, mark=triangle*, mark size=2.4pt] coordinates { (32463,1.860220000000e+06) (69777,1.872270000000e+06) (122997,1.878140000000e+06) (205515,1.878380000000e+06) (306981,1.878890000000e+06) (430479,1.879400000000e+06) (609507,1.879580000000e+06) (815547,1.879480000000e+06) (1087011,1.879510000000e+06) (1348707,1.879630000000e+06) (1703745,1.879560000000e+06) };
\addlegendentry{$k{=}1,\ell{=}0$}
\addplot[teal, very thick, solid, mark=*, mark size=2.4pt] coordinates { (4020,2.126240000000e+06) (7911,2.033810000000e+06) (13023,1.971420000000e+06) (20328,1.974750000000e+06) (29019,1.935720000000e+06) (39033,1.931320000000e+06) (52077,1.926950000000e+06) (66621,1.926690000000e+06) (84714,1.917930000000e+06) (103302,1.914610000000e+06) (125859,1.914980000000e+06) };
\addlegendentry{$k{=}\ell{=}0$}
\addplot[violet, very thick, solid, mark=square*, mark size=2.4pt] coordinates { (37287,1.855350000000e+06) (78873,1.870930000000e+06) (137373,1.878210000000e+06) (227091,1.878640000000e+06) (336909,1.879120000000e+06) (469599,1.879680000000e+06) (659475,1.879840000000e+06) (877275,1.879610000000e+06) (1162347,1.879680000000e+06) (1439019,1.879700000000e+06) };
\addlegendentry{$k{=}\ell{=}1,\gamma_{0}=10^{-7}$}

\nextgroupplot[xmin=4000, xmax=2400000, title={(b)}, legend pos=north east]

\addplot[red, thick, dashed, mark=star, mark size=2.4pt] coordinates { (91320,9.014370000000e+06) (196893,9.013870000000e+06) (347757,9.013020000000e+06) (582330,9.012990000000e+06) (871029,9.013220000000e+06) (1222893,9.013370000000e+06) (1734255,9.013520000000e+06) (2323335,9.013600000000e+06) };
\addlegendentry{$k{=}2,\ell{=}1,\gamma_{0}=10^{-6}$}
\addplot[blue, thick, dashed, mark=triangle*, mark size=2.4pt] coordinates { (32463,9.014870000000e+06) (69777,9.015960000000e+06) (122997,9.018150000000e+06) (205515,9.016520000000e+06) (306981,9.016160000000e+06) (430479,9.015660000000e+06) (609507,9.015590000000e+06) (815547,9.015180000000e+06) (1087011,9.015040000000e+06) (1348707,9.014950000000e+06) (1703745,9.014960000000e+06) };
\addlegendentry{$k{=}1,\ell{=}0$}
\addplot[teal, very thick, solid, mark=*, mark size=2.4pt] coordinates { (4020,9.393860000000e+06) (7911,9.257340000000e+06) (13023,9.195210000000e+06) (20328,9.152190000000e+06) (29019,9.137430000000e+06) (39033,9.115280000000e+06) (52077,9.104460000000e+06) (66621,9.092400000000e+06) (84714,9.083570000000e+06) (103302,9.074350000000e+06) (125859,9.068820000000e+06) };
\addlegendentry{$k{=}\ell{=}0$}
\addplot[violet, very thick, solid, mark=square*, mark size=2.4pt] coordinates { (37287,8.998610000000e+06) (78873,9.008270000000e+06) (137373,9.014380000000e+06) (227091,9.014250000000e+06) (336909,9.014640000000e+06) (469599,9.014480000000e+06) (659475,9.014620000000e+06) (877275,9.014410000000e+06) (1162347,9.014370000000e+06) (1439019,9.014390000000e+06) };
\addlegendentry{$k{=}\ell{=}1,\gamma_{0}=10^{-7}$}
\end{groupplot}
\begin{axis}[
  at={($(group c1r1.south east)!0.5!(group c2r1.south west)+(0,-1.4cm)$)},
  anchor=north, width=0.50\textwidth, height=0.45\textwidth,
  xlabel={Degrees of Freedom}, xmode=log,
  xmin=4000, xmax=2400000, ymin=1e7, ymax=1.3e7, ymode=log,
  title={(c)},
  legend style={at={(0.45,0.19)}, anchor= west,
                font=\scriptsize, draw=none, fill=none},
  every axis plot/.append style={mark size=2.4pt}
]

\addplot[red, thick, dashed, mark=star, mark size=2.4pt] coordinates { (91320,1.150180000000e+07) (196893,1.154780000000e+07) (347757,1.155780000000e+07) (582330,1.155430000000e+07) (871029,1.155480000000e+07) (1222893,1.155390000000e+07) (1734255,1.155390000000e+07) (2323335,1.155370000000e+07) };
\addlegendentry{$k{=}2,\ell{=}1$}
\addplot[blue, thick, dashed, mark=triangle*, mark size=2.4pt] coordinates { (32463,1.151490000000e+07) (69777,1.148900000000e+07) (122997,1.153300000000e+07) (205515,1.154100000000e+07) (306981,1.154500000000e+07) (430479,1.154950000000e+07) (609507,1.155070000000e+07) (815547,1.155000000000e+07) (1087011,1.155050000000e+07) (1348707,1.155060000000e+07) (1703745,1.155180000000e+07) };
\addlegendentry{$k{=}1,\ell{=}0$}
\addplot[teal, very thick, solid, mark=*, mark size=2.4pt] coordinates { (4020,1.063480000000e+07) (7911,1.096470000000e+07) (13023,1.126500000000e+07) (20328,1.134570000000e+07) (29019,1.138380000000e+07) (39033,1.144040000000e+07) (52077,1.143730000000e+07) (66621,1.144550000000e+07) (84714,1.146910000000e+07) (103302,1.147370000000e+07) (125859,1.148860000000e+07) };
\addlegendentry{$k{=}\ell{=}0$}
\addplot[violet, very thick, solid, mark=square*, mark size=2.4pt] coordinates { (37287,1.152410000000e+07) (78873,1.148640000000e+07) (137373,1.152950000000e+07) (227091,1.153840000000e+07) (336909,1.154360000000e+07) (469599,1.154810000000e+07) (659475,1.155000000000e+07) (877275,1.154980000000e+07) (1162347,1.155060000000e+07) (1439019,1.155070000000e+07) };
\addlegendentry{$k{=}\ell{=}1, ,\gamma_{0}=10^{-7}$}
\end{axis}
\end{tikzpicture}
\caption{Results of the test case in Section~\ref{exm:intersecting.fracture}: mean normal component of the discrete Lagrange multiplier
on fractures (a) F2, (b) F3, and (c) F4 for
$(k,\ell)=(0,0),(1,0),(1,1),(2,1)$ with projected
jump~\eqref{eq:jump.modify}. Solid thick lines denote $k=\ell$;
dashed lines denote $\ell=k-1$.}
\label{fig:lambda.fractures234.multi}
\end{figure}

\begin{table}[H]
\centering
\caption{Results of the test case in Section~\ref{exm:intersecting.fracture}: Number of Newton iterations required to converge within
the desired threshold tolerance of $10^{-12}$.}
\label{tab:dfm.newton.comparison}
\small
\setlength{\tabcolsep}{6pt}
\begin{tabular}{llrrrrr}
\toprule
\multirow{2}{*}{$(k,\ell,\gamma_{0})=(0,0, 10^{-6})$}
 & Nb. Newton iterations & 7 & 8 & 8 & 8 & 8 \\
 & Degrees of Freedom & 4020 & 7911 & 66621 & 84714 & 103302 \\
\hline
\multirow{2}{*}{$(k,\ell,\gamma_{0})=(0,0, 10^{-8})$}
 & Nb Newton iterations & 6 & 6 & 7 & 8 & 8 \\
 & Degrees of Freedom & 4020 & 7911 & 66621 & 84714 & 103302 \\
 \hline
\multirow{2}{*}{$(k,\ell,\gamma_{0})=(1,0,10^{-6})$}
 & Nb Newton iterations & 5 & 5 & 7 & 6 & 6 \\
 & Degrees of Freedom & 32463 & 69777 & 815547 & 1087011 & 1348707\\
 \hline
\multirow{2}{*}{$(k,\ell,\gamma_{0})=(1,1,10^{-7})$}
 & Nb Newton iterations & 9 & 10 & 9 & 9 & 10 \\& Degrees of Freedom & 37287 & 78873 & 877275 & 1162347 & 1439019\\
 \hline
\multirow{2}{*}{$(k,\ell,\gamma_{0})=(1,1,10^{-8})$}
  & Nb Newton iterations & 7 & 7 & 8 & 8 & 8 \\
 & Degrees of Freedom & 37287 & 78873 & 877275 & 1162347 & 1439019\\
\bottomrule
\end{tabular}
\end{table}

\begin{remark}[Choice of $\gamma_0$ and $\delta$]
\label{rem:gamma0.delta}
The results shown in Figures~\ref{fig:tresca.tet}-\ref{fig:tresca.hex} correspond to $\gamma_0 = 0.1$ and $\delta = 10$. Note that the scheme exhibits similar convergence
behaviour for $\delta \in \{1, 10, 100, 1000\}$. In contrast, for $\delta=0$, the full-jump scheme exhibits essentially the same convergence behaviour as the projected-jump scheme, suggesting that stabilisation may not be necessary in practice. Table~\ref{tab:dfm.newton.comparison} addresses $\gamma_0$: at both $(k,\ell)=(0,0)$ and $(1,1)$, reducing $\gamma_0$ by one or two orders of magnitude never increases the Newton iteration count and typically decreases it (by up to $2$ and $3$
iterations, respectively), so Newton's method is robust to $\gamma_0$ over this range. The table also confirms that $(k,\ell)=(1,0)$ needs both fewer degrees of freedom and fewer Newton iterations ($5$--$7$) than $(0,0)$ at either $\gamma_0$ tested ($6$-$8$ at best), i.e., the higher-order scheme is more efficient by both metrics at once.
\end{remark}

\section{Existence and uniqueness of the discrete solution}\label{sec:Existence.uniqueness}

\subsection{Preliminary lemmas}

To prove the existence and uniqueness of a discrete solution in Proposition~\ref{wellposedness}, we need to establish a few lemmas.

\begin{lemma}[Consistency of the reconstruction operators]
\label{lem:prop.reconstructions}
The gradient and displacement reconstruction operators satisfy the 
following polynomial consistency properties:
\begin{alignat}{2}
\UpS{\Ksig}\IhS\bo{q} &= \bo{q} &&\qquad \forall K \in \cM,\quad \forall \sigma \in \cF_K,\quad \forall \bo{q} \in \Poly{k+1}(\sigma)^{3},\label{eq:UpS.face_1a}\\
\Pi^{k-1}_{\sigma}\Upsilon_{\Ksig}^{k+1} \uvec{v}_{h}&= \bo{v}_{\Ksig} && \qquad \forall K \in \cM,\quad \forall \sigma\in \faces_{K}, \quad \forall \uvec{v}_{h} \in \ddgsk.\label{eq:Pi.k.Ups}
\end{alignat}
By slight abuse of notation, the interpolator $\IhS$ applied above to a function defined only on a cell/face/edge corresponds to only interpolating on the degrees of freedom associated with that cell/face/edge (and its boundary); note that the corresponding gradient or displacement reconstruction depends only on these degrees of freedom.
\end{lemma}
\begin{proof}
We begin by proving \eqref{eq:UpS.face_1a}. Let $K\in \cells$ and let
$\sigma\in\faces_{\Gamma,K}^+$ of $K$ on the positive side of the fracture. Recalling the definition \eqref{eq:UpS.face}, we set $\uS{v}=\IhS \bo{q},$ where $\bo{q}\in \Poly{k+1}(\sigma)^3.$ Then,
\begin{align*}
\UpS{\Ksig}\IhS \bo{q}&=\Upsilon^{k+1}_{\Ksig}\IhS \bo{q} + \krond\Big[\I_{\cK\sigma}\bo{q} +\I_{\Ksig}^{\bub}\bo{q} -\Pi^\ell_{\sigma}\Upsilon^{k+1}_{\Ksig}\IhS\bo{q}\Big].
\\
\overset{\eqref{eq:DOF.Ih},\eqref{eq:IhS.decomposition}}&= 
\Upsilon^{k+1}_{\Ksig}\underline{\mathbb I}_h \bo{q} + \krond\Big[\Pi^{k-1}_{\sigma}\bo{q} +\Pi^{\ell}_{\sigma}(\bo{q}-\Pi^{k-1}_{\sigma}\bo{q}) -\Pi^\ell_{\sigma}\Upsilon^{k+1}_{\Ksig}\underline{\mathbb I}_h\bo{q}\Big]\nonumber\\
&= \bo{q} + \krond\Big[\Pi^{k-1}_{\sigma}\bo{q} +\Pi^{\ell}_{\sigma}(\bo{q}-\Pi^{k-1}_{\sigma}\bo{q}) -\Pi^\ell_{\sigma}\bo{q}\Big]=\bo{q}.
\end{align*} 
In the third equality, we used the polynomial consistency property $\Upsilon^{k+1}_{\Ksig}\underline{\mathbb I}_h \bo{q}=\bo{q},$ see \cite[Eq.~(4.7)]{DKMR25}. In the last step, we used the projector property $\Pi^{\ell}_{\sigma}\circ \Pi^{k-1}_{\sigma}=\Pi^{k-1}_{\sigma}$, valid since $\ell\ge k-1$.

Next, let $\sigma \notin \faces_{\Gamma,K}^+.$ In this case, by definition \eqref{eq:UpS.face_1} of $\UpS{\Ksig}$ and again invoking \cite[Eq.~(4.7)]{DKMR25}, we obtain $\UpS{\Ksig}\IhS \bo{q}=\Upsilon^{k+1}_{\Ksig}\underline{\mathbb I}_h \bo{q}=\bo{q}.$ This completes the proof of \eqref{eq:UpS.face_1a}.

The lemma only states the properties that are explicitly invoked in the analysis. Using \eqref{eq:UpS.face_1a} and the same arguments as in \cite[Lemma~3]{ddr-variant}, one also shows that $\GS{K}\IhS\bo{q}=\nabla\bo{q}$ and $\UpS{K}\IhS\bo{q}=\bo{q}$ for all $K\in\cM$ and all $\bo{q}\in\Poly{k+1}(K)^3$. Combined with the polynomial consistency of the edge reconstruction \cite[Eq.~(4.6)]{DKMR25}, this gives $\mathbb{S}_K(\IhS\bo{q},\IhS\bo{q})=0$.
\end{proof}

\begin{lemma}[Discrete lifting]\label{rem:lifting}
There exists $C_{\mathrm{R}},C_\Gamma,C_{\mathrm{D}}>0$  constants depending only on $\Omega$, $\Gamma$, the mesh regularity parameters, and the polynomial degrees $k$ and $\ell$ (and, in particular, independent of $h$, $\gamma_0,$ and of the Lam\'e coefficients) such that, for every $\bos{\mu}_{h}\in \Ms$, there exists $\uS{R} \in \ddgskS$ satisfying
\begin{subequations}\label{eq:prop.Rh}
\begin{align}
\label{eq:prop.Rh.jump}
\Pi_{\cF_{\Gamma}}^{\ell}\djumph{\uS{R}}&=\bos{\mu}_{h}, \\
\label{eq:prop.Rh.norm.energy}
\NORM{1,h}{\uS{R}}& 
 \leq C_{\mathrm{R}} h_{\Gamma}^{-\frac{1}{2}}\NORM{L^2(\Gamma)}{\bos{\mu}_{h}} ,\\
\label{eq:prop.Rh.norm.gamma}
\NORM{L^{2}(\Gamma)}{\jump{\uS{R}}_{h}} &\leq C_{\Gamma}\NORM{L^{2}(\Gamma)}{\bos{\mu}_{h}},\\
\label{eq:prop.Rh.ups.bound}
\left(\sum_{\sigma\in \faces_{\Gamma}}\NORM{L^{2}(\sigma)}{\Upsilon_{\Ksig}^{k+1}\uvec{R}_{h}-\Pi_{\sigma}^{\ell}\Upsilon_{\Ksig}^{k+1}\uvec{R}_{h}}^{2}\right)^{\frac{1}{2}} &\leq C_{\mathrm{D}}  \left(\sum_{e\in \edges_{\Gamma}^{\text{int}}}h_e \NORM{L^{2}(e)} {\jump{\bos{\mu}_{h}}_{e}}^2\right)^{\frac{1}{2}}.
\end{align}
\end{subequations}
\end{lemma}

\begin{proof}
In this proof, all the hidden constants in $\lesssim$ have the same dependency as in the statement of the lemma.

\textit{\uline{Decompose $\bos{\mu}_{h}$}:}
Since $\bos{\mu}_{h} \in \Ms$, we have $\bos{\mu}_{\sigma}=(\bos{\mu}_{h})_{|_\sigma} \in \Poly{\ell}(\sigma)^3$ for all faces $\sigma \in \cF_{\Gamma}.$ Then we can write $(\bos{\mu}_{h})_{|_\sigma} = \Pi^{k-1}_{\sigma}\bos{\mu}_{\sigma} + \Pi_{\sigma}^{\ell}(\bos{\mu}_{\sigma}-\Pi^{k-1}_{\sigma}\bos{\mu}_{\sigma}).$

\textit{\underline{Construction of $\RhS$}:} We first introduce the sets $\faces_{s}:= \left\{ \sigma^{\prime} \in \faces_{\Gamma}; s\in \nodes_{\sigma^{\prime}} \right\},$ $\faces_{e}:= \left\{ \sigma^{\prime} \in \faces_{\Gamma}; e\in \edges_{\sigma^{\prime}} \right\},$ and $\cells_{\Gamma}^{+}:=\{K\in \cells; \overline{K}\cap \Gamma \neq \emptyset,\;\faces_{\Gamma,K}^{+}\neq \emptyset\}.$ In words, $\cells^{+}_{\Gamma}$ is the set of mesh cells having at least one fracture face whose outward normal points to the positive side of $\Gamma$ (i.e.\ $\normal_{K\sigma}\cdot \normal^{+}>0$) --- the layer of cells bordering the fracture from its positive side.
The vector $\RhS\in\ddgskS$ is constructed by setting all components to zero except those attached to the positive side of $\Gamma$, which are defined as
\begin{equation}\label{eq:Rh.construction}
\begin{cases}
\bo{R}_K=\bo{0}& \forall K\in\cM_{\Gamma}^{+},\\[2pt]
\bo{R}_{\Ks} = \frac{1}{|\faces_{s}|} \sum\limits_{\sigma^{\prime} \in \faces_{s}} \bos{\mu}_{\sigma^{\prime}}(x_{s})& \forall\, s\in\cV_K^{\text{int}},\; K\in\cM_{\Gamma}^{+},\\[4pt]
\bo{R}_{\Ke} = \frac{1}{|\faces_{e}|}\Pi_{e}^{k-1}\sum\limits_{\sigma^{\prime}\in \faces_{e}}(\bos{\mu}_{\sigma^{\prime}})_{|_e}& \forall\, e\in\cE_K^{\text{int}},\; K\in\cM_{\Gamma}^{+},\\[4pt]
\bo{R}_{\cK\sigma} =\Pi_{\sigma}^{k-1}\bos{\mu}_{\sigma},\quad
\bo{R}^\bub_{\Ksig} = \krond \Pi_{\sigma}^{\ell}(\bos{\mu}_{\sigma}-\Pi_{\sigma}^{k-1}\bos{\mu}_{\sigma})
& \forall\,\sigma\in\cF^+_{\Gamma,K},\; K\in\cM_{\Gamma}^{+}.
\end{cases}
\end{equation}
\textit{\underline{Proof of \eqref{eq:prop.Rh.jump}}:}
Let $\sigma\in \faces_\Gamma$ and first assume that $\krond=1.$ From \eqref{eq:Rh.construction} and the property $\Pi_{\sigma}^{\ell}\circ\Pi_{\sigma}^{k-1}=\Pi_{\sigma}^{k-1}$, we observe that 
\begin{equation}\label{eq:sum_bub.Rh}
\bo{R}_{\Ksig}+\bo{R}_{\Ksig}^{\bub}=\Pi_{\sigma}^{k-1}\bos{\mu}_{\sigma}+\Pi_{\sigma}^{\ell}(\bos{\mu}_{\sigma}-\Pi_{\sigma}^{k-1}\bos{\mu}_{\sigma})= \bos{\mu}_{\sigma}.
\end{equation}
Denoting by $L$ the cell on the other side of $\sigma$ from $K$, we have $\bo{R}_{\Ls}=\bo{0},$ $\bo{R}_{\Le}=\bo{0},$ and $\bo{R}_{\Lsig}=0$, which implies that $\UpS{\Lsig}\uS{R}=\Upsilon^{k+1}_{\Lsig}\uvec{R}_{h}=\bo{0}$ from \eqref{eq:Up.face}-\eqref{eq:UpS.face}. Together with the definition \eqref{eq:jump.n} of discrete jump on $\sigma$
and \eqref{eq:sum_bub.Rh}, this yields
\begin{equation}\label{eq:jump.Rhn}
\Jsig{\uS{R}}=\Upsilon^{k+1}_{\Ksig}\uvec{R}_h + \bos{\mu}_{\sigma} -\Pi^\ell_\sigma\Upsilon^{k+1}_{K\sigma}\uvec{R}_h.    
\end{equation}
Now apply $\Pi^\ell_\sigma$ to both sides of the above equation to get
\begin{equation}\label{eq:krond.1}
\Pi_{\sigma}^{\ell}\Jsig{\uS{R}}= \bos{\mu}_{\sigma}.
\end{equation}
Next, let us consider $\krond=0,$ i.e., $\ell=k-1$ for $k\geq 1$ and $\ell=0$ for $k=0.$ Then from \eqref{eq:UpS.face_1}, the definition \eqref{eq:jump.n} of discrete jump, and using again $\UpS{\Lsig}\uS{R}=\Upsilon^{k+1}_{\Lsig}\uvec{R}_{h}=\bo{0}$, we get
\begin{equation*}
\Jsig{\uS{R}}=\Upsilon^{k+1}_{\Ksig}\uvec{R}_h.    
\end{equation*}
Now again apply $\Pi_{\sigma}^{\ell}$ to the above equation and  to get
\begin{equation}\label{eq:krond.0}
\Pi_{\sigma}^{k-1}\Jsig{\uS{R}}=\Pi_{\sigma}^{k-1}\Upsilon^{k+1}_{\Ksig}\uvec{R}_h\overset{\eqref{eq:Pi.k.Ups}}=\bos{\mu}_{\sigma}.    
\end{equation}
Finally, \eqref{eq:krond.1} and \eqref{eq:krond.0} completes the proof that $\Pi_{\cF_{\Gamma}}^{\ell}\djumph{\uS{R}}=\bos{\mu}_{h}$. \qed

\textit{\uline{Proof of \eqref{eq:prop.Rh.norm.energy}}:} 
The definition \eqref{eq:G1K_enhance} of $\GS{K}$ and the construction of $\RhS$ from \eqref{eq:Rh.construction} leads to, for all $\bo{\xi}\in\Poly{k}(K)^{3\times 3}$, 
\begin{equation*}
\int_K\GS{K}\RhS:\bos{\xi} =\sum_{\sigma\in\cF_K}\int_\sigma \UpS{\Ksig}\RhS\cdot\bos{\xi}\,\normal_{\Ksig}.
\end{equation*}
Substitute $\bos{\xi}=\GS{K}\RhS$  in the above equation and then apply the Cauchy--Schwarz's inequality to obtain
\begin{align*}
\NORM{L^2(K)}{\GS{K}\RhS}^2
&\leq \sum_{\sigma\in\cF_K}\NORM{L^2(\sigma)}{\UpS{\Ksig}\uS{R}}\NORM{L^2(\sigma)}{\GS{K}\RhS}\\
&\leq h_K^{-\frac{1}{2}}\NORM{L^2(K)}{\GS{K}\RhS}\sum_{\sigma\in\cF_K}\NORM{L^2(\sigma)}{\UpS{\Ksig}\uS{R}},
\end{align*}
where the conclusion follows from the discrete trace inequality \cite[Lemma~1.32]{hho-book}.
Simplifying by $\NORM{L^2(K)}{\GS{K}\RhS}$ and invoking \eqref{eq:UpS.face} leads to
\begin{equation}
\NORM{L^2(K)}{\GS{K}\RhS} \lesssim h_K^{-\frac{1}{2}}\left(\sum_{\sigma\in\cF_K}\NORM{L^2(\sigma)}{\Upsilon_{\Ksig}^{k+1}\uvec{R}_{h}}+\krond\sum_{\sigma\in \mathcal F_{\Gamma,K}^+} \NORM{L^{2}(\sigma)}{\bo{R}_{\cK\sigma} + \bo{R}^{\bub}_{\Ksig} -\Pi^\ell_{\sigma}\Upsilon^{k+1}_{\Ksig}\uvec{R}_h}\right).
\label{eq:GK.UPK}
\end{equation}
Now, the estimates from~\cite[Lemma 8]{Di-Pietro.Droniou.ea:24} give, for all $\uvec{v}_{h}\in \dgsk$,
\begin{equation}
\lVert \Up_{\Ksig}^{k+1}\uvec{v}_h\rVert_{L^2(\sigma)} \lesssim \NORM{L^2(\sigma)}{\bo{v}_{\cK\sigma}} + h_\sigma^{\frac{1}{2}}\sum_{e\in\cE_\sigma}\NORM{L^2(e)}{\bo{v}_\Ke} + h_\sigma \sum_{s\in\mathcal V_\sigma}\lvert \bo{v}_\Ks\rvert. \label{eq:facePotdof.bound}
\end{equation}
Discrete trace inequalities (from edges to faces, and from vertex to edges to faces) and the mesh regularity assumption ($h_{e}\approx h_{\sigma}$) give, for any edge $e$ or vertex $s$ on $\Gamma$,
\begin{equation}\label{eq:edge.face.tr.ineq}
\begin{alignedat}{2}
  \NORM{L^2(e)}{(\bos{\mu}_{\sigma^{\prime}})_{|_e}} &\lesssim h_{e}^{-\frac{1}{2}} \NORM{L^2(\sigma^{\prime})}{\bos{\mu}_{\sigma^{\prime}}} &&\quad\forall \sigma^\prime \in\faces_{e},\\
  |\bos{\mu}_{\sigma^{\prime}}(x_{s})|
   \lesssim  h_{e}^{-\frac{1}{2}} 
   \NORM{L^2(e)}{(\bos{\mu}_{\sigma^{\prime}})_{|_e}} &\lesssim h_{e}^{-1} \NORM{L^2(\sigma^{\prime})}{\bos{\mu}_{\sigma^{\prime}}}&&\quad\forall \sigma\in\faces_s.
\end{alignedat}
\end{equation}
Make $\uvec{v}_h=\uvec{R}_h$ in \eqref{eq:facePotdof.bound}, recall the definition \eqref{eq:Rh.construction} and apply~\eqref{eq:edge.face.tr.ineq} to get
\begin{equation}\label{eq:up.Rh.dof.bound}
\lVert \Up_{\Ksig}^{k+1}\uvec{R}_h\rVert_{L^2(\sigma)} \lesssim \sum_{\sigma'\in\faces_{\sigma}}\NORM{L^{2}(\sigma')}{\bos{\mu}_{{\sigma'}}},
\end{equation}
where $\faces_\sigma$ is the set of all $\sigma'\in\faces_\Gamma$ connected to $\sigma$ by an edge or a vertex. Using \eqref{eq:up.Rh.dof.bound}, the construction of $\uS{R}$, \eqref{eq:sum_bub.Rh}, and the definition of $\Pi_{\sigma}^{\ell}$ in right-hand side of \eqref{eq:GK.UPK} leads to 
\begin{align}
\NORM{L^2(K)}{\GS{K}\RhS}^2 
&\lesssim h_{K}^{-1} \sum_{\sigma\in\faces_{\Gamma,K}^+}\sum_{\sigma'\in\faces_{\sigma}}\NORM{L^2({\sigma'})}{\bos{\mu}_{{\sigma'}}}^{2}\nonumber\\
\overset{\eqref{eq:quasi.uniform.fracture}}&\lesssim 
h_{\Gamma}^{-1}\sum_{\sigma\in\faces_{\Gamma,K}^+}{\sum_{\sigma'\in\faces_{\sigma}}}\NORM{L^2(\sigma')}{\bos{\mu}_{\sigma'}}^{2}.
\label{eq:GK.Rhs.corr_1}
\end{align}
Now, the summation of all $K\in \cells$ in the above inequality yields, since each $\sigma'\in\faces_\Gamma$ appears in $\lesssim 1$ sets $\faces_\sigma$,
\begin{equation}\label{Gh.RhS.new}
\sum_{K\in \cells} \NORM{L^2(K)}{\GS{K}\RhS}^2 \lesssim h_{\Gamma}^{-1} \NORM{L^{2}(\Gamma)}{\bos{\mu}_{h}}^2.
\end{equation}
Let us turn to bounding $\stabin{\uS{R}}{\uS{R}}$. Using similar arguments as in \cite[Lemma 8]{Di-Pietro.Droniou.ea:24}, we have, for all $\uS{v}\in \dgskS$,
\begin{align}
  \lVert \UpS{K}\uS{v}\rVert_{L^2(K)} \lesssim{}&\NORM{L^2(K)}{\bo{v}_K} + h_K^{\frac{1}{2}}\sum_{\sigma\in\cF_K} \NORM{L^2(\sigma)}{\bo{v}_{\cK\sigma}} +h_K^{\frac{1}{2}}\krond \sum_{\sigma\in\cF_{\Gamma,K}^{+}} \NORM{L^2(\sigma)}{\bo{v}_{\Ksig}^{\bub}}\notag\\
  &  + h_{K} \sum_{e\in\cE_K} \NORM{L^2(e)}{\bo{v}_\Ke} + h_K^{\frac{3}{2}}\sum_{s\in\mathcal V_K}\lvert \bo{v}_\Ks\rvert,\label{eq:cellPotdof.bound}\\
\lVert \Up_{Ke}^{k+1}\uvec{v}_h\rVert_{L^2(e)}\lesssim{}& \NORM{L^2(e)}{\bo{v}_\Ke} + h^{\frac{1}{2}}_e\sum_{s\in\mathcal V_e}\lvert \bo{v}_\Ks\rvert. \label{eq:edgPotdof.bound}
\end{align}
Use the definition \eqref{eq:stabK} of $\mathbb{S}_{K}$ along with the construction of $\RhS$ and the triangle inequality, we get
\begin{align}
\mathbb S_{K}(\uS{R},\uS{R})\leq {}&h^{-1}_K\sum_{\sigma\in\cF_K}
\NORM{L^{2}(\sigma)}{\UpS{K}\uS{R}-{\Upsilon}^{k+1}_{\Ksig}\uvec{R}_{h}}^{2}\nonumber\\
& +\sum_{e\in\cE_K}\NORM{L^{2}(e)}{\UpS{K}\uS{R}}^{2} +\sum_{e\in\cE_K}\NORM{L^{2}(e)}{\Upsilon_{Ke}^{k+1}\uvec{R}_{h}}^{2}
+h_K\sum_{s\in\mathcal V_K}|\UpS{K}\uS{R}(\bx_s)|^{2}\nonumber\\
& \quad + h^{-1}_{K}\krond\sum_{\sigma\in \mathcal F_{\Gamma,K}^+}\NORM{L^{2}(\sigma)}{\bo{R}_{\cK\sigma} + \bo{R}^{\bub}_{\Ksig} 
-\Pi^\ell_{\sigma}\Upsilon^{k+1}_{\Ksig}\uvec{R}_h}^{2}.\label{eq:stab.Sh}
\end{align}
Note that the terms $\bo{R}_{\Ks}$ that should appear have been absorbed, using discrete trace inequalities, into the terms $\NORM{L^{2}(e)}{\Upsilon_{Ke}^{k+1}\uvec{R}_{h}}$, owing to $\Upsilon_{Ke}^{k+1}\uvec{R}_{h}(\bo{x}_s)=\bo{R}_{\Ks}$ (\eqref{eq:def.UpS.Ke}). 
Applying the discrete trace inequality~\cite[Lemma~1.32]{hho-book} once more (from $\sigma$ down to each edge $e\in \cE_\sigma$ and vertex $s\in \mathcal{V}_{\sigma}$) and using the mesh regularity assumption, we have
$$
\begin{aligned}
\NORM{L^{2}(e)}{\UpS{K}\uS{R}}^{2} &\lesssim h_K^{-1}\NORM{L^{2}(\sigma)}{\UpS{K}\uS{R}}^{2}, \\
h_{K}|\UpS{K}\uS{R}(\bx_s)|^{2} &\lesssim h_{K}^{-1}\NORM{L^{2}(\sigma)}{\UpS{K}\uS{R}}^{2}.
\end{aligned}
$$
Using the triangle inequality on the first term of \eqref{eq:stab.Sh} we therefore infer
\begin{align}
\mathbb S_{K}(\uS{R},\uS{R})\lesssim {}&h^{-1}_K\sum_{\sigma\in\cF_K} \NORM{L^{2}(\sigma)}{\UpS{K}\uS{R}}^2
+h^{-1}_K\sum_{\sigma\in\cF_K}\NORM{L^{2}(\sigma)}{\Upsilon^{k+1}_{\Ksig}\uvec{R}_{h}}^{2}\nonumber\\
& +\sum_{e\in\cE_K}\NORM{L^{2}(e)}{\Upsilon_{Ke}^{k+1}\uvec{R}_{h}}^{2}
 + h^{-1}_{K}\krond\sum_{\sigma\in \mathcal F_{\Gamma,K}^+}\NORM{L^{2}(\sigma)}{\bo{R}_{\cK\sigma} + \bo{R}^{\bub}_{\Ksig} 
-\Pi^\ell_{\sigma}\Upsilon^{k+1}_{\Ksig}\uvec{R}_h}^{2}.\label{eq:stab.Sh.bis}
\end{align}
We now bound $\NORM{L^{2}(\sigma)}{\UpS{K}\uS{R}}^{2}.$
An application of the discrete trace inequality leads to 
\begin{align*}
h_K\NORM{L^{2}(\sigma)}{\UpS{K}\uS{R}}^2&\lesssim \NORM{L^{2}(K)}{\UpS{K}\uS{R}}^2 \nonumber\\ \overset{\eqref{eq:cellPotdof.bound},\eqref{eq:Rh.construction}}&\lesssim h_{K}\sum_{\sigma\in \faces_{K}} \NORM{L^2(\sigma)}{\Pi_{\sigma}^{k-1}\bos{\mu}_{\sigma}}^2+h_{K}\krond \sum_{\sigma\in\cF_{\Gamma,K}^{+}} \NORM{L^2(\sigma)}{\Pi_{\sigma}^{\ell}(\bos{\mu}_{\sigma}-\Pi_{\sigma}^{k-1}\bos{\mu}_{\sigma})}^2\nonumber\\&
\qquad\qquad +h_{K}^2 \sum_{e \in \cE_{K}^{\text{int}}} \NORM{L^{2}(e)}{\bo{R}_{\Ke}}^2 + h_{K}^3 \sum_{s \in \mathcal{V}_{K}^{\text{int}}} |\bo{R}_{\Ks}|^2.
\end{align*}
In a similar way as in~\eqref{eq:up.Rh.dof.bound}, using~\eqref{eq:edge.face.tr.ineq}, the $L^2$-boundedness of $\Pi_{e}^{k-1},$ and the mesh regularity,
$$
\NORM{L^{2}(e)}{\bo{R}_{\Ke}}^2 \lesssim h_{e}^{-1} \sum_{\sigma^{\prime} \in \cF_{e}} \NORM{L^{2}(\sigma^{\prime})}{\mu_{\sigma^{\prime}}}^2,
\qquad |\bo{R}_{\Ks}|^2 \lesssim h_{e}^{-2}\sum_{\sigma^{\prime} \in \cF_{s}}\NORM{L^{2}(\sigma^{\prime})}{\mu_{\sigma^{\prime}}}^2.
$$
Combined with the $L^2$-boundedness of $\Pi_{\sigma}^{k-1},$ $\Pi_{\sigma}^{\ell}$, $h_{K}\approx h_{e}$ and grouping all index set via $\cF_{\sigma},$ this gives
\begin{equation}\label{eq:cell.Pot.Rh}
   \NORM{L^{2}(\sigma)}{\UpS{K}\uS{R}}\lesssim 
   \left(\sum_{\sigma\in\faces_{\Gamma,K}^+}{\sum_{\sigma'\in\faces_{\sigma}}}\NORM{L^2(\sigma')}{\bos{\mu}_{\sigma'}}^{2}\right)^{\frac{1}{2}}.
\end{equation}
The last term of~\eqref{eq:stab.Sh.bis} requires no separate argument: splitting it via the triangle inequality using the DOFs $\bo{R}_{\cK\sigma}$ and $\bo{R}^{\bub}_{\Ksig}$ of~\eqref{eq:Rh.construction}, the first two terms recombine exactly, by~\eqref{eq:sum_bub.Rh} and the $L^2(\sigma)$-orthogonality of $\Poly{k-1}(\sigma)$ and $\Poly{\ell \perp k-1}(\sigma)$, into  
$\NORM{L^{2}(\sigma)}{\bo{R}_{\cK\sigma} + \bo{R}^{\bub}_{\Ksig} 
}^{2}=\NORM{L^{2}(\sigma)}{\bo{R}_{\cK\sigma}}^2 + \NORM{L^{2}(\sigma)}{\bo{R}^{\bub}_{\Ksig} 
}^{2}= \NORM{L^{2}(\sigma)}{\mu_{\sigma}}^2 $, while the $L^2$-boundedness of $\Pi_{\sigma}^{\ell}$ bound the remaining term by $\Upsilon^{k+1}_{\Ksig}\uvec{R}_h,$ which is the same term controlled by \eqref{eq:up.Rh.dof.bound}. Hence, we have
\begin{equation}\label{eq:final.term.stab}
h^{-1}_{K}\sum_{\sigma\in \mathcal F_{\Gamma,K}^+}\NORM{L^{2}(\sigma)}{\bo{R}_{\cK\sigma} + \bo{R}^{\bub}_{\Ksig} 
-\Pi^\ell_{\sigma}\Upsilon^{k+1}_{\Ksig}\uvec{R}_h}^{2}
\lesssim h_{K}^{-1}\sum_{\sigma\in \mathcal F_{\Gamma,K}^+}\sum_{\sigma^{\prime}\in \cF_{\sigma}} \NORM{L^{2}(\sigma^{\prime})}{\bos{\mu}_{\sigma^{\prime}}}^{2}.
\end{equation}
Applying~ \eqref{eq:cell.Pot.Rh}, \eqref{eq:up.Rh.dof.bound}, \eqref{eq:edgPotdof.bound}, and \eqref{eq:final.term.stab} into \eqref{eq:stab.Sh.bis} yields, by mesh regularity assumption,
 \begin{equation}\label{eq:stab.Sh.2}
  \stabin{\RhS}{\RhS}^{\frac{1}{2}} \lesssim \left(\sum_{\sigma\in \faces_{\Gamma}}h_{\sigma}^{-1}\NORM{L^{2}(\sigma)}{\bos{\mu}_{\sigma}}^2\right)^{\frac{1}{2}}
\overset{\eqref{eq:quasi.uniform.fracture}}\lesssim h_{\Gamma}^{-\frac{1}{2}}\NORM{L^{2}(\Gamma)}{\bos{\mu}_{h}} .
 \end{equation}
Finally, the definition of $\NORM{1,h}{\cdot}$ along with \eqref{Gh.RhS.new} and \eqref{eq:stab.Sh.2} completes the proof of \eqref{eq:prop.Rh.norm.energy}.  \qed

\textit{\uline{Proof of \eqref{eq:prop.Rh.norm.gamma}}:} 
This estimate follows from \eqref{eq:jump.Rhn} and \eqref{eq:up.Rh.dof.bound} and summing over $\sigma\in\faces_\Gamma.$ \qed

\textit{\uline{Proof of \eqref{eq:prop.Rh.ups.bound}}:} 
Recalling that $\bos{\mu}_{\sigma}\in \Poly{\ell}(\sigma)$, the polynomial consistency property \cite[Eq.~(4.7)]{DKMR25} implies $\Upsilon^{k+1}_{\Ksig}\underline{\mathbb I}_h \bos{\mu}_{\sigma}=\bos{\mu}_{\sigma}$. From this and the definition of $\Pi_{\sigma}^{\ell},$ we have 
$$
 \NORM{L^{2}(\sigma)}{\Upsilon_{\Ksig}^{k+1}\uvec{R}_{h}-\Pi_{\sigma}^{\ell}\Upsilon_{\Ksig}^{k+1}\uvec{R}_{h}}
 =\NORM{L^{2}(\sigma)}{\Upsilon_{\Ksig}^{k+1}\uvec{\widetilde{R}}_{\sigma}-\Pi_{\sigma}^{\ell}\Upsilon_{\Ksig}^{k+1}\uvec{\widetilde{R}}_{\sigma}}   
$$
where $\uvec{\widetilde{R}}_{\sigma}$ is the local vector, on $\sigma$ and its boundary, whose components are those of $\uvec{R}_{h}-\underline{\mathbb I}_h \bos{\mu}_\sigma$. Applying the triangle inequality and again invoking the definition of 
$\Pi_{\sigma}^{\ell}$ in the above equation, we obtain
\begin{align}
 \sum_{\sigma \in\faces_\Gamma} \NORM{L^{2}(\sigma)}{\Upsilon_{\Ksig}^{k+1}\uvec{R}_{h}-\Pi_{\sigma}^{\ell}\Upsilon_{\Ksig}^{k+1}\uvec{R}_{h}}^2 &\lesssim \sum_{\sigma \in\faces_\Gamma} \NORM{L^{2}(\sigma)}{\Upsilon_{\Ksig}^{k+1}\uvec{\widetilde{R}}_{{\sigma}{}}}^2 \notag\\
 \overset{\eqref{eq:facePotdof.bound}}& \lesssim  \sum_{\sigma \in\faces_\Gamma}\left(h_\sigma\sum_{e\in\cE_\sigma}\NORM{L^2(e)}{\widetilde{\bo{R}}_{{\sigma},\Ke}}^2 + h_\sigma^2 \sum_{s\in\mathcal V_\sigma}\lvert \widetilde{\bo{R}}_{\sigma,\Ks}\rvert^2\right),\label{eq:rhs.Rh.tilde}
\end{align}
where we have used, in the conclusion, the fact that $\widetilde{\bo{R}}_{\sigma,\Ksig}=\bo{R}_{\Ksig}-\Pi_{\sigma}^{k-1}\bo{\mu}_\sigma=\bo{0}$ by \eqref{eq:Rh.construction}. We will now estimate both terms on the right-hand side of \eqref{eq:rhs.Rh.tilde}. From the definition of $\widetilde{\bo{R}}_{\Ke}$ and $h_{\sigma}\approx h_{e}$, we have 
\begin{align}\label{Re.tilde}
\sum_{\sigma \in\faces_\Gamma}h_{\sigma}\sum_{e\in\cE_\sigma}\NORM{L^{2}(e)}{\widetilde{\bo{R}}_{\sigma,\Ke}}^2&\lesssim \sum_{\sigma \in\faces_\Gamma} \sum_{e\in\cE_\sigma}h_{e}\NORM{L^{2}(e)}{\bo{R}_{\Ke}- \Pi_{e}^{k-1}(\bos{\mu}_{\sigma})_{|_e}}^2 \notag\\
 \overset{\eqref{eq:Rh.construction}} &\lesssim \sum_{\sigma \in\faces_\Gamma}\sum_{e\in\cE_\sigma}h_{e}\NORM{L^{2}(e)}{\Big(\frac{1}{|\faces_{e}|}\Pi_{e}^{k-1}\sum_{\sigma^{\prime}\in \faces_{e}}(\bos{\mu}_{\sigma^{\prime}})_{|_e}\Big)-\Pi_{e}^{k-1}(\bos{\mu}_{\sigma})_{|_e}}^2 \nonumber\\& \lesssim \sum_{\sigma \in\faces_\Gamma}\sum_{e\in\cE_\sigma}h_{e} \NORM{L^{2}(e)}{\jump{\bos{\mu}_{\sigma}}_{e}}^2 \lesssim \sum_{e \in\edges_\Gamma^{\text{int}}} h_{e} \NORM{L^{2}(e)}{\jump{\bos{\mu}_{\sigma}}_{e}}^2,
\end{align}
where we have written
$$
\Big(\frac{1}{|\faces_{e}|}\Pi_{e}^{k-1}\sum_{\sigma^{\prime}\in \faces_{e}}(\bos{\mu}_{\sigma^{\prime}})_{|_e}\Big)-\Pi_{e}^{k-1}(\bos{\mu}_{\sigma})_{|_e}=\frac{1}{|\faces_{e}|}\Pi_{e}^{k-1}\sum_{\sigma^{\prime}\in \faces_{e}}\left[(\bos{\mu}_{\sigma^{\prime}})_{|_e}-(\bos{\mu}_{\sigma})_{|_e}\right]
$$
to pass to the third step, and used the $L^2$-boundedness of $\Pi_{e}^{k-1}$ in the fourth step. An application of the discrete trace inequality \cite[Lemma~1.32]{hho-book} and \cite[Eq.~(7.63)]{hho-book} lead to
\begin{align}
\sum_{\sigma \in\faces_\Gamma} h_\sigma^2 \sum\limits_{s\in\mathcal V_\sigma}\lvert \widetilde{\bo{R}}_\Ks\rvert^2 &\lesssim \sum_{\sigma \in\faces_\Gamma} \NORM{L^{2}(\sigma)}{\Big(\frac{1}{|\faces_{s}|} \sum_{\sigma^{\prime} \in \faces_{s}} \bos{\mu}_{\sigma^{\prime}}\Big)-\bos{\mu}_{\sigma}}^2\notag \\ &\lesssim \sum_{\sigma \in\faces_\Gamma} \sum_{e\in \edges_{s,\sigma}}h_{e}\NORM{L^{2}(e)}{\jump{\bos{\mu}_{\sigma}}_{e}}^2\lesssim \sum_{e \in\edges_\Gamma^{\text{int}}} h_{e} \NORM{L^{2}(e)}{\jump{\bos{\mu}_{\sigma}}_{e}}^2 ,\label{eq:Rs.tilde}
\end{align}
where $\edges_{s,\sigma}:=\left\{e\in \edges_{\Gamma}; \overline{e}\cap\partial\sigma \neq \emptyset \right\}.$  Combining \eqref{eq:rhs.Rh.tilde}, \eqref{Re.tilde} and \eqref{eq:Rs.tilde} concludes the lemma.
\end{proof}

\begin{remark}[Comparison with the discrete lifting of \cite{Burman.Hansbo.Larson:17}]
\label{rem:comparison.lifting}
The lifting of \cite[Lemma~3.5]{Burman.Hansbo.Larson:17} is built through a Lagrange-nodal cut-off function~(global indicator function) $\xi_h$ vanishing at the boundary nodes of the contact zone and required to satisfy both \cite[Eq.~(3.6)]{Burman.Hansbo.Larson:17} and \cite[Assumption~3.1]{Burman.Hansbo.Larson:17}; the latter is established only via the existence, on each simplex, of a quadrature rule on $\P^{\ell+1}(K)$ with strictly positive weights and only interior points, a condition that must be verified degree by degree and, in practice, caps the analysis at $k\le 12$ in three dimensions. Our lifting $\uS{R}$ is built directly from the DDR DOFs structure with face enrichment, so neither a cut-off nor a quadrature-positivity condition is needed, and Lemma~\ref{rem:lifting} therefore holds for every $k\ge 0$ and $\ell\in\{\max(0,k-1),k\}$.
\end{remark}
\begin{lemma}[Monotonicity property]\label{lem:mono.prop}
For all $\bo{a}_1,\bo{a}_2,\boldsymbol{\xi}_1,\boldsymbol{\xi}_2$ functions $\Gamma\to\R^3$, it holds
$$
\NORM{L^{2}(\Gamma)}{\bPg{\bo{a}_1}{\boldsymbol{\xi}_1}-\bPg{\bo{a}_2}{\boldsymbol{\xi}_{2}}}^{2}
\leq \int_{\Gamma}\big[\bPg{\bo{a}_1}{\boldsymbol{\xi}_{1}}-\bPg{\bo{a}_2}{\boldsymbol{\xi}_{2}}\big]
\cdot\big[\big(\bo{a}_1-\bo{a}_2\big)+\gamma\big(\boldsymbol{\xi}_{1}-\boldsymbol{\xi}_{2}\big)\big].
$$
\end{lemma}
\begin{proof}
By the definition~\eqref{eq:combine.Pg.def} of $\mathbf{P}_{\gamma}$, for $i=1,2$,
\begin{equation*}
\bPg{\bo{a}_i}{\boldsymbol{\xi}_{i}}
=\Pgp{\bo{a}_{i,\normal}}{\boldsymbol{\xi}_{i,\normal}}\normal^{+}
+\Pt{\bo{a}_{i,\tang}}{\boldsymbol{\xi}_{i,\tang}},
\end{equation*}
the first summand being along $\normal^{+}$ and the second lying in the tangent plane. Both factors on the right-hand side of the claimed inequality share this same orthogonal splitting, so the inequality decouples into its normal and tangential parts. The normal part follows from the scalar monotonicity~\cite[Lemma~3.2]{Burman.Hansbo.Larson:17} applied with arguments $\bo{a}_{i,\normal}+\gamma \boldsymbol{\xi}_{i,\normal}$, and the tangential part from the projection monotonicity~\cite[Eq.~8]{Chouly2014} (with radius $\gamma g$) applied with arguments $\bo{a}_{i,\tang}+\gamma \boldsymbol{\xi}_{i,\tang}$. Summing these two bounds and recombining via the orthogonal decompositions $\bo{a}_i=\bo{a}_{i}\normal^{+}+\bo{a}_{i,\tang}$ and $\boldsymbol{\xi}_{i}=\boldsymbol{\xi}_{i,\normal}\normal^{+}+\boldsymbol{\xi}_{i,\tang}$ concludes the lemma.
\end{proof}

\begin{lemma}[Stability]\label{lem:stability_results}
There exists a constant $C$ independent of $h$ and $\gamma_0$ such that, if $\alpha >0$ satisfies $\alpha C\gamma _{0} < 1$ and $2{\alpha}C_{\Gamma}^2 < 1$, then for  all $(\uS{u}, \laf_{h}) \in \ddgskS \times \Ms$, it holds
\begin{align*} 
  \left(1-\alpha C \gamma_{0} \right)&\NORM{\mathrm{en},h}{\uS{u}}^{2}
     + \left(1-2{\alpha}C_{\Gamma}^2\right)
    \NORM{L^{2}(\Gamma)}{\gamma^{-\frac{1}{2}}_h\bPg{\djumph{\uS{u}}}{\laf_h}
       - \gamma^{\frac{1}{2}}_h \laf_h
     }^2 \notag\\ &\qquad +\frac{5\alpha}{8}\NORM{L^{2}(\Gamma)}{\gamma^{\frac{1}{2}}_{h}\laf_h}^2
   \leq \mathbb{A}_h((\uS{u},\laf_h);(\uS{u}-\alpha\RhS,\laf_h)),
\end{align*}
where $\uS{R}$ is defined from Lemma~\ref{rem:lifting} with $\bos{\mu}_h=-\gamma_{h}\laf_h$.
\end{lemma}

\begin{proof}
By definition \eqref{eq:def.bilnear.form} of $\mathbb{A}_h$, \eqref{eq:innerproduct} of the inner product corresponding to $\NORM{\mathrm{en},h}{{\cdot}},$ \eqref{form1_ch} of $\mathbb{C}_{h}$, together with elementary manipulations, it follows that
\begin{multline}\label{lm1_eq1}
\mathbb{A}_h((\uS{u},\laf_h);(\uS{u},\laf_h)) \geq \NORM{\mathrm{en},h}{\uS{u}}^2 +\int_{\Gamma}\gamma^{-1}_{h}\bPg{\djumph{\uS{u}}}{\laf_h}\cdot \Big[\djumph{\uS{u}} - \gamma_{h}\laf_h \Big] \\ + \NORM{L^{2}(\Gamma)}{\gamma^{\frac{1}{2}}_{h}\laf_h}^2+ \mathbb{S}_{\Gamma}(\laf_{h},\laf_{h}),
\end{multline}
where we used the fact that $\zeta(E_K,\nu_K)\ge\mu_1>0$ for every $K\in\mathcal{M}.$
We now simplify all the terms~(except first and last) on the right-hand side of \eqref{lm1_eq1}.

From Lemma~\ref{lem:mono.prop} with $\uSe{v}{1}=\uS{\bu}$, $\uSe{v}{2}=\uS{\mathbf{0}},$ $\bos{\mu}_{1,h}=\laf_h,$ and $\bos{\mu}_{2,h}=\mathbf{0}_h,$ we observe that
\begin{equation}\label{eq:obs.2a}
\NORM{L^{2}(\Gamma)}{\bPg{\djumph{\uS{u}}}{\laf_h}}^2 \leq \int_{\Gamma} \bPg{\djumph{\uS{u}}}{\laf_h}\cdot\Big[\djumph{\uS{u}}+\gamma_{h}\laf_h\Big].
    \end{equation}
The second and third  terms on the right-hand side of~\eqref{lm1_eq1} therefore lead to
\begin{align}\label{eq:observation.2}
&\int_{\Gamma}\gamma^{-1}_{h}\bPg{\djumph{\uS{u}}}{\laf_h}\cdot \Big[
\djumph{\uS{u}} - \gamma_{h}\laf_h \Big] + \NORM{L^{2}(\Gamma)}{\gamma^{\frac{1}{2}}_{h}\laf_h}^2 \notag\\
& =  \int_{\Gamma}\gamma^{-1}_{h}\bPg{\djumph{\uS{u}}}{\laf_h}\cdot \Big[
\djumph{\uS{u}} + \gamma_{h}\laf_h \Big] -2\int_{\Gamma}\bPg{\djumph{\uS{u}}}{\laf_h}\cdot  \laf_h + \NORM{L^{2}(\Gamma)}{\gamma^{\frac{1}{2}}_{h}\laf_h}^2 \notag\\
\overset{\eqref{eq:obs.2a}}&\geq \NORM{L^{2}(\Gamma)}{\gamma^{-1}_{h}\bPg{\djumph{\uS{u}}} {\laf_h}}^2 -2\int_{\Gamma}\bPg{\djumph{\uS{u}}}{\laf_h}\cdot\laf_h + \NORM{L^{2}(\Gamma)}{\gamma^{\frac{1}{2}}_{h}\laf_h}^2 \notag\\
      &= \NORM{L^{2}(\Gamma)}{\gamma^{-\frac{1}{2}}_{h}\bPg{\djumph{\uS{u}}}{\laf_h}-\gamma^{\frac{1}{2}}_h\laf_h}^2.
\end{align}
Use \eqref{eq:observation.2}  in \eqref{lm1_eq1} to get
\begin{equation}\label{equivalent_norm_id}
  \mathbb{A}_h((\uS{u},\laf_h);(\uS{u},\laf_h))
  \geq \NORM{\mathrm{en},h}{\uS{u}}^2 + \NORM{L^{2}(\Gamma)}{\gamma^{-\frac{1}{2}}_{h}\bPg{\djumph{\uS{u}}} {\laf_h}-\gamma^{\frac{1}{2}}_{h}\laf_h}^2 + \mathbb{S}_{\Gamma}(\laf_{h},\laf_{h}).
\end{equation}
On the other hand, we have
\begin{align}
  \mathbb{A}_h{}&((\uS{u},\laf_h);(-\alpha\RhS,\bo{0}_h))\nonumber\\
  &= \int_\Omega\bbsigma_h(\uS{u}):\bbeps_h(-\alpha\RhS)
     + \stabc{\uS{u}}{-\alpha\RhS} 
       -\a \int_{\Gamma} \gamma^{-1}_{h}\bPg{\djumph{\uS{u}}}{\laf_h}\cdot \djumph{\RhS}\notag\\
  \overset{\pm\laf_{h}}&= \int_\Omega\bbsigma_h(\uS{u}):\bbeps_h(-\alpha\RhS)
    + \stabc{\uS{u}}{-\alpha\RhS}
  \notag\\& \quad \qquad 
  -\a \int_{\Gamma}\Big[ \gamma^{-1}_{h}\bPg{\djumph{\uS{u}}}{\laf_h} - \laf_h\Big]\cdot \djumph{\RhS} - \a \int_{\Gamma} \laf_h \cdot \jump{\RhS}. \label{lm1_manineq}
\end{align}
We now derive upper bounds for each term on the right-hand side of \eqref{lm1_manineq}; since these terms are all linear with respect to $-\alpha$, we drop for the moment this term (when we re-introduce it later, these upper bounds will become lower bounds). 
For the first term, using $\bbtau:\bbeps_{h}(\RhS)=\bbtau:\GS{K}(\RhS)$, the definition \eqref{eq:def.sigma.h} of $\bbsigma_{h}(\cdot)$ and the Cauchy--Schwarz inequality, we get
\begin{align}\label{eq:sigma_Rh_uh}
\int_\Omega\bbsigma_h(\uS{u}):\bbeps_h(\RhS) 
& \leq   \NORM{\mathrm{en},h}{\uS{u}}\NORM{\mathrm{en},h}{\uS{R}}\nonumber\\
\overset{\eqref{eq:norm_energy},\eqref{eq:prop.Rh.norm.energy}}& \leq C_{0}C_{\mathrm{R}}
\gamma^{\frac{1}{2}}_{0}
 \NORM{\mathrm{en},h}{\uS{u}}\NORM{L^{2}(\Gamma)}{\gamma^{\frac{1}{2}}_{h}\laf_{h}}\nonumber\\ 
&\leq 2(C_{0}C_{\mathrm{R}})^{2}\gamma_{0}\NORM{\mathrm{en},h}{\uS{u}}^{2}+ \frac{1}{8}\NORM{L^{2}(\Gamma)}{\gamma^{\frac{1}{2}}_{h}\laf_h}^2,
\end{align}
where $C_0$ is the equivalence constant from \eqref{eq:norm_energy} and  Young's inequality is used in the last step.

Next, the Cauchy--Schwarz inequality and definition of $\NORM{\mathrm{en},h}{{\cdot}}$ leads to
\begin{align}
\stabc{\uS{u}}{\RhS} &\leq \NORM{\mathrm{en},h}{\uS{u}}\stabc{\uS{R}}{\RhS}^{\frac{1}{2}}\notag\\
\overset{\eqref{eq:stab.Sh.2}}&\leq 2C_{1}C_{\mathrm{R}}\gamma^\frac{1}{2}_{0}\NORM{\mathrm{en},h}{\uS{u}}\NORM{L^{2}(\Gamma)}{\gamma^{\frac{1}{2}}_{h}\laf_{h}}\notag\\
&\leq 2(C_{1} C_{\mathrm{R}})^{2}\gamma_{0}\NORM{\mathrm{en},h}{\uS{u}}^2 + \frac{1}{8}\NORM{L^{2}(\Gamma)}{\gamma^{\frac{1}{2}}_{h}\laf_h}^2.\label{bilform_cont}
\end{align}
where $C_1$ is the equivalence constant from \eqref{eq:norm_energy} and  Young's inequality is used in the last step.
The third term in \eqref{lm1_manineq} is estimated by the Cauchy--Schwarz inequality and Young's inequality to get
\begin{multline*}
\int_{\Gamma}\Big[ \gamma^{-1}_{h}\bPg{\djumph{\uS{u}}}{\laf_h} -\laf_h\Big]\cdot \djumph{\uS{R}}\\ 
\overset{\eqref{eq:prop.Rh.norm.gamma}}\leq 2C_{\Gamma}^2 \NORM{L^{2}(\Gamma)}{ \gamma^{-\frac{1}{2}}_{h}\bPg{\djumph{\uS{u}}}{\laf_h} - \gamma^{\frac{1}{2}}_{h}\laf_h}^2
  + \frac{1}{8}\NORM{L^{2}(\Gamma)}{\gamma^{\frac{1}{2}}_{h}\laf_h}^2.
\end{multline*}
For the last term, using the definition of $\Pi_{\cF_{\Gamma}}^{\ell},$ we have 
\begin{equation}\label{lm1_ineq1}
\int_{\Gamma} \laf_h \cdot \jump{\RhS}=\int_{\Gamma} \laf_{h}\cdot \Pi_{\cF_{\Gamma}}^{\ell} \jump{\RhS} \overset{\eqref{eq:prop.Rh.jump}}=-\int_{\Gamma} \laf_h \cdot \gamma_{h} \laf_h
= -\NORM{L^2(\Gamma)}{\gamma^{\frac{1}{2}}_{h}\laf_h}^2.
\end{equation}
Use \eqref{eq:sigma_Rh_uh}--\eqref{lm1_ineq1}~(with multiplication of $-\a$ both sides)  in \eqref{lm1_manineq} to get
\begin{multline*}
  -2\alpha(C_{0}^2 +C_{1}^2) C_{\mathrm{R}}^{2}\gamma_{0}\NORM{\mathrm{en},h}{\uS{u}}^{2}
     -2{\alpha}C_{\Gamma}^2
     \NORM{L^{2}(\Gamma)}{
       \gamma^{-\frac{1}{2}}_{h}\bPg{\djumph{\uS{u}}}{\laf_h}
       - \gamma^{\frac{1}{2}}_{h}\laf_h}^2 \\ + \frac{5\alpha}{8}\NORM{L^{2}(\Gamma)}{\gamma^{\frac{1}{2}}_{h}\laf_h}^2\leq \mathbb{A}_h((\uS{u},\laf_h);(-\alpha\RhS,\bo{0}_h)).
\end{multline*}
Combine \eqref{equivalent_norm_id}~(by dropping non-negative term $\mathbb{S}_{\Gamma}(\laf_{h},\laf_{h})$) and the inequality above to get
\begin{multline*} 
  \left(1-\alpha C \gamma_{0} \right)\NORM{\mathrm{en},h}{\uS{u}}^{2}
     +\left(1-2{\alpha}C_{\Gamma}^2\right)
     \NORM{L^{2}(\Gamma)}{
       \gamma^{-\frac{1}{2}}_{h}\bPg{\djumph{\uS{u}}}{\laf_h}
       - \gamma^{\frac{1}{2}}_{h}\laf_h
     }^2 \notag\\  +\frac{5\alpha}{8}\NORM{L^{2}(\Gamma)}{\gamma^{\frac{1}{2}}_{h}\laf_h}^2
  \leq \mathbb{A}_h((\uS{u},\laf_h);(\uS{u}-\alpha\RhS,\laf_h)),
\end{multline*}
where we set $C:=2(C_{0}^2 +C_{1}^2) C_{\mathrm{R}}^{2}.$
This completes the proof. 
\end{proof}

\subsection{Existence and uniqueness of the solution}\label{sec:proof.existence.uniqueness}
We are now ready to prove Proposition~\ref{wellposedness}. The existence of a solution to the discrete scheme \eqref{eq:scheme} is established using the topological degree. 

\textit{\uline{Setting}:} We consider the finite dimensional space  $\dsml=\ddgskS\times \Ms$ equipped with the norm
\begin{equation*}
\NORM{\dsml}{V_{h}}:=\Big(\NORM{\mathrm{en},h}{\uS{v}}^2
+\NORM{L^2(\Gamma)}{\gamma^{\frac{1}{2}}_{h}\bos{\mu}_h}^2\Big)^{\frac{1}{2}} \;\; \text{ for all } \;\; V_{h}=(\uS{v},\bos{\mu}_h)\in\dsml.  
\end{equation*}
Let $(\cdot,\cdot)_{\dsml}$ be the inner product associated with this norm,
and define $\Psi(W_{h},\rho):\dsml\times[0,1]
\to\dsml$ such, that for all $W_{h}=(\uS{w},\bo{r}_h)$, $V_{h}=(\uS{v},\bos{\mu}_h)$ and $\rho\in [0,1]$,
\begin{align*}
(\Psi(W_{h},\rho),V_{h})_{\dsml} & := \int_\Omega\bbsigma_h(\uS{w}):\bbeps_h(\uS{v})
    + \stabc{\uS{w}}{\uS{v}}
     + \rho \mathbb{C}_h((\uS{w},\bo{r}_h);(\uS{v},\bos{\mu}_h))
     \nonumber\\&\qquad+ (1-\rho)\int_{\Gamma} \gamma_{h} \bos{r}_h\cdot \bos{\mu}_h - \sum_{K\in\cM}\int_K\mathbf{f}\cdot \bo{v}_K.
\end{align*}
\textit{\uline{Existence}:}
We consider the problem, for $\rho\in [0,1]$: Find $W_h\in\dsml$ such that
  \begin{equation}\label{eq:pb.Gmap}
  \Psi(W_h,\rho)=0.
  \end{equation}
For $\rho=1$, a solution to this problem is a solution to the scheme \eqref{eq:scheme}. We will prove that such a solution exists by using the topological degree argument encapsulated in \cite[Lemma~9.6]{hho-book}. We need for this check the three conditions in this lemma. 

(i) \emph{Continuity of the map}: Since $\dsml$ is a finite-dimensional space, it is trivial to check that
the bilinear form $\mathbb{A}_{h}$, the function $\max\{0,\cdot\}$, the contact operator $\mathbb{C}_{h}$, and the relevant scalar products are continuous; hence $\Psi$ is continuous. 

(ii) \emph{Bound on a priori solutions}: Let $\rho\in[0,1]$ and suppose that $W_h$ solves \eqref{eq:pb.Gmap}. We need to prove that
\begin{equation}\label{eq:bound.td}
\NORM{\dsml}{W_{h}} \leq  C_{\mathrm{data}},
\end{equation}
 with $C_{\mathrm{data}}$ independent of $\rho$. 

To prove \eqref{eq:bound.td}, we test \eqref{eq:pb.Gmap}
with  $V_{h}=(\uS{w}-\alpha\RhS, \bo{r}_h)$~($\RhS$ is from Lemma~\ref{rem:lifting} with $\bos{\mu}_h=-\gamma_{h}\bo{r}_h$)  and then rewrite the resultant equation by splitting by $\rho$ and $1-\rho$ as
\begin{multline}\label{eq:rho.spliiting}
 \rho\mathbb{A}_h((\uS{w},\bo{r}_h);(\uS{w}-\alpha\RhS,\bo{r}_h)) 
\\ 
\quad +(1-\rho)\Bigg(\int_\Omega\bbsigma_h(\uS{w}):\bbeps_h(\uS{w}-\alpha\RhS)+ \stabc{\uS{w}}{\uS{w}-\alpha\RhS}+\NORM{L^2(\Gamma)}{\gamma_{h}^{\frac{1}{2}} \bo{r}_{h}}^{2}\Bigg) \\
=\sum_{K\in\cM}\int_K\mathbf{f}\cdot\bo{w}_K.  
\end{multline}
By Lemma~\ref{lem:stability_results}, the first part on the left-hand side of \eqref{eq:rho.spliiting} satisfies  
\begin{equation*} 
  C_\alpha\rho\NORM{\dsml}{W_{h}}^2
  \leq \rho\mathbb{A}_h((\uS{w},\bo{r}_h);(\uS{w}-\alpha\RhS,\bo{r}_h)),
\end{equation*}
where $C_{\a}:=\min\{1-\alpha C \gamma_{0}, 1- 2\a C_{\Gamma}^2, \frac{5\a}{8}\}$ and $\a$ has been chosen to ensure $C_\a>0.$
From \eqref{eq:sigma_Rh_uh}-\eqref{bilform_cont}, we have
\begin{equation*}
 C_\a\NORM{\dsml}{W_{h}}^2
  \leq 
    \int_\Omega\bbsigma_h(\uS{w}):\bbeps_h(\uS{w}-\alpha\RhS)
    + \stabc{\uS{w}}{\uS{w}-\alpha\RhS}
    + \NORM{L^2(\Gamma)}{\gamma^{\frac{1}{2}}_{h}\bo{r}_h}^2.
\end{equation*}
Use these two above inequalities in \eqref{eq:rho.spliiting} and then the Cauchy--Schwarz inequality leads to
\begin{align*}
C_\a\NORM{\dsml}{W_{h}}^2
&\leq \left|\sum_{K\in\cM}\int_K\mathbf{f}\cdot \bo{w}_K\right|\leq \sum_{K\in \cM}\NORM{L^{2}(K)}{\mathbf{f}}  \NORM{L^{2}(K)}{\bo{w}_{K}} \nonumber\\& \leq C_P\NORM{L^2(\Omega)}{\mathbf{f}}\NORM{\mathrm{en},h}{\uS{w}} \leq C_P\NORM{L^2(\Omega)}{\mathbf{f}}\NORM{\dsml}{W_{h}},
\end{align*}
where the discrete Poincar\'e inequality~\cite[Theorem~4]{di2025uniform} is used in the second last step, and the conclusion follows from the definition of $\NORM{\dsml}{{\cdot}}$.
The above inequality leads to
\begin{equation*}
\NORM{\dsml}{W_{h}} \leq \frac{C_P}{C_\a}\NORM{L^2(\Omega)}{\mathbf{f}}=:C_{\mathrm{data}}.
\end{equation*}
This concludes the proof of \eqref{eq:bound.td}.

(iii) \emph{Homotopy to an invertible affine map}: The map $\Psi(\cdot, 0)$ is an affine function from $\dsml\to \dsml$.
 Its invertibility is ensured by the fact that, owing to \eqref{eq:bound.td}, its kernel is trivial.
Hence, the map $\Psi$ satisfies all three conditions of the topological degree lemma. Consequently, there exists $W_h \in \dsml$ such that $\Psi(W_h,1)=0$
which, as noted above, means that $W_h$ is a solution to \eqref{eq:scheme}.

\textit{\uline{Uniqueness}:} 
Let $(\uSe{u}{1},\laf_{1,h})$, $(\uSe{u}{2},\laf_{2,h})$ solve \eqref{eq:scheme}, and set $\uS{w}:=\uSe{u}{1}-\uSe{u}{2}$ and $\bo{r}_{h}:=\laf_{1,h}-\laf_{2,h}$.
Uniqueness is proved if we show that $(\uS{w}, \bo{r}_{h})=(\uS{0},\mathbf{0}_{h}).$ 

We start by proving $\uS{w}=\uS{0}$, which expresses the uniqueness of displacement. Taking the difference of discrete scheme~\eqref{eq:scheme} written first for $(\uS{u},\laf_{h})=(\uSe{u}{1},\laf_{1,h})$, then for $(\uS{u},\laf_{h})=(\uSe{u}{2},\laf_{2,h})$, we infer that
\begin{equation}\label{eq:diff.eqn}
\mathbb{A}_h((\uSe{u}{1},\laf_{1,h});(\uS{v},\bos{\mu}_h))-\mathbb{A}_h((\uSe{u}{2},\laf_{2,h});(\uS{v},\bos{\mu}_h))=0
  \quad\forall(\uS{v},\bos{\mu}_h)\in\ddgskS\times \Ms.
\end{equation}
Making $(\uS{v},\bos{\mu}_h)=(\uS{w},\bo{r}_{h})$ in \eqref{eq:diff.eqn} and using the definitions of $\mathbb{A}_{h}$ and $\mathbb{C}_{h}$ lead to
\begin{multline}\label{eq:E1}
  \NORM{\mathrm{en},h}{\uS{w}}^2 + \int_{\Gamma} \Big[\gamma^{-1}_h\bPg{\djumph{\uSe{u}{1}}}{\laf_{1,h}}
-\gamma^{-1}_h\bPg{\djumph{\uSe{u}{2}}}{\laf_{2,h}}\Big]\cdot\\\Big[\left(\djumph{\uSe{u}{1}}-\gamma_h \laf_{1,h}\right) -\left(\djumph{\uSe{u}{2}}-\gamma_h\laf_{2,h}\right)\Big]+ \NORM{L^2(\Gamma)}{\gamma_h^{\frac{1}{2}} \bo{r}_{h}}^2  = 0.
\end{multline}
Use Lemma~\ref{lem:mono.prop} together with analogous arguments to \eqref{eq:observation.2} in the last two terms on the left-hand side  of \eqref{eq:E1} to get
$$\NORM{\mathrm{en},h}{\uS{w}}^2 
    +\NORM{L^2(\Gamma)}{
      \gamma^{\frac{1}{2}}_{h}\bo{r}_{h}
      - \big[\gamma^{-\frac{1}{2}}_{h}\bPg{\djumph{\uSe{u}{1}}}{\laf_{1,h}}
        -\gamma^{-\frac{1}{2}}_{h}\bPg{\djumph{\uSe{u}{2}}}{\laf_{2,h}}
    \big]}^2   
    \leq  0.$$
This leads to $\uS{w}=\uS{0}$, hence $\uSe{u}{1}=\uSe{u}{2},$ and
\begin{equation}
\gamma_{h} \bo{r}_{h} = \bPg{\djumph{\uSe{u}{1}}}{\laf_{1,h}} - \bPg{\djumph{\uSe{u}{2}}}{\laf_{2,h}}.\label{eq:E2.q}
\end{equation}
We next prove the uniqueness of the discrete Lagrange multiplier, i.e., $\bo{r}_{h}=\bo{0}_h.$ To this end, we test the difference equation~\eqref{eq:diff.eqn} with $(\uS{v},\bos{\mu}_h)=(\RhS,\bo{0}_h)$,
where $\RhS$ is chosen from Lemma~\ref{rem:lifting} with $\bos{\mu}_{h}= \gamma_{h} \bo{r}_{h}.$ 
This yields
\begin{equation*}
  \mathbb{A}_h((\uS{w},\bo{r}_{h});(\RhS,\bo{0}_h))=0.
\end{equation*}
The above equation, the definitions of $\mathbb{A}_{h}$ and $\mathbb{C}_{h}$, $\uS{w}=\uS{0}$, and definition of $\Pi_{\cF_{\Gamma}}^{\ell}$ lead to
\begin{align*}
& \int_{\Gamma}\Big[\gamma^{-1}_h\bPg{\djumph{\uSe{u}{1}}}{\laf_{1,h}}
    - \gamma^{-1}_h \bPg{\djumph{\uSe{u}{2}}}{\laf_{2,h}}\Big]\cdot \djumph{\RhS} \nonumber\\ \overset{\eqref{eq:E2.q}}&=\int_{\Gamma} \bo{r}_{h}\cdot \Pi_{\cF_{\Gamma}}^{\ell}\jump{\RhS}
\overset{\eqref{eq:prop.Rh.jump}}= \NORM{L^2(\Gamma)}{\gamma_{h}^{\frac{1}{2}}\bo{r}_{h}}^2=0.
\end{align*}
Therefore, $\bo{r}_{h}=\mathbf{0}_{h}$, and hence $\laf_{1,h}=\laf_{2,h}.$ This proves the uniqueness of the solution to the discrete scheme \eqref{eq:scheme}, and this concludes the proof of the proposition. \qed

\section{Proof of the error estimates}\label{sec:proof_error}

\subsection{Proof of Theorem~\ref{thm:abs.err.bound}}

The proof of the abstract error bound is divided into five steps. Let us define the errors
$$ \uS{e} := \IhS\bo{u}-\uS{u},
  \qquad \bos{\eta}_h :=\Pi_{\cF_{\Gamma}}^{\ell}\laf-\laf_h.$$

\textit{\uline{Step~1}:}~(Error equation)
Since $\bu$ is a weak solution of the contact problem~\eqref{Form1_nonlinearform} and by regularity assumed on $\laf$ in the theorem, we have $-\div(\bbsigma(\bu))=\mathbf{f}\in \boldsymbol{L}^2(\Omega)$ and $\laf = -\gamma_\normal^+\bbsigma(\bu)\in \boldsymbol{L}^2(\Gamma)$. Hence, $\bbsigma(\bu)\in\Sigmaspace$ (see \eqref{sigma.space}) and, using the definition \eqref{eq:adjoint.consist} of $w_h$ with $\uS{v}=\uS{e}$, we get
\begin{equation}\label{cons.dual_equ_1}
\int_{\Omega} \bbsigma(\bu): \bbeps_{h}(\uS{e})  =  -w_{h}(\bbsigma(\bu),\uS{e}) - \int_{\Gamma} \laf \cdot \jump{\uS{e}}_{h} + \sum_{K\in \cells} \int_{K}  \mathbf{f} \cdot \bo{e}_{K}.
\end{equation}
Subtracting the scheme~\eqref{eq:scheme} (with $(\uS{v},\bos{\mu}_h)=(\uS{e},\bos{\eta}_h)$) from \eqref{cons.dual_equ_1}, we obtain
\begin{align}\label{eq:thm41.cons.dual.somme}
\int_{\Omega} \left(\bbsigma(\bu)-\bbsigma_h(\uS{u}) \right): \bbeps_{h}(\uS{e})& - \stabc{\uS{u}}{\uS{e}} =  -w_{h}(\bbsigma(\bu),\uS{e}) \notag \\&\qquad - \int_{\Gamma} \laf \cdot \jump{\uS{e}}_{h}  +\mathbb{C}_h((\uS{u},\laf_h);(\uS{e},\bos{\eta}_h)).
\end{align}
The definition of $\NORM{\mathrm{en},h}{{\cdot}}$~(norm induced from the inner product \eqref{eq:innerproduct}) along with $\zeta(E_K,\nu_K)\ge\mu_1>0$ for every $K\in\mathcal{M}$ leads to
\begin{align}
\NORM{\mathrm{en,h}}{\uS{e}}^2 &\leq \int_{\Omega} \bbsigma_h(\uS{e}): \bbeps_{h}(\uS{e}) + \stabc{\uS{e}}{\uS{e}} \notag\\
   \overset{\pm\bbsigma(\bu)}&= -\int_{\Omega} \left(\bbsigma(\bu)-\bbsigma_h(\IhS\bu)\right): \bbeps_{h}(\uS{e}) +\stabc{\IhS \bu}{\uS{e}}  \notag\\
   & \qquad + \int_{\Omega} \left(\bbsigma(\bu)-\bbsigma_h(\uS{u}) \right): \bbeps_{h}(\uS{e})   - \stabc{\uS{u}}{\uS{e}} \notag\\
   \overset{\eqref{eq:thm41.cons.dual.somme}}&= -\int_{\Omega} \left(\bbsigma(\bu)-\bbsigma_h(\IhS\bu)\right): \bbeps_{h}(\uS{e}) +\stabc{\IhS \bu}{\uS{e}} \notag\\
   & \qquad -w_{h}(\bbsigma(\bu),\uS{e}) - \int_{\Gamma} \laf \cdot \jump{\uS{e}}_{h}  +\mathbb{C}_h((\uS{u},\laf_h);(\uS{e},\bos{\eta}_h)).\label{eq:thm41:error.equation.intermidate}
\end{align}
Wenow simplify the last two terms of~\eqref{eq:thm41:error.equation.intermidate}. 
The definition~\eqref{form1_ch} of $\mathbb{C}_{h},$ the fact that $\laf=\gamma^{-1}_h\bPg{\jump{\bu}}{\laf},$ and elementary manipulations lead to
 \begin{align}\label{eq:thm41.last.two.terms}
 -\int_{\Gamma} \laf \cdot \jump{\uS{e}}_{h} &+\mathbb{C}_h((\uS{u},\laf_h);(\uS{e},\bos{\eta}_h))\notag\\
 &=
\int_{\Gamma}\Big[\gamma^{-1}_h\bPg{\djumph{\uS{u}}}{\laf_h}-\gamma^{-1}_h\bPg{\jump{\bu}}{\laf}\Big]
 \cdot \Big[\djumph{\uS{e}}-\gamma_h \bos{\eta}_{h}\Big]\notag\\
 &\qquad + \int_{\Gamma} \gamma_h(\laf_{h}-\laf) \cdot \bos{\eta}_{h} +\mathbb{S}_{\Gamma}(\laf_{h},\bos{\eta}_{h}).
\end{align}
Substitute \eqref{eq:thm41.last.two.terms} into \eqref{eq:thm41:error.equation.intermidate} to get
\begin{align}\label{eq:error.equation}
 \NORM{\mathrm{en,h}}{\uS{e}}^2 &\leq -\underbrace{\int_{\Omega} \left(\bbsigma(\bu)-\bbsigma_h(\IhS\bu)\right): \bbeps_{h}(\uS{e}) +\stabc{\IhS \bu}{\uS{e}} - w_{h}(\bbsigma(\bu),\uS{e})}_{I_{1}} 
 \notag\\& \qquad + \underbrace{ \int_{\Gamma}\Big[\gamma^{-1}_h\bPg{\djumph{\uS{u}}}{\laf_h}-\gamma^{-1}_h\bPg{\jump{\bu}}{\laf}\Big]
 \cdot \Big[\djumph{\uS{e}}-\gamma_h \bos{\eta}_{h}\Big]}_{I_{2}}\notag\\&  \qquad\quad + \underbrace{\int_{\Gamma} \gamma_h (\laf_{h}-\laf) \cdot \bos{\eta}_{h}}_{I_{3}} + \underbrace{\mathbb{S}_{\Gamma}(\laf_{h},\bos{\eta}_{h})}_{I_{4}}. 
\end{align}

\textit{\uline{Step~2}:}~(Estimate for $I_{1}$)
Applying the Cauchy--Schwarz inequalities and invoking the definitions \eqref{eq:primal.consist} of $\bos{\fC}_{h}$ and \eqref{eq:adjoint.consist} of $\bos{\cW}_h,$ we infer that
  \begin{align*}
I_{1}&\overset{\eqref{eq:norm_energy}}\lesssim\Big(\bos{\fC}_h(\bu,\IhS\bu)+\bos{\cW}_h(\bbsigma(\bu))\Big)\NORM{\mathrm{en},h}{\uS{e}}.
  \end{align*}
\textit{\uline{Step~3}:}~(Estimate for $I_{2}$)
Considering $I_{2},$ we can write 
\begin{align*}
I_{2}&= \underbrace{\int_{\Gamma}\Big[\gamma^{-1}_h\bPg{\djumph{\uS{u}}}{\laf_h}-\gamma^{-1}_h\bPg{\jump{\bu}}{\laf}\Big]\cdot
 \Big[\djumph{\IhS\bu}-\jump{\bu}+ \gamma_h\left(\Pi^{\ell}_{\cF_{\Gamma}}\laf- \laf\right)\Big]}_{I_{2\rm{a}}}\\&\quad+ \underbrace{\int_{\Gamma}\Big[\gamma^{-1}_h\bPg{\djumph{\uS{u}}}{\laf_h}-\gamma^{-1}_h\bPg{\jump{\bu}}{\laf}\Big]\cdot
 \Big[\left(\jump{\bu}+\gamma_h \laf\right) - \left(\djumph{\uS{u}}+\gamma_h \laf_h\right)\Big]}_{I_{2\rm{b}}} 
 \\&\qquad +\underbrace{\int_{\Gamma}\Big[\gamma^{-1}_h \bPg{\djumph{\uS{u}}}{\laf_h}-\gamma^{-1}_h\bPg{\jump{\bu}}{\laf}\Big]\cdot
 \Big[-2\gamma_h \bos{\eta}_{h}\Big]}_{I_{2\rm{c}}}  
\end{align*}
 Applying the Cauchy--Schwarz inequality and then the triangle inequality, we estimate $I_{2\rm{a}}$ as
\begin{multline*}
I_{2\rm{a}}\overset{\eqref{eq:tang.lagrange.multiplier}} \leq \NORM{L^{2}(\Gamma)}
  {\gamma^{-\frac{1}{2}}_h\bPg{\djumph{\uS{u}}}{\laf_h}-\gamma^{\frac{1}{2}}_h\laf} \Big(\gamma^{-\frac{1}{2}}_h\NORM{L^2(\Gamma)}{\jump{\bo{u}}
    -\djumph{\IhS\bo{u}}} +\gamma^{\frac{1}{2}}_h\NORM{L^2(\Gamma)}{\laf-\Pi_{\cF_{\Gamma}}^{\ell}\laf}\Big).
\end{multline*}
Next, use Lemma~\ref{lem:mono.prop}   to get the estimate for $I_{2\rm{b}}$
\begin{equation*}
I_{2\rm{b}} \leq -\NORM{L^{2}(\Gamma)}
  {\gamma^{-\frac{1}{2}}_h\bPg{\djumph{\uS{u}}}{\laf_h}-\gamma^{\frac{1}{2}}_h\laf}^{2}.
\end{equation*}
Finally, to estimate $I_{2\rm{c}},$ we need some preparations. To do so, test the discrete scheme~\eqref{eq:scheme} with
$(\uS{v},\bos{\mu}_h)=(\RhS,\bo{0}_h)$, where the definition of $\RhS$ follows from Lemma~\ref{rem:lifting} with  $\bos{\mu}_h=-2\gamma_h \bos{\eta}_h$ and use the definition~\eqref{form1_ch} of $\mathbb{C}_{h}$ together with $\bo{R}_{K}=\mathbf{0}_{K}$ for all $K\in\cells$ to get
\begin{equation}
  \int_\Omega\bbsigma_h(\uS{u}):\bbeps_h(\RhS)
  + \stabc{\uS{u}}{\RhS}
   + \int_{\Gamma}
      \gamma^{-1}_h\bPg{\djumph{\uS{u}}}{\laf_h}\cdot
      \djumph{\RhS} = 0. \label{eq:scheme.Rh2}
\end{equation}
Since $\bbtau:\bbeps_{h}(\RhS)=\bbtau:\GS{K}(\RhS)$ for any symmetric matrix $\bbtau$, the definitions of the orthogonal projector $\Pi_{K}^{k}$ and of $\GS{K}$ (see \eqref{eq:G1K_enhance}), together with $\bo{R}_{K}=\mathbf{0}_{K}$,
lead to
\begin{align}
  \int_{\Omega}\bbsigma(\bo{u}):\bbeps_h(\RhS)
 =& \sum_{K\in \cells}\int_{K}\Pi^k_K\bbsigma(\bo{u}):\GS{K}\RhS \nonumber\\
 =& \sum_{K\in \cells}\sum_{\sigma\in\cF_K}\int_\sigma \UpS{\Ksig}\RhS\cdot\Pi^k_K\bbsigma(\bo{u})\normal_{\Ksig}.\nonumber\\
  =&\sum_{K\in \cells}\sum_{\sigma\in\cF_K}\int_\sigma \UpS{\Ksig}\RhS\cdot \bbsigma(\bo{u})\normal_{\Ksig}\nonumber\\
  & \qquad -\sum_{K\in \cells}\sum_{\sigma\in\cF_K}\int_\sigma \UpS{\Ksig}\RhS\cdot\left(\bbsigma(\bo{u})-{\Pi}^k_K\bbsigma(\bo{u})\right)\normal_{\Ksig}.
\label{eq:sigma.Rh2}
\end{align}
For all $\sigma\in \cF_\Gamma$, we denote by $K$ and $L$ the cells that contain $\sigma$ and are, respectively, on the positive and negative side of $\Gamma$. 
Since $\UpS{\Ksig}\uS{R}=\UpS{\Lsig}\uS{R}$ for all $\sigma\in \cF^\inte\setminus\cF_\Gamma$, $\UpS{\Ksig}\uS{R}=\mathbf{0}$ for all $\sigma\in\cF_K\cap\cF^\ext,$ and the normal traces of $\bbsigma(\bu)$ are continuous across $\Gamma$, we have
\begin{align}
&\sum_{K\in \cells}\sum_{\sigma\in\cF_K}\int_\sigma  \UpS{\Ksig}\RhS\cdot \bbsigma(\bo{u})\normal_{\Ksig} \nonumber\\
&= \sum_{K\in \cells}\sum_{\sigma\in \cF_K \cap \cF_{\Gamma}}\int_\sigma  \UpS{\Ksig}\RhS\cdot \bbsigma(\bo{u})\normal_{\Ksig} +\underbrace{\sum_{K\in \cells}\sum_{\sigma\in\cF_K \setminus\cF_\Gamma}\int_\sigma  \UpS{\Ksig}\RhS\cdot \bbsigma(\bo{u})\normal_{\Ksig}}_{0}\nonumber\\
&= -\sum_{\sigma\in\cF_\Gamma}\int_\sigma
   \left(\UpS{\Ksig}\uS{R}-\UpS{\Lsig}\uS{R}\right)\cdot \laf \nonumber\\& = -\sum_{\sigma\in\cF_\Gamma}\int_\sigma
    \laf \cdot \Jsig{\RhS}=-\int_\Gamma
    \laf \cdot \djumph{\RhS},
\label{traction_sum_R_h}
\end{align}
where $\gamma_{\normal}^{+}\bbsigma(\bu)=-\laf$  and we have used $\UpS{\Lsig}\uS{R}=0$ (since $L$ is on the negative side of $\sigma$) in the second last step. Plug \eqref{traction_sum_R_h} into \eqref{eq:sigma.Rh2} together with the fact $\UpS{\Ksig}\uS{R}=\bo{0}$ for all $\sigma\notin\faces_{\Gamma,K}^{+}$ to get
\begin{equation}\label{eq:D1.Rh2}
  \int_\Omega\bbsigma(\bo{u}):\bbeps_h(\RhS)
  + \int_\Gamma
    \laf \cdot \djumph{\RhS}
 = -\sum_{K\in \cells}\sum_{\sigma\in\cF_{\Gamma,K}^+}\int_\sigma \UpS{\Ksig}\RhS\cdot\left(\bbsigma(\bo{u})-{\Pi}^k_K\bbsigma(\bo{u})\right)\normal_{\Ksig},
\end{equation}

Subtracting \eqref{eq:D1.Rh2} from \eqref{eq:scheme.Rh2} and rearranging the resulting terms, we obtain
\begin{multline*}
\int_{\Gamma}\Big[\gamma^{-1}_h\bPg{\djumph{\uS{u}}}{\laf_h}-\laf\Big]\cdot\djumph{\RhS} = -\int_\Omega\left(\bbsigma_h(\uS{u})-\bbsigma(\bo{u})\right):\bbeps_h(\RhS)-\stabc{\uS{u}}{\RhS}\\
+\sum_{K\in \cells}\sum_{\sigma\in\cF_K}\int_\sigma \UpS{\Ksig}\RhS\cdot\left(\bbsigma(\bo{u})-{\Pi}^k_K\bbsigma(\bo{u})\right)\normal_{\Ksig}.
\end{multline*}
Adding and subtracting $\Pi_{\cF_{\Gamma}}^{\ell}\djumph{\RhS}\overset{\eqref{eq:prop.Rh.jump}}=-2\gamma_h\boldsymbol{\eta}_h$ in the first term  on the left-hand side of the above equation, and writing $\uS{u}=\IhS\bo{u}-\uS{e}$ in the first two terms on the right-hand side, followed by a rearrangement of the terms, yield
\begin{align*}
&I_{2\rm{c}} = \underbrace{\int_\Omega(\bbsigma(\bo{u})-\bbsigma_h(\IhS\bo{u})):\bbeps_h(\RhS)-\stabc{\IhS\bo{u}}{\RhS} + \int_\Omega \bbsigma_h(\uS{e}):\bbeps_h(\RhS) +\stabc{\uS{e}}{\RhS}}_{I_{3\rm{c}}}\\ &\qquad\quad + \underbrace{\int_{\Gamma}\Big[\gamma^{-1}_h \bPg{\djumph{\uS{u}}}{\laf_h}-\laf\Big]\cdot\Big[\Pi_{\cF_{\Gamma}}^{\ell}\djumph{\RhS}-\djumph{\RhS}\Big]}_{I_{4\rm{c}}}
    \\ &\qquad \qquad
    +\underbrace{\sum_{K\in \cells}\sum_{\sigma\in\cF_K}\int_\sigma \UpS{\Ksig}\RhS\cdot\left(\bbsigma(\bo{u})-{\Pi}^k_K\bbsigma(\bo{u})\right)\normal_{\Ksig}}_{I_{5\rm{c}}}.
\end{align*}
Now, we will estimate each term on the right-hand side of the above equation. Using Cauchy--Schwarz inequalities and recalling the definition \eqref{eq:primal.consist} of $\bos{\fC}_{h}$,
  \begin{align*}
I_{3\rm{c}}
&\lesssim\bos{\fC}_h(\bu,\IhS\bu)\NORM{\mathrm{en},h}{\uS{R}} + 
\NORM{\mathrm{en},h}{\uS{e}}\NORM{\mathrm{en},h}{\RhS}\\
\overset{\eqref{eq:norm_energy},\eqref{eq:prop.Rh.norm.energy}}
&\lesssim C_{\mathrm{R}}\gamma_{0}^{\frac{1}{2}}\bos{\fC}_h(\bu,\IhS\bu)\NORM{L^2(\Gamma)}{\gamma^{\frac{1}{2}}_h\bos{\eta}_{h}} + \underbrace{C_{\mathrm{R}}\gamma_{0}^{\frac{1}{2}}\NORM{\mathrm{en},h}{\uS{e}}\NORM{L^2(\Gamma)}{\gamma^{\frac{1}{2}}_h\bos{\eta}_{h}}}_{I_{3\rm{c}}^{\prime}}.
\end{align*}
To estimate ${I_{4\rm{c}}},$ we first use the the definitions 
\eqref{eq:jump.n} of $\djumph{\cdot}$, \eqref{eq:Rh.construction} of $\RhS,$ and of $\Pi_{\cF_{\Gamma}}^{\ell}$ to write
\begin{equation}\label{eq:jump.approx}
 \Pi_{\sigma}^{\ell}\djumph{\RhS}-\djumph{\RhS}=\Pi_{\sigma}^{\ell}\Upsilon_{\Ksig}^{k+1}\uvec{R}_{h}-\Upsilon_{\Ksig}^{k+1}\uvec{R}_{h}\quad \forall \sigma \in \faces_{\Gamma,K}^+.
 \end{equation}
This identity, together with the Cauchy--Schwarz inequality and \eqref{eq:prop.Rh.ups.bound}, yields
\begin{align*}
{I_{4\rm{c}}}&\leq 2C_{\mathrm{D}}\left(\NORM{L^{2}(\Gamma)}{\gamma^{-\frac{1}{2}}_h \bPg{\djumph{\uS{u}}}{\laf_h}-\gamma^{\frac{1}{2}}_h\laf}\right)  \left(\sum_{e\in \edges_{\Gamma}^{\text{int}}}{\gamma_{h}}{} h_{e} \NORM{L^{2}(e)} {\jump{\bos{\eta}_{h}}_{e}}^2\right)^{\frac{1}{2}}\\
\overset{\eqref{def.S.Gamma}}&\leq \left(\NORM{L^{2}(\Gamma)}{\gamma^{-\frac{1}{2}}_h\bPg{\djumph{\uS{u}}}{\laf_h}-\gamma^{\frac{1}{2}}_h\laf}\right)  \left(2 C_{\mathrm{D}}\delta^{-\frac{1}{2}}\mathbb{S}_{\Gamma}(\bos{\eta}_h,\bos{\eta}_h)^{\frac{1}{2}}\right)=:I_{4\rm{c}}^{\prime}.
\end{align*}
 The Cauchy--Schwarz inequality and an argument analogous to~\eqref{eq:up.Rh.dof.bound} leads to
\begin{equation*}
 I_{5\rm{c}}\lesssim  \left(\sum_{K\in \cells}\sum_{\sigma\in\cF_{\Gamma, K}^+}\NORM{L^{2}(\sigma)}{\gamma^{\frac{1}{2}}_h\big(\bbsigma(\bo{u})-\Pi_{K}^{k}\bbsigma(\bo{u})\big)}^2\right)^{\frac{1}{2}}\NORM{L^{2}(\Gamma)}{\gamma^{\frac{1}{2}}_h\bos{\eta}_{h}}.
\end{equation*}

\textit{\uline{Step~4}:}~(Estimates for $I_{3}$ and $I_{4}$) For $I_{3}$, the definition of $\Pi_{\cF_{\Gamma}}^{\ell}$ gives
\begin{equation*}
I_{3}=-\NORM{L^{2}(\Gamma)}{\gamma^{\frac{1}{2}}_h\bos{\eta}_{h}}^2.
\end{equation*}
For $I_{4}$, adding and subtracting $\Pi_{\cF_{\Gamma}}^{\ell}\laf_h$ and applying the Cauchy--Schwarz inequality yield
\begin{equation*}
 I_{4} \le  -\mathbb{S}_{\Gamma}(\bos{\eta}_h,\bos{\eta}_h)+ \mathbb{S}_{\Gamma}(\Pi_{\cF_{\Gamma}}^{\ell}\laf_h,\Pi_{\cF_{\Gamma}}^{\ell}\laf_h)^{\frac{1}{2}}\mathbb{S}_{\Gamma}(\bos{\eta}_h,\bos{\eta}_h)^{\frac{1}{2}}.   
\end{equation*}

\textit{\uline{Step~5}:}~(abstract bound)
Substituting the bounds from Steps 3--5 into the right-hand side of \eqref{eq:error.equation}, and moving to the left-hand side the negative contributions of $I_{2\rm{b}},$ $I_3$, the first part of $I_4$, and rearranging the terms lead to
\begin{multline*} 
 \NORM{\mathrm{en},h}{\uS{e}}^2 +  \NORM{L^2(\Gamma)}{
\gamma^{-\frac{1}{2}}_h\bPg{\djumph{\uS{u}}}{\laf_h}-\gamma^{\frac{1}{2}}_h\laf
   }^2+\NORM{L^2(\Gamma)}{\gamma^{\frac{1}{2}}_h \bos{\eta}_h}^2 + \mathbb{S}_{\Gamma}(\bos{\eta}_h,\bos{\eta}_h) \\ \leq C \bigg[ \bos{\fC}_h(\bu,\IhS\bu)\NORM{\mathrm{en},h}{\uS{e}} +\bos{\cW}_h(\bbsigma(\bu))\NORM{\mathrm{en},h}{\uS{e}}  + C_{\mathrm{R}}\gamma_{0}^{\frac{1}{2}}\bos{\fC}_h(\bu,\IhS\bu)\NORM{L^2(\Gamma)}{\gamma^{\frac{1}{2}}_h\bos{\eta}_{h}} + I_{3\rm{c}}^{\prime}
   \\
+ \NORM{L^{2}(\Gamma)}
  {\gamma^{-\frac{1}{2}}_h\bPg{\djumph{\uS{u}}}{\laf_h}-\gamma^{\frac{1}{2}}_h\laf} \left(\NORM{L^2(\Gamma)}{\gamma^{-\frac{1}{2}}_h\big(\jump{\bo{u}}
    -\djumph{\IhS\bo{u}}\big)} +\NORM{L^2(\Gamma)}{\gamma^{\frac{1}{2}}_h\big(\laf-\Pi_{\cF_{\Gamma}}^{\ell}\laf\big)}\right)
   + I_{4\rm{c}}^{\prime}\\
\\+ \left(\sum_{K\in \cells}\sum_{\sigma\in\cF_{\Gamma, K}^+}\NORM{L^{2}(\sigma)}{\gamma^{\frac{1}{2}}_{\sigma}\big(\bbsigma(\bo{u})-\Pi_{K}^{k}\bbsigma(\bo{u})\big)}^2\right)^{\frac{1}{2}}\NORM{L^{2}(\Gamma)}{\gamma^{\frac{1}{2}}_h\bos{\eta}_{h}}+ \mathbb{S}_{\Gamma}(\Pi_{\cF_{\Gamma}}^{\ell}\laf,\Pi_{\cF_{\Gamma}}^{\ell}\laf)^{\frac{1}{2}}\mathbb{S}_{\Gamma}(\bos{\eta}_h,\bos{\eta}_h)^{\frac{1}{2}}\bigg]. 
\end{multline*}
We conclude by applying Young's inequality to the bilinear terms on the right-hand side. Each term involving a consistency or approximation quantity consists of the product of a controlled quantity with a single error contribution and can therefore be bounded by Young's inequality with an arbitrary fixed parameter. The resulting squared error terms are then absorbed into the left-hand side.

The terms $I_{3\rm{c}}^{\prime}$ and $I_{4\rm{c}}^{\prime}$ involve products of two error quantities and consequently require a parameter-dependent splitting. Applying Young's inequality yields
$$
  C I_{3\rm{c}}^{\prime}\leq C^2 C_{\mathrm{R}}^{2}\gamma_{0}\NORM{\mathrm{en},h}{\uS{e}}^{2}
  + \frac{1}{4}\NORM{L^2(\Gamma)}{\gamma^{\frac{1}{2}}_h\bos{\eta}_{h}}^2,
$$
and 
$$
  C I_{4\rm{c}}^{\prime} \leq \frac{1}{4}\NORM{L^{2}(\Gamma)}
  {\gamma^{-\frac{1}{2}}_h\bPg{\djumph{\uS{u}}}{\laf_h}-\gamma^{\frac{1}{2}}_h\laf}^2
  + 4 C^2 C^2_{\mathrm{D}}\delta^{-1}\mathbb{S}_{\Gamma}(\bos{\eta}_h,\bos{\eta}_h).
$$
Choosing $\gamma_{0}$ sufficiently small and $\delta$ sufficiently large ensures that the coefficients
$C^2 C_{\mathrm{R}}^{2}\gamma_{0}$ and $4 C^2 C_{\mathrm{D}}^{2}\delta^{-1}$ are smaller than the corresponding unit coefficients of $\NORM{\mathrm{en},h}{\uS{e}}^{2}$ and $\mathbb{S}_{\Gamma}(\bos{\eta}_h,\bos{\eta}_h)$ on the left-hand side. Consequently, these contributions, together with the associated $\frac14$-weighted terms, can be absorbed into the left-hand side. Collecting the remaining estimates, retaining only the consistency and approximation contributions on the right-hand side, and finally taking square roots concludes the proof of the theorem.

\subsection{Proof of Theorem \ref{thm:error.estimate}}\label{subsec:error.bound}

To prove the error bounds in terms of powers of $h$, we first establish a few results.

\begin{lemma}[Adjoint consistency]\label{lm:adjoint.consis}
Recalling the definition $\eqref{eq:adjoint.consist}$ of the adjoint consistency error $\bos{\cW}_h$, it holds that, for all $\bbtau \in \bo{\Sigma} \cap H^{r}(\cM)^{3\times 3}$ with $\frac{1}{2}<r\leq k+1$, 
\begin{equation*}
  \bos{\cW}_h(\bbtau) \lesssim h^{r}|\bbtau|_{H^{r}(\cM)}.  
\end{equation*}
\end{lemma}
\begin{proof}
 Let $\uS{v}\in \ddgskS$ and  $\bbtau \in \Sigmaspace \cap H^r(\cells)^{3\times3}.$ Since $\bbtau$ is symmetric, we have $\bbtau:\bbeps_h(\uS{u})=\bbtau:\GS{K} \uS{u}$
and thus, by the definition of $\Pi^k_K$, 
\begin{align*}
  &\int_K \bbtau:\bbeps_h(\uS{v})
  = \int_K\Pi^k_K\bbtau:\GS{K}\uS{v}
 \nonumber\\&  \quad \overset{\eqref{eq:G1K_enhance}}= -\int_K\DIV(\Pi^k_K\bbtau)\cdot \bo{v}_K + \sum_{\sigma\in\cF_K}\int_\sigma \UpS{\Ksig}\uS{v}\cdot\Pi^k_K\bbtau\,\normal_{\Ksig}.
\end{align*} 
Use this in the definition \eqref{eq:adjoint.consist} of $w_{h}(\bbtau, \uS{v}).$ Then, applying the definition of $\Pi_{K}^{k}$ and rearranging the terms yields
\begin{align}
  w_{h}(\bbtau, \uS{v}) &= -\sum_{K\in\cM}\int_K\DIV(\bbtau-\Pi^k_K\bbtau)\cdot \bo{v}_K +\int_\Gamma \gamma^+_{\bo{n}}\bbtau\cdot \djumph{\uS{v}}
     - \sum_{K\in\cM}\sum_{\sigma\in\cF_K}\int_\sigma
       \UpS{\Ksig}\uS{v}\cdot\Pi^k_K\bbtau\normal_{\Ksig} 
     \notag\\
  \overset{\text{IBP}}&=\sum_{K\in\cM}\cancelto{0}{\int_K(\bbtau-\Pi^k_K\bbtau)\cdot \nabla \bo{v}_K} 
    +\int_\Gamma \gamma^+_{\bo{n}}\bbtau\cdot \djumph{\uS{v}}
    - \sum_{K\in\cM}\sum_{\sigma\in\cF_K}\int_\sigma
       \UpS{\Ksig}\uS{v}\cdot\bbtau\,\normal_{\Ksig} \nonumber\\
       & \qquad \quad + \sum_{K\in\cM}\sum_{\sigma\in\cF_K}\int_\sigma
       (\UpS{\Ksig}\uS{v}-\bo{v}_K)\cdot(\bbtau-\Pi^k_K\bbtau)\normal_{\Ksig},\label{eq:W1.expand}
        \end{align}
        where the cancellation after integrating-by-parts comes from the definition of $\Pi_K^k$. Similar arguments as those leading to \eqref{traction_sum_R_h} show that 
\begin{equation*}
\sum_{K\in\cM}\sum_{\sigma\in\cF_K}\int_\sigma
       \UpS{\Ksig}\uS{v}\cdot\bbtau\,\normal_{\Ksig}
         =\int_{\Gamma} \gamma_{\normal}^{+}\bbtau\cdot  \djumph{\uS{v}}.
\end{equation*}
An application of the above equation in \eqref{eq:W1.expand} leads to
\begin{align}\label{eq:wh.simplify.eq}
w_{h}(\bbtau, \uS{v}) =\sum_{K\in\cM}\sum_{\sigma\in\cF_K}\int_\sigma
       (\UpS{\Ksig}\uS{v}-\bo{v}_K)\cdot(\bbtau-\Pi^k_K\bbtau)\normal_{\Ksig}.
\end{align}
The triangle inequality together with the definition definition~\eqref{eq:UpS.face} of $\UpS{\Ksig},$ \eqref{eq:stabK}, and \cite[Lemma~5.1]{DKMR25} lead to
\begin{equation}\label{eq:UPS.vk.diff.bound}
h_{K}^{-1} \NORM{L^{2}(\sigma)}{\UpS{\Ksig}\uS{v}-\bo{v}_K}^{2} \leq h_{K}^{-1} \NORM{L^{2}(\sigma)}{\Upsilon_{\Ksig}^{k+1}\uvec{v}_h-\bo{v}_K}^{2} +\mathbb{S}_{K}(\uS{v},\uS{v}) \lesssim \NORM{1,K}{\uS{v}}^{2} \;\; \forall \sigma \in \cF_{K},
\end{equation}
Now, applying the Cauchy–Schwarz inequality, the fractional multiplicative trace inequality~\cite[Remark~2.6]{dong2022hybrid}, and the approximation properties of $\Pi_{K}^{k}$ -- obtained by interpolating the estimates in~\cite[Theorem~1.45]{hho-book} for integral Sobolev exponents -- on the right side of \eqref{eq:wh.simplify.eq}, we obtain
\begin{align*}
w_{h}(\bbtau,\uS{v})&\lesssim  \sum_{K\in\cM}\bigg(\sum_{\sigma\in\cF_K}\lVert \UpS{\Ksig}\uS{v} -\bo{v}_K\rVert_{L^2(\sigma)}\\
 &\qquad\quad\qquad\times \Big[h^{-\frac{1}{2}}_K\lVert \bbtau-{\Pi}^k_K\bbtau\rVert_{L^2(K)}
 +h^{\min(r-\frac{1}{2},\frac{1}{2})}_K\lvert \bc-{\Pi}^k_K\bc\rvert_{H^{\min(r,1)}(K)}\Big]\bigg)\\
&\lesssim\sum_{K\in\cM}\sum_{\sigma\in\cF_K}h^{-\frac{1}{2}}_K\lVert \UpS{\Ksig}\uS{v} - \bo{v}_K \rVert_{L^2(\sigma)}h^r_K \lvert\bc\rvert_{H^r(K)}\\ 
&\lesssim\left(\sum_{K\in\cM}h^{2r}_K \lvert\bc\rvert^2_{H^r(K)}\right)^{\frac{1}{2}}\left(\sum_{K\in\cM}\Big[\sum_{\sigma\in\cF_K}h^{-\frac{1}{2}}_K\lVert \UpS{\Ksig}\uS{v} - \bo{v}_K\rVert_{L^2(\sigma)}\Big]^2\right)^{\frac{1}{2}}\\
\overset{\eqref{eq:UPS.vk.diff.bound}}&\lesssim h^r\lvert\bc\rvert_{H^r(\cM)}\left(\sum_{K\in\cM}\lVert\uS{v}\rVert^2_{1,K}\right)^{\frac{1}{2}}.
\end{align*}
Finally, combining the above inequality with \eqref{eq:adjoint.consist} completes the proof of the lemma.
\end{proof}

\begin{lemma}[Consistency of jump reconstruction]\label{lem:consis.jump.recons} If $\bu\in \bo{U}_0\cap H^{k+2}(\cM)^{3}$ then
$$\NORM{L^2(\Gamma)}{\gamma^{-\frac{1}{2}}_h\big(\jump{\bo{u}} -\djumph{\IhS\bo{u}}\big)} \lesssim h^{k+1}\SEMINORM{H^{k+2}(\cM)}{\bu}.$$
\end{lemma}
\begin{proof}
Fix $\sigma \in \cF_{\sigma}$ and let $K$~(resp. $L$) denote, as usual, the cell on the positive~(resp. negative) side of $\sigma.$ Let $\bo{q}_K\in\Poly{k+1}(K)^3$~(resp.~$\bo{q}_L\in\Poly{k+1}(L)^3$) be the $L^2(K)$-(resp. $L^2(L)$)-orthogonal projection of $\bu_{|K}$~(resp.~$\bu_{|L}$), and let $\bo{q}$ be the~(generally discontinuous across $\sigma$) piecewise polynomial equal to $\bo{q}_{K}$ on $K$ and $\bo{q}_{L}$ on $L.$ By the approximation properties of the $L^2$-projector~\cite[Theorem~1.45]{hho-book},
  \begin{equation}\label{eq:u.q.semi.norm.H.bound}
  |\bu -\bo{q}|_{H^{m}(X)} \lesssim h_{K}^{k+2-m}|\bu|_{H^{k+2}(X)} \;\; \text{ for all } m\in\{0,\cdots,k+2\} \text{ and for any }  X\in \{K,L\}.\end{equation}
Now, an application of the Sobolev-embedding estimate~\cite[Eq. (5.110)]{hho-book} yield
\begin{align}
\NORM{L^\infty(\overline{X})}{\bu-\bq}&\lesssim
|X|^{-\frac{1}{2}}\sum_{m=0}^{k+2} h_{X}^{m}|\bu|_{H^{m}(X)}\overset{\eqref{eq:u.q.semi.norm.H.bound}}\lesssim |X|^{-\frac{1}{2}}h_{X}^{k+2}|\bu|_{H^{k+2}(X)}\notag\\& \lesssim h_{X}^{k+\frac{1}{2}}\SEMINORM{H^{k+2}(X)}{\bu} \quad \text{ for any } X\in \{K,L\}.\label{eq:approx.uq}
\end{align}
Applying~\eqref{eq:approx.uq} on each side of $\sigma$ gives
\begin{equation}\label{eq:approx.jump.uq}
\NORM{L^2(\sigma)}{\jump{\bu-\bq}} \leq |\sigma|^{\frac{1}{2}} \NORM{L^\infty(\sigma)}{\jump{\bu-\bq}}\lesssim |\sigma|^{\frac{1}{2}} \left(h_K^{k+\frac{1}{2}}\SEMINORM{H^{k+2}(K)}{\bu}+h_L^{k+\frac{1}{2}}\SEMINORM{H^{k+2}(L)}{\bu}\right).
\end{equation}
Next, set $\uS{w}=\IhS(\bu-\bq).$ The bound \eqref{eq:facePotdof.bound} on $\Upsilon^{k+1}_{\Ksig}\uvec{v}_h$ and the relation \eqref{eq:UpS.face} show that, for all $\sigma \in \mathcal F_{\Gamma,K}^{+},$
 \begin{equation}
\lVert \UpS{\Ksig}\uS{w}\rVert_{L^2(\sigma)} \lesssim \NORM{L^2(\sigma)}{\bo{w}_{\cK\sigma}} + \krond \NORM{L^2(\sigma)}{\bo{w}_{K\sigma}^{\bub}}+ h_\sigma^{\frac{1}{2}}\sum_{e\in\cE_\sigma}\NORM{L^2(e)}{\bo{w}_\Ke} + h_\sigma \sum_{s\in\mathcal V_\sigma}\lvert \bo{w}_\Ks\rvert. \label{eq:face.new.Potdof.bound}
\end{equation}
The definition~\eqref{eq:DOF.Ih} of $\IhS$ easily shows that the degree of freedom of $\uS{w}$ on $\sigma$ are bounded the maximum of $|\sigma|^{\frac{1}{2}}|\bu-\bo{q}|$ on $K\cup L$~(this can be proof in the similar fashion as in~\cite{DKMR25} with additional bubble enrichment). Hence, by the definition~\eqref{eq:jump.n} of discrete jump and \eqref{eq:face.new.Potdof.bound}, we get
\begin{equation}
   \lVert \djumph{\IhS(\bu-\bq)}\rVert_{L^2(\sigma)}= \lVert \djumph{\uS{w}}\rVert_{L^2(\sigma)}  \overset{\eqref{eq:approx.uq}}\lesssim |\sigma|^{\frac{1}{2}}\left(h_K^{k+\frac{1}{2}}\SEMINORM{H^{k+2}(K)}{\bu}+h_L^{k+\frac{1}{2}}\SEMINORM{H^{k+2}(L)}{\bu}\right).  \label{eq:approx.jump.IUDuq}
\end{equation}
On the other hand, the polynomial consistency~\eqref{eq:UpS.face_1a} applied on each side of $\sigma$ gives $\Jsig{\IhS\bo{q}}=\jump{\bo{q}}.$ Then, we can write 
\begin{align*}
\NORM{L^2(\sigma)}{\gamma^{-\frac{1}{2}}_h\big(\Jsig{\IhS\bu}-\jump{\bu}\big)}
& \le 
  \NORM{L^2(\sigma)}{\gamma^{-\frac{1}{2}}_{h}\big(\jump{\IhS(\bu-\bq)}_\sigma\big)}
  + \NORM{L^2(\sigma)}{\gamma^{-\frac{1}{2}}_h\big(\jump{\bu-\bq}\big)}\nonumber\\
  &  \lesssim 
 h_{\Gamma}^{-\frac{1}{2}} \left(\NORM{L^2(\sigma)}{\big(\jump{\IhS(\bu-\bq)}_\sigma\big)}
  + \NORM{L^2(\sigma)}{\big(\jump{\bu-\bq}\big)}\right)\nonumber\\
\overset{\eqref{eq:approx.jump.IUDuq},\eqref{eq:approx.jump.uq}} & \lesssim  h_{\Gamma}^{-\frac{1}{2}} |\sigma|^{\frac{1}{2}}\left(h_K^{k+\frac{1}{2}}\SEMINORM{H^{k+2}(K)}{\bu}+h_L^{k+\frac{1}{2}}\SEMINORM{H^{k+2}(L)}{\bu}\right),\\
& \lesssim  h^{k+1}_{\sigma}\left(\SEMINORM{H^{k+2}(K)}{\bu}+\SEMINORM{H^{k+2}(L)}{\bu}\right),
\end{align*}
where $\gamma_{h}=\gamma_{0}h_{\Gamma}$ is used in the second step. The last estimate follows from the quasi-uniformity of the mesh on $\Gamma$~\eqref{eq:quasi.uniform.fracture}, which gives $h_\Gamma\approx h_\sigma$, combined with the mesh regularity assumptions $|\sigma|^{\frac{1}{2}}\approx h_{\sigma}$, $h_{K}\approx h_{\sigma},$ and $h_{L}\approx h_{\sigma}$; together these give $h_{\Gamma}^{-\frac{1}{2}} |\sigma|^{\frac{1}{2}} \approx h_{\sigma}^{-\frac12}\,h_\sigma = h_{\sigma}^{\frac{1}{2}}\approx h_{K}^{\frac{1}{2}} \approx h_{L}^{\frac{1}{2}}.$ Now, Squaring, summing over $\sigma\in\faces_\Gamma$, and taking the square root concludes the proof.
\end{proof}

\begin{lemma}[Consistency of gradient reconstruction]\label{lm:primal.consis}
Let $\bu\in \bo{U}_0\cap H^{k+2}(\cM)^{3}$. Recalling the definition \eqref{eq:primal.consist} of $\bos{\fC}_{h}$, it holds  
\begin{align*}
\bos{\fC}_{h}(\bu,\IhS\bu) \lesssim  h^{k+1}|\bu|_{H^{k+2}(\cM)}.
\end{align*}
\end{lemma}
\begin{proof}
The estimate is obtained exactly as the analogous gradient–consistency result
for the (non-enriched) DDR reconstructions in \cite[Lemma~5.8]{DKMR25};
We only indicate why the face-bubble enrichment leaves the argument unchanged.

The regularity $\mathbf u\in \bo{U}_0\cap H^{k+2}(\mathcal M)^3$ guarantees
$\mathbf u\in \bo{C}^0_0(\Omega\setminus\Gamma)$, so that $\IhS\mathbf u$ is well
defined. On each $K\in\mathcal M$ one introduces the $L^2(K)$-orthogonal
projection $\mathbf q$ of $\mathbf u$ onto $\mathbb P^{k+1}(K)^3$ and writes
$\IhS\mathbf u=\IhS\mathbf q+\IhS(\mathbf u-\mathbf q)$ (also locally, cell by cell). The polynomial part is
reproduced exactly: by Lemma~\ref{lem:prop.reconstructions}, the enriched operators satisfy
$\GS{K}\IhS\mathbf q=\nabla\bq$ and $\UpS{K\sigma}\IhS\bq=\bq$,
the bubble interpolator $\IhS$ having been designed precisely
so that the enriched reconstructions retain the same polynomial consistency as
their standard counterparts. Hence, the bubble degrees of freedom contribute no
additional term on the polynomial argument, and the consistency error reduces to
the interpolation of the residual $\mathbf u-\mathbf q$. Controlling
$\IhS(\bu-\bq)$ through its degrees of freedom by the
$L^\infty$-approximation bound \cite[Eq.(5.110)]{hho-book} on $\mathbf u-\mathbf q$ and the boundedness of
the enriched reconstructions in the local norm $\NORM{1,K}{\cdot},$ and then
invoking the approximation properties of the $L^2$-projection~\cite[Theorem~1.45]{hho-book} on $\mathbb P^{k+1}(K)^3$, yields the local bound by $ h_K^{k+1}\,|\mathbf u|_{H^{k+2}(K)}$. Squaring, summing over
$K\in\mathcal M$, using $h_K\le h$ and taking the square root, concludes the lemma.
\end{proof}

\begin{proof}[\textit{\textbf {Proof of Theorem~\ref{thm:error.estimate}:}}] 
We estimate each of the terms appearing on the right-hand side of the abstract error bound stated in Theorem~\ref{thm:abs.err.bound}.

First, Lemma~\ref{lm:primal.consis} yields
$$\bos{\fC}_{h}(\bu,\IhS\bu) \lesssim h^{k+1}|\bu|_{H^{k+2}(\cM)}.
$$
Next, applying Lemma~\ref{lm:adjoint.consis}, we obtain
$$\bos{\cW}_h(\bbsigma(\bu))  \lesssim h^{k+1}|\bbsigma(\bu)|_{H^{k+1}(\cells)}.$$
For the third term, an application of Lemma~\ref{lem:consis.jump.recons} 
yields
$$\NORM{L^2(\Gamma)}{\gamma^{-\frac{1}{2}}_{h}\big(\jump{\bo{u}}
    -\djumph{\IhS\bo{u}}\big)}
    \lesssim h^{k+1}\SEMINORM{H^{k+2}(\cells)}{\bu}.
    $$  

For the fourth term, using the approximation properties of
$\Pi_{\cF_{\Gamma}}^{\ell}$ (see~\cite[Theorem~1.45]{hho-book}), $h_{\sigma } \lesssim h_{\Gamma}$ and $\gamma_h=\gamma_0 h_{\Gamma}$, we infer that
\begin{equation}
\NORM{L^2(\Gamma)}{\gamma^{\frac{1}{2}}_{h}\big(\laf-\Pi_{\cF_{\Gamma}}^{\ell}\laf\big)} \lesssim h^{\frakr}|\laf|_{H^{\frakr-\frac12}(\cF_\Gamma)},
\label{eq:bound.lambda.Piell}
\end{equation}
where 
$$
\frakr=
\begin{cases}
k+1, &\text{ if $\ell=k$ (active bubble)},\\[1mm]
k+\frac{1}{2}, & \text{if $\ell= k-1$ (inactive bubble)}.
\end{cases}
$$
To estimate the penultimate term, we combine the local trace inequality~\cite[Lemma~1.31]{hho-book}, the approximation properties of $\Pi_K^k$ \cite[Theorem~1.45]{hho-book},  $\gamma_h=\gamma_0h_{\Gamma}$ and $h_{\sigma}\approx h_{K}$ to obtain
\begin{align*}
&\left(\sum_{K\in \cells}\sum_{\sigma\in\cF_{\Gamma, K}^+}\NORM{L^{2}(\sigma)}{\gamma^{\frac{1}{2}}_{h} \big(\bbsigma(\bo{u})-\Pi_{K}^{k}\bbsigma(\bo{u})\big)}^2\right)^{\frac{1}{2}} \\&  \lesssim \left(\sum_{K\in \cells}\Big[\NORM{L^{2}(K)}{\bbsigma(\bo{u})-\Pi_{K}^{k}\bbsigma(\bo{u})}^2 +h^2_{K}\NORM{L^{2}(K)}{\nabla\big(\bbsigma(\bo{u})-\Pi_{K}^{k}\bbsigma(\bo{u})\big)}^2\Big]\right)^{\frac{1}{2}} \\& \lesssim h^{k+1}|\bbsigma(\bu)|_{H^{k+1}(\cells)}.
\end{align*}
 
Finally, since $\laf\in H^{\frac{1}{2}+\epsilon}(\Gamma)$, its jumps vanish across interface edges, namely
$\jump{\laf}_{e}=0$ for all $e\in\edges_{\Gamma}^{\mathrm{int}}$.
Using this property together with the trace inequality
\cite[Lemma~1.31]{hho-book} with $h_{e}\approx h_{\sigma}$ and the approximation property of the $L^2$-projector $\Pi_{\cF_{\Gamma}}^{\ell}$ \cite[Lemma~1.45]{hho-book}, we obtain 
\begin{align*}
 \mathbb{S}_{\Gamma}(\Pi_{\cF_{\Gamma}}^{\ell}\laf,\Pi^{\ell}_{\cF_{\Gamma}}\laf)^\frac{1}{2}&= {\left( \sum_{e\in \edges_{\Gamma}^{\text{int}}}\delta \gamma_{h}  \int_{e}  h_{e} \jump{\Pi_{\cF_{\Gamma}}^{\ell}\laf-\laf}_{e}^2\right)^{\frac{1}{2}}}\nonumber\\& \lesssim   \left(\sum_{\sigma \in \faces_{\Gamma}}\Big[\NORM{L^2(\sigma)}{\gamma^{\frac{1}{2}}_{h}\big(\Pi_{\cF_{\Gamma}}^{\ell}\laf-\laf \big)}^2 +h^2_{\sigma}\NORM{L^{2}(K)}{\nabla\big(\gamma^{\frac{1}{2}}_{h}\big(\Pi_{\cF_{\Gamma}}^{\ell}\laf-\laf\big)\big)}^2\Big]\right)^{\frac{1}{2}}\nonumber\\
 &\lesssim h^{\frakr}|\laf|_{H^{\frakr-\frac12}(\cF_\Gamma)}.
\end{align*}
Plugging the six estimates above into the abstract bound of Theorem~\ref{thm:abs.err.bound} concludes the proof.
\end{proof}

\section{A stabiliser-free $\mathbb{S}_{\Gamma}$ variant: suboptimal bound by projecting the discrete jump}
\label{sec:nostab}

The discrete jump~\eqref{eq:jump.n} is valued in $\Poly{k+1}(\sigma)^3$, whereas the
multiplier belongs to $\Ms$, whose components lie in $\Poly{\ell}(\sigma)^3$ with
$\ell\le k$. The mismatch between these two polynomial degrees on the fracture faces is exactly what the stabilisation term $\mathbb{S}_{\Gamma}$ in~\eqref{def.S.Gamma}
compensates for in the error analysis, through the bound on $I_{4\rm{c}}$ in the proof of Theorem~\ref{thm:abs.err.bound}. In this section, we show that projecting the discrete jump onto $\Poly{\ell}(\sigma)^3$ removes this mismatch, allows $\mathbb{S}_{\Gamma}$ to be discarded, and still yields a convergent scheme. 
The advantage is a scheme without direct connectivity of Lagrange multiplier degrees of freedom between adjacent faces, which may result in local elimination of those and reduction of the global system size (see Remark \ref{remark:eliminationlambda} below). The price to pay is a loss of rate of convergence in the consistency of jump reconstruction, and thus suboptimal error estimates (see also the remark \ref{remark:equivalencemixed} below for the specific case $\ell=0$).

We replace the discrete jump~\eqref{eq:jump.n} by its degree-$\ell$ projection on each side
of the fracture: for all $\uS{v}\in\ddgskS$ and all $\sigma\in\cF_{\Gamma}$,
\begin{equation}\label{eq:jump.modify}
\Jsig{\uS{v}}:=\Pi_{\sigma}^{\ell}\UpS{\Ksig}\uS{v} - \Pi_{\sigma}^{\ell}\UpS{\Lsig}\uS{v}
\;\in\;\Poly{\ell}(\sigma)^3 ,
\end{equation}
and we set $\mathbb{S}_{\Gamma}\equiv 0$ in the contact operator~\eqref{form1_ch}. The scheme~\eqref{eq:scheme} is otherwise unchanged.

Two changes occur in the proof of Theorem~\ref{thm:abs.err.bound}.
First, the term $I_{4\rm{c}}$ vanishes. For the previous jump \eqref{eq:jump.n},
identity~\eqref{eq:jump.approx} gives
$\Pi_{\sigma}^{\ell}\djumph{\RhS}-\djumph{\RhS}
=\Pi_{\sigma}^{\ell}\Upsilon_{\Ksig}^{k+1}\uvec{R}_{h}-\Upsilon_{\Ksig}^{k+1}\uvec{R}_{h}$,
the nonzero high-degree remainder that $\mathbb{S}_{\Gamma}$ was introduced to control (see \eqref{eq:prop.Rh.ups.bound}). With the modified jump~\eqref{eq:jump.modify}, we have
$\djumph{\RhS}\in\Poly{\ell}(\cF_{\Gamma})^3$ so $\Pi_{\cF_{\Gamma}}^{\ell}\djumph{\RhS}-\djumph{\RhS}=0$ and therefore $I_{4\rm{c}}=0$.
Consequently $\mathbb{S}_{\Gamma}$ no longer enters the bound on $I_{4\rm{c}}^{\prime}$, and
the two stabilisation contributions of Theorem~\ref{thm:abs.err.bound} namely, the energy-norm
term $\mathbb{S}_{\Gamma}(\Pi_{\cF_{\Gamma}}^{\ell}\laf-\laf_h,\Pi_{\cF_{\Gamma}}^{\ell}\laf-\laf_h)^{\frac{1}{2}}$
on the left and the consistency term
$\mathbb{S}_{\Gamma}(\Pi_{\cF_{\Gamma}}^{\ell}\laf,\Pi_{\cF_{\Gamma}}^{\ell}\laf)^{\frac{1}{2}}$
on the right are removed.

Second, the primal identity~\eqref{eq:D1.Rh2} is modified. Since
$\big(\djumph{\uS{v}}\big)_{|_\sigma}\in \Poly{\ell}(\sigma)^3,$ the multiplier may be projected onto $\Poly{\ell}(\sigma)^3$ when paired against it; equivalently, $\Pi_{\faces_{\Gamma}}^{\ell}\bbsigma(\bo{u})$ may replace $\bbsigma(\bo{u})$ in the traction sum over fracture faces. Repeating the steps leading to~\eqref{eq:D1.Rh2} (by introducing $\pm \Pi^\ell_{\faces_\Gamma}\bbsigma(\bo{u})\normal_{\Ksig}$ instead of $\pm \bbsigma(\bo{u})\normal_{\Ksig}$ in \eqref{eq:sigma.Rh2}), we obtain
\begin{align*}
\int_\Omega\bbsigma(\bo{u}):\bbeps_h(\RhS)
  + \int_\Gamma\laf \cdot \djumph{\RhS}
 &= -\sum_{K\in \cells}\sum_{\sigma\in\cF_{\Gamma,K}^+}\int_\sigma
      \UpS{\Ksig}\RhS\cdot\big(\Pi_{\faces_{\Gamma}}^{\ell}\bbsigma(\bo{u})
      -{\Pi}^k_K\bbsigma(\bo{u})\big)\normal_{\Ksig}\nonumber\\
 &= \sum_{K\in \cells}\sum_{\sigma\in\cF_{\Gamma,K}^+}\int_\sigma
      \UpS{\Ksig}\RhS\cdot\big(\Pi_{\faces_{\Gamma}}^{\ell}\laf-\laf\big)\nonumber\\
 &\quad -\sum_{K\in \cells}\sum_{\sigma\in\cF_{\Gamma,K}^+}\int_\sigma
      \UpS{\Ksig}\RhS\cdot\big(\bbsigma(\bo{u})-{\Pi}^k_K\bbsigma(\bo{u})\big)\normal_{\Ksig},
\end{align*}
where we have used $\gamma_{\normal}^{+}\bbsigma(\bo{u})=-\laf$ together with
$\Pi_{\faces_{\Gamma}}^{\ell}\gamma_{\normal}^{+}\bbsigma(\bo{u})
-\gamma_{\normal}^{+}\bbsigma(\bo{u})=-(\Pi_{\faces_{\Gamma}}^{\ell}\laf-\laf)$. The first term on the right-hand side is controlled by the multiplier-approximation quantity $\NORM{L^2(\Gamma)}{\gamma_h^{\frac{1}{2}}(\laf-\Pi_{\cF_{\Gamma}}^{\ell}\laf)}$ already present in Theorem~\ref{thm:abs.err.bound}, and the second is the term $I_{5\rm{c}}$; no new contribution is needed.

Collecting these two modifications, the discrete solution of the stabiliser-free scheme satisfies the abstract bound of Theorem~\ref{thm:abs.err.bound} with the  two
$\mathbb{S}_{\Gamma}$ contributions deleted, the jump-consistency term being now governed by Lemma~\ref{lem:consis.jump.recons_modified} below instead of Lemma~\ref{lem:consis.jump.recons}.

\begin{lemma}[Consistency of the modified jump reconstruction]\label{lem:consis.jump.recons_modified}
Let $\bu\in \bo{U}_0\cap H^{k+2}(\cM)^{3}$. Then
$$\NORM{L^2(\Gamma)}{\gamma^{-\frac{1}{2}}_{h}\big(\jump{\bo{u}} -\djumph{\IhS\bo{u}}\big)}
\lesssim h^{\ell+\frac{1}{2}}\left(\SEMINORM{H^{k+2}(\cM)}{\bu}
+\SEMINORM{H^{\ell+1}(\cF_\Gamma)}{\jump{\bu}} \right).$$
\end{lemma}
\begin{proof}
The proof follows the lines of Lemma~\ref{lem:consis.jump.recons} with one change. From \eqref{eq:jump.modify} and the polynomial consistency~\eqref{eq:UpS.face_1a} on each side of $\sigma$, the modified jump reproduces only the degree-$\ell$ part of a polynomial trace, so $\Jsig{\IhS\bo{q}}=\Pi_{\sigma}^{\ell}\jump{\bo{q}}$ while, in contrast  $\Jsig{\IhS\bo{q}}=\jump{\bo{q}}$ is used in Lemma~\ref{lem:consis.jump.recons}. Hence the triangle inequality gives
\begin{align*}
\NORM{L^2(\sigma)}{\gamma^{-\frac{1}{2}}_{\sigma}\big(\Jsig{\IhS\bu}-\jump{\bu}\big)}
&\le
  \NORM{L^2(\sigma)}{\gamma^{-\frac{1}{2}}_{h}\big(\jump{\IhS(\bu-\bq)}_\sigma\big)}
  + \NORM{L^2(\sigma)}{\gamma^{-\frac{1}{2}}_{h}\big(\Pi_{\sigma}^{\ell}\jump{\bq-\bu}\big)}  \\
&\qquad + \NORM{L^2(\sigma)}{\gamma^{-\frac{1}{2}}_{h}\big(\Pi_{\sigma}^{\ell}\jump{\bu}-\jump{\bu}\big)}\\
& \lesssim h_{\Gamma}^{-\frac{1}{2}}\left(\NORM{L^2(\sigma)}{\jump{\IhS(\bu-\bq)}_\sigma} + \NORM{L^2(\sigma)}{\jump{\bq-\bu}} + \NORM{L^2(\sigma)}{\Pi_{\sigma}^{\ell}\jump{\bu}-\jump{\bu}}\right)\\
\overset{\eqref{eq:approx.jump.uq},\eqref{eq:approx.jump.IUDuq}} &\lesssim
 h_{\Gamma}^{-\frac{1}{2}} \left(|\sigma|^{\frac{1}{2}}\big(h_K^{k+\frac{1}{2}}\SEMINORM{H^{k+2}(K)}{\bu}+h_L^{k+\frac{1}{2}}\SEMINORM{H^{k+2}(L)}{\bu}\big)
  +h_{\sigma}^{\ell+1}\SEMINORM{H^{\ell+1}(\sigma)}{\jump{\bu}}\right)\\
&\lesssim  h^{\ell+\frac{1}{2}}_{\sigma}\big(\SEMINORM{H^{k+2}(K)}{\bu}
  +\SEMINORM{H^{k+2}(L)}{\bu}+\SEMINORM{H^{\ell+1}(\sigma)}{\jump{\bu}}\big),
\end{align*}
where $\gamma_{h}=\gamma_{0}h_{\Gamma}$ and the boundedness property of $\Pi_{\sigma}^{\ell}$~(only for third term) are used in the second step. The third estimate, which is indeed the dominant term in the second-to-last step, follows from the approximation properties of $\Pi_{\sigma}^{\ell}$~\cite[Theorem~1.45]{hho-book}. The last estimate follows from the quasi-uniformity of the mesh on $\Gamma$~\eqref{eq:quasi.uniform.fracture}, which gives $h_\Gamma\approx h_\sigma$, combined with the mesh regularity assumptions $|\sigma|^{\frac{1}{2}}\approx h_{\sigma}$, $h_{K}\approx h_{\sigma},$ and $h_{L}\approx h_{\sigma}$; together these give $h_{\Gamma}^{-\frac{1}{2}} |\sigma|^{\frac{1}{2}} \approx h_{\sigma}^{-\frac12}\,h_\sigma = h_{\sigma}^{\frac{1}{2}}\approx h_{K}^{\frac{1}{2}} \approx h_{L}^{\frac{1}{2}}.$ Finally,
Squaring, summing over $\sigma\in\faces_\Gamma$ and taking the square root concludes the
proof.
\end{proof}
\begin{proof}[\textit{\textbf {Proof of Theorem~\ref{thm:error.estimate.sub}:}}] 
Repeating the proof of Theorem~\ref{thm:error.estimate} with Lemma~\ref{lem:consis.jump.recons_modified} in place of Lemma~\ref{lem:consis.jump.recons}, and with $\mathbb{S}_{\Gamma}\equiv 0$, yields the suboptimal estimate stated in Theorem~\ref{thm:error.estimate.sub}, of order
$h^{\ell+\frac{1}{2}}$. This is to be compared with the optimal rate $h^{\frakr+\frac{1}{2}}$ of
Theorem~\ref{thm:error.estimate}: the stabiliser-free variant loses half an order when $\ell=k$ and a full order when $\ell=k-1$. 
\end{proof}

\begin{remark}[Local elimination of the Lagrange multiplier]\label{remark:eliminationlambda}
We first notice that the Lagrange multiplier $\laf_h$ can be used in the variational mechanical equilibrium rather than the projection $\Pi_{\cF_{\Gamma}}^{\ell} \gamma^{-1}_h\bPg{\djumph{\uS{u}}}{\laf_h}$ resulting in a linear mechanical equilibrium equation. Substituting the unknowns $( \bo{v}_{\cK\sigma}, \bo{v}^{\bub}_{\Ksig}) \in \Poly{k-1}(\sigma)^{3} \times \Poly{\ell \perp k-1}(\sigma)^3$ by the discrete jump $\Jsig{\uS{v}} \in \Poly{\ell}(\sigma)^{3}$ for $\sigma \in \faces_\Gamma$ diagonalises the coupling operator between the displacement and Lagrange multiplier in the mechanical equilibrium equation. It results that the Lagrange multiplier can be locally eliminated from the mechanical equilibrium and substituted in the contact equation
$$
\laf_h = \Pi_{\cF_{\Gamma}}^{\ell}  \gamma^{-1}_h\bPg{\djumph{\uS{u}}}{\laf_h}. 
$$
This opens the way to efficient block preconditioners to solve the linear systems resulting from the semi-smooth Newton linearisation of the discrete non linear system.   
\end{remark}

\begin{remark}[Equivalence with mixed formulations for $\ell=0$]\label{remark:equivalencemixed}
In the specific case $\ell=0$ combined with the stabilisation free and the projected jump formulation, the contact equation reduces to 
$$
\laf_h = \gamma^{-1}_h\bPg{\djumph{\uS{u}}}{\laf_h},  
$$
which in turn is equivalent, for any $\gamma_0 > 0$, to the exact contact conditions for the discrete jump $\jump{\uS{v}}_{\sigma}$ and Lagrange multiplier $\laf_{\sigma}$ for each fracture face $\sigma \in \faces_\Gamma$. 
It results that the Augmented Lagrangian scheme is equivalent to the mixed formulation developed in \cite{droniou2023bubble,droniou2025analysis} for $k=0$ and in \cite{DKMR25} for $k=1$ for which improved error estimates have been derived for Tresca friction of order $h^{3/4}$ ($h$ in the frictionless case) for $k=\ell=0$,  and of order $h$  ($h^{3/2}$ in the frictionless case) for $k=1, \ell=0$. 
\end{remark}

\section{Conclusion}\label{sec:conclusion}
In this work, two augmented Lagrangian Discrete de Rham (DDR) schemes are proposed for contact mechanics problems subject to Tresca friction. The first formulation, based on the standard displacement jump, incorporates an edge-based jump stabilizer for the Lagrange multiplier and yields optimal error estimates. Conversely, the variant utilizing a projected jump bypasses both the stabilizer and its associated parameter~$\delta$, thereby enabling the local elimination of the multiplier. Although the analytical error bound derived for the projected-jump formulation is suboptimal, numerical experiments demonstrate convergence rates that significantly exceed this theoretical limit.

Intersecting-fracture benchmarks illustrate the practical implications of each approach. For the standard jump scheme, increasing the stabilization parameter $\delta$ successfully mitigates localized fluctuations in the normal multiplier, though excessive stabilisation artificially suppresses its peak magnitude. In contrast, the higher-order projected-jump schemes accurately capture the contact transition without tuning of~$\delta$. This observations suggest that, at high-order, the projected-jump variant provides a robust strategy and that the standard displacement jump with the contact stabilisation calls for an adaptive choice of the parameter $\delta$ following for instance the approach of \cite{MORETTO2026}. Deriving sharper analytical error estimates for the projected jump scheme remains an open question for future research.

\section*{Acknowledgement}
Funded by the European Union (ERC Synergy, NEMESIS, project number 101115663). Views and opinions expressed are, however, those of the authors only and do not necessarily reflect those of the European Union or the European Research Council Executive Agency. Neither the European Union nor the granting authority can be held responsible for them.
%
\printbibliography

@article{MORETTO2026,
title = {A stabilized mortar method for discontinuities in geological media with non-conforming grids},
journal = {Computer Methods in Applied Mechanics and Engineering},
volume = {460},
pages = {119087},
year = {2026},
issn = {0045-7825},
doi = {https://doi.org/10.1016/j.cma.2026.119087},
url = {https://www.sciencedirect.com/science/article/pii/S0045782526003609},
author = {Daniele Moretto and Andrea Franceschini and Massimiliano Ferronato}
}

@article{droniou2026locking,
  title={A locking-free nodal-based polytopal method for linear elasticity},
  author={Droniou, Jerome and Kumar, Raman},
  journal={arXiv preprint arXiv:2601.16728},
  year={2026}
}

@article{gmsh,
author = {Geuzaine, Christophe and Remacle, Jean-François},
title = {Gmsh: A 3-D finite element mesh generator with built-in pre- and post-processing facilities},
journal = {International Journal for Numerical Methods in Engineering},
volume = {79},
number = {11},
pages = {1309-1331},
doi = {https://doi.org/10.1002/nme.2579},
url = {https://onlinelibrary.wiley.com/doi/abs/10.1002/nme.2579},
eprint = {https://onlinelibrary.wiley.com/doi/pdf/10.1002/nme.2579},
year = {2009}
}

@article{burman2010interior,
  title={Interior-penalty-stabilized Lagrange multiplier methods for the finite-element solution of elliptic interface problems},
  author={Burman, Erik and Hansbo, Peter},
  journal={IMA journal of numerical analysis},
  volume={30},
  number={3},
  pages={870--885},
  year={2010},
  publisher={Oxford University Press}
}

@article{burman2010fictitious,
  title={Fictitious domain finite element methods using cut elements: I. A stabilized Lagrange multiplier method},
  author={Burman, Erik and Hansbo, Peter},
  journal={Computer Methods in Applied Mechanics and Engineering},
  volume={199},
  number={41-44},
  pages={2680--2686},
  year={2010},
  publisher={Elsevier}
}

@article {Chouly2014,
    AUTHOR = {Chouly, Franz},
     TITLE = {An adaptation of {N}itsche's method to the {T}resca friction
              problem},
   JOURNAL = {J. Math. Anal. Appl.},
  FJOURNAL = {Journal of Mathematical Analysis and Applications},
    VOLUME = {411},
      YEAR = {2014},
    NUMBER = {1},
     PAGES = {329--339},
      ISSN = {0022-247X,1096-0813},
   MRCLASS = {65N30 (35Q74 65N12 74M10 74M15 74S05)},
  MRNUMBER = {3118488},
MRREVIEWER = {Marius\ Ghergu},
       DOI = {10.1016/j.jmaa.2013.09.019},
       URL = {https://doi.org/10.1016/j.jmaa.2013.09.019},
}

@book{chouly2023finite,
  title={Finite element approximation of contact and friction in elasticity},
  author={Chouly, Franz and Hild, Patrick and Renard, Yves},
  year={2023},
  publisher={Springer}
}

@article{dong2022hybrid,
  title={Hybrid high-order and weak Galerkin methods for the biharmonic problem},
  author={Dong, Zhaonan and Ern, Alexandre},
  journal={SIAM Journal on Numerical Analysis},
  volume={60},
  number={5},
  pages={2626--2656},
  year={2022},
  publisher={SIAM}
}

@article{droniou2025analysis,
  title={Analysis of a VEM--fully discrete polytopal scheme with bubble stabilisation for contact mechanics with Tresca friction},
  author={Droniou, J{\'e}r{\^o}me and Haidar, Ali and Masson, Roland},
  journal={ESAIM: Mathematical Modelling and Numerical Analysis},
  volume={59},
  number={2},
  pages={1043--1074},
  year={2025},
  publisher={EDP Sciences}
}

@article{di2025uniform,
  title={Uniform Poincar$\backslash$'$\{$e$\}$ inequalities for the discrete de Rham complex of differential forms},
  author={Di Pietro, Daniele and Droniou, J{\'e}r{\^o}me and Hanot, Marien-Lorenzo and Pitassi, Silvano},
  journal={arXiv preprint arXiv:2501.16116},
  year={2025}
}

@Article{Di-Pietro.Droniou.ea:24,
  title         = {A pressure-robust {Discrete de Rham} scheme for the {Navier--Stokes} equations},
  author        = {Di Pietro, D. A. and Droniou, J. and Qian, J. J.},
  journal       = {Comput. Meth. Appl. Mech. Engrg.},
  year          = {2024},
  volume        = {421},
  number        = {116765},
  doi           = {10.1016/j.cma.2024.116765}
}

@article {Burman.Hansbo.Larson:17,
    AUTHOR = {Burman, Erik and Hansbo, Peter and Larson, Mats G.},
     TITLE = {Augmented {L}agrangian finite element methods for contact
              problems},
   JOURNAL = {ESAIM Math. Model. Numer. Anal.},
  FJOURNAL = {ESAIM. Mathematical Modelling and Numerical Analysis},
    VOLUME = {53},
      YEAR = {2019},
    NUMBER = {1},
     PAGES = {173--195},
      ISSN = {2822-7840,2804-7214},
   MRCLASS = {65N30 (65N12 74M15 74S05)},
  MRNUMBER = {3937350},
MRREVIEWER = {Corina-\c Stefania\ Drapaca},
       DOI = {10.1051/m2an/2018047},
       URL = {https://doi.org/10.1051/m2an/2018047},
}

@article{alart1991mixed,
  title={A mixed formulation for frictional contact problems prone to Newton like solution methods},
  author={Alart, Pierre and Curnier, Alain},
  journal={Computer methods in applied mechanics and engineering},
  volume={92},
  number={3},
  pages={353--375},
  year={1991},
  publisher={Elsevier}
}

@article{DKMR25,
  title={A higher order polyhedral method for contact mechanics with Tresca friction},
  author={Droniou, J{\'e}r{\^o}me and Kumar, Raman and Masson, Roland and Singla, Ritesh},
  journal={Journal of Computational Physics},
  pages={114923},
  year={2026},
  publisher={Elsevier}
}

@article{ddr-variant,
	AUTHOR = {Di Pietro, Daniele A. and Droniou, J\'er\^ome},
  TITLE = {An arbitrary-order discrete de Rham complex on polyhedral meshes: Exactness, Poincar\'e inequalities, and consistency},
	JOURNAL={Found. Comput. Math.},
	VOLUME={23},
  PAGES={85--164},
  YEAR={2023},
  DOI={10.1007/s10208-021-09542-8},
  URL={https://arxiv.org/abs/2101.04940},
}

@article{WRR2016,
	author = {Wriggers, P. and Rust, W. T. and Reddy, B. D.},
	journal = {Computational Mechanics},
	number = {6},
	pages = {1039--1050},
	title = {A virtual element method for contact},
	volume = {58},
	year = {2016}}

@article{CEP20,
author = {Chouly, Franz and Ern, Alexandre and Pignet, Nicolas},
title = {A Hybrid High-Order Discretization Combined with Nitsche's Method for Contact and Tresca Friction in Small Strain Elasticity},
journal = {SIAM Journal on Scientific Computing},
volume = {42},
number = {4},
pages = {A2300-A2324},
year = {2020}
}

@article{droniou2023bubble,
title = {A bubble VEM--fully discrete polytopal scheme for mixed-dimensional poromechanics with frictional contact at matrix-fracture interfaces},
journal = {Computer Methods in Applied Mechanics and Engineering},
volume = {422},
pages = {116838},
year = {2024},
issn = {0045-7825},
doi = {https://doi.org/10.1016/j.cma.2024.116838},
url = {https://www.sciencedirect.com/science/article/pii/S004578252400094X},
author = {Droniou, J\'er\^ome and Ench\'ery, Guillaume and Faille, Isabelle and Haidar, Ali and Masson, Roland},
}

@article{Renard03,
	author = {Ben Belgacem, F. and Renard, Y.},
	journal = {Mathematics of Computation},
	number = {},
	pages = {1117--1145},
	title = {Hybrid Finite Element Methods for the Signorini problem},
	volume = {72},
	year = {2003}}

@article{BHL2023,
	author = {Burman, Erik and Hansbo, Peter and Larson, Mats G.},
	journal = {Archives of Computational Methods in Engineering},
	number = {4},
	pages = {2579--2604},
	title = {The Augmented {Lagrangian} Method as a Framework for Stabilised Methods in Computational Mechanics},
	volume = {30},
	year = {2023}}

@article{NEJATI2016123,
title = {A finite element framework for modeling internal frictional contact in three-dimensional fractured media using unstructured tetrahedral meshes},
journal = {Computer Methods in Applied Mechanics and Engineering},
volume = {306},
pages = {123-150},
year = {2016},
issn = {0045-7825},
doi = {https://doi.org/10.1016/j.cma.2016.03.028},
url = {https://www.sciencedirect.com/science/article/pii/S004578251630113X},
author = {Morteza Nejati and Adriana Paluszny and Robert W. Zimmerman}
}

@article{BoonNordbotten22,
	Author = {Boon, W. M. and Nordbotten, J. M.},
	Da = {2022/12/04},
	Doi = {10.1007/s00707-022-03378-1},
	Id = {Boon2022},
	Isbn = {1619-6937},
	Journal = {Acta Mechanica},
	Title = {Mixed-dimensional poromechanical models of fractured porous media},
	Ty = {JOUR},
	Url = {https://doi.org/10.1007/s00707-022-03378-1},
	Year = {2022}}

@article{Lleras-2009,
	author = {Lleras, V.},
	title = {A Stabilized {Lagrange} Multiplier Method for the Finite Element
Approximation of Frictional Contact Problems in Elastostatics},
	DOI= "10.1051/mmnp/20094108",
	url= "https://doi.org/10.1051/mmnp/20094108",
	journal = {Mathematical Modelling of Natural Phenomena},
	year = 2009,
	volume = 4,
	number = 1,
	pages = "163-182",
}

@article {BDMP:21,
    AUTHOR = {Bonaldi, F. and Droniou, J. and Masson, R. and
              Pasteau, A.},
     TITLE = {Energy-stable discretization of two-phase flows in deformable
              porous media with frictional contact at matrix-fracture
              interfaces},
  FJOURNAL = {J. Comput. Phys.},
  JOURNAL = {Journal of Computational Physics},
    VOLUME = {455},
      YEAR = {2022},
     PAGES = {Paper No. 110984, 28},
      ISSN = {0021-9991},
   MRCLASS = {65M08 (65M12 76S05 76T06)},
  MRNUMBER = {4371138},
       DOI = {10.1016/j.jcp.2022.110984},
       URL = {https://doi.org/10.1016/j.jcp.2022.110984},
}

@incollection{HASLINGER1996313,
title = {Numerical methods for unilateral problems in solid mechanics},
series = {Handbook of Numerical Analysis},
publisher = {Elsevier},
volume = {4},
pages = {313-485},
year = {1996},
booktitle = {Finite Element Methods (Part 2), Numerical Methods for Solids (Part 2)},
issn = {1570-8659},
doi = {https://doi.org/10.1016/S1570-8659(96)80005-6},
url = {https://www.sciencedirect.com/science/article/pii/S1570865996800056},
author = {Haslinger J.  and Hlav\'acek, I. and Necas, J.}
}

@book{book-contact-2005,
	author = {Eck, C. and Jaru\v{s}ek, J. and Krbec, C.},
	doi = {10.1201/9781420027365},
	publisher = {CRC Press},
	title = {Unilateral Contact Problems: Variational Methods and Existence Theorems},
	year = {2005}}

@article{JHA2014,
	author = {Jha, B. and Juanes, R.},
	doi = {https://doi.org/10.1016/j.egypro.2014.11.360},
	issn = {1876-6102},
	journal = {Energy Procedia},
	note = {12th International Conference on Greenhouse Gas Control Technologies, GHGT-12},
	pages = {3313-3329},
	title = {Coupled Modeling of Multiphase Flow and Fault Poromechanics During Geologic CO2 Storage},
	url = {https://www.sciencedirect.com/science/article/pii/S1876610214021754},
	volume = {63},
	year = {2014}}

@book{oden-contact,
	author = {Kikuchi, N. and Oden, J.T.},
	publisher = {SIAM Society for Industrial and Applied Mathematics},
	series = {Studies in Applied and Numerical Mathematics},
	title = {Contact Problems in Elasticity: A Study of Variational Inequalities and Finite Element Methods},
	year = {1988}}

@book{wriggers2006,
	author = {Wriggers, P.},
	edition = {2nd},
	publisher = {Springer},
	title = {Computational Contact Mechanics},
	year = {2006}}

@unpublished{CHLR2020,
	author = {Chouly, F. and Hild, P. and Lleras, V. and Renard, Y.},
	note = {Preprint \hal{hal-02938032}},
	title = {{Nitsche method for contact with Coulomb friction: existence results for the static and dynamic finite element formulations}},
	year = {2020}}

@article{HR2007,
	author = {Hild, P. and Renard, Y.},
	journal = {SIAM Journal on Numerical Analysis},
	number = {5},
	pages = {2012-2031},
	title = {An Error Estimate for the {S}ignorini Problem with {C}oulomb Friction Approximated by Finite Elements},
	volume = {45},
	year = {2007}}

@article{GH19,
	author = {T.T. Garipov and M.H. Hui},
	doi = {https://doi.org/10.1016/j.ijrmms.2019.104075},
	issn = {1365-1609},
	journal = {International Journal of Rock Mechanics and Mining Sciences},
	pages = {104075},
	title = {Discrete Fracture Modeling approach for simulating coupled thermo-hydro-mechanical effects in fractured reservoirs},
	url = {https://www.sciencedirect.com/science/article/pii/S1365160918311018},
	volume = {122},
	year = {2019}}

@article{thm-bergen,
title = {A fully coupled numerical model of thermo-hydro-mechanical processes and fracture contact mechanics in porous media},
journal = {Computer Methods in Applied Mechanics and Engineering},
volume = {386},
pages = {114122},
year = {2021},
issn = {0045-7825},
doi = {https://doi.org/10.1016/j.cma.2021.114122},
url = {https://www.sciencedirect.com/science/article/pii/S0045782521004539},
author = {Stefansson, Ivar and Berre, Inga and Keilegavlen, Eirik},
}

@article{contact-norvegiens,
	author = {Berge, R. L. and Berre, I. and Keilegavlen, E. and Nordbotten, J. M. and Wohlmuth, B.},
	doi = {10.1002/nme.6238},
	journal = {International Journal for Numerical Methods in Engineering},
	pages = {644--663},
	title = {Finite volume discretization for poroelastic media with fractures modeled by contact mechanics},
	volume = {121},
	year = {2019}}

@article{Wohlmuth11,
	author = {Wohlmuth, B.},
	journal = {Acta Numerica},
	pages = {569--734},
	title = {Variationally consistent discretization schemes and numerical algorithms for contact problems},
	volume = {20},
	year = {2011}}

@inproceedings{Chouly2017,
	address = {Cham},
	author = {Chouly, F. and Fabre, M. and Hild, P. and Mlika, R. and Pousin, J. and Renard, Y.},
	booktitle = {Geometrically Unfitted Finite Element Methods and Applications},
	editor = {Bordas, St{\'e}phane P. A. and Burman, Erik and Larson, Mats G. and Olshanskii, Maxim A.},
	isbn = {978-3-319-71431-8},
	pages = {93--141},
	publisher = {Springer International Publishing},
	title = {An Overview of Recent Results on {N}itsche's Method for Contact Problems},
	year = {2017}}

@book{hho-book,
	TITLE={The Hybrid High-Order Method for Polytopal Meshes: Design, Analysis, and Applications},
	AUTHOR={Di Pietro, Daniele Antonio and Droniou, J\'er\^ome},
	YEAR={2020},
	PAGES={xxxi + 525p},
	SERIES = {Modeling, Simulation and Applications},
	PUBLISHER = {Springer International Publishing},
  volume        = {19},
  isbn          = {978-3-030-37202-6 (Hardcover), 978-3-030-37203-3 (eBook)},
  doi           = {10.1007/978-3-030-37203-3},
	URL={https://hal.archives-ouvertes.fr/hal-02151813},
	coverthumb = {images/cover-HHO_book-reduced.png},
	cover = {images/cover-HHO_book.png},
}

@article{keilegavlen-nordbotten,
	author = {Keilegavlen, E. and Nordbotten, J.M.},
	journal = {International Journal for Numerical Methods in Engineering},
	pages = {939--962},
	title = {Finite volume methods for elasticity with weak symmetry},
	volume = {112},
	year = {2017}}

@article{beirao-brezzi-marini,
	author = {Beir{\~a}o Da Veiga, L. and Brezzi, F. and Marini, L.D.},
	journal = {SIAM Journal on Numerical Analysis},
	pages = {794--812},
	title = {Virtual elements for linear elasticity problems},
	volume = {51},
	year = {2013}}

@article{dipietro-ern,
	author = {Di Pietro, D. and Ern, A.},
	journal = {Computer Methods in Applied Mechanics and Engineering},
	pages = {1--21},
	title = {A hybrid high-order locking-free method for linear elasticity on general meshes},
	volume = {283},
	year = {2015}}

@article{hansbo-larson,
	author = {Hansbo, P. and Larson, M.G.},
	journal = {ESAIM: Mathematical Modelling and Numerical Analysis},
	pages = {63--72},
	title = {Discontinuous {G}alerkin and the {C}rouzeix--{R}aviart element: application to elasticity},
	volume = {37},
	year = {2003}}

@article{MJE05,
	author = {Martin, V. and Jaffr\'e, J. and Roberts, J. E.},
	journal = {SIAM Journal on Scientific Computing},
	pages = {1667--1691},
	title = {Modeling fractures and barriers as interfaces for flow in porous media},
	volume = {26},
	year = {2005}}

@book{gdm,
	author = {Droniou, J. and Eymard, R. and Gallou\"et, T. and Guichard, C. and Herbin, R.},
	doi = {10.1007/978-3-319-79042-8},
	isbn = {978-3-319-79041-1 (Softcover) 978-3-319-79042-8 (eBook)},
	pages = {511p},
	publisher = {Springer},
	series = {Mathematics \& Applications},
	title = {The {G}radient {D}iscretisation {M}ethod},
	volume = {82},
	year = {2018}}

@article{beaude2023mixed,
  title={Mixed and Nitsche’s discretizations of Coulomb frictional contact-mechanics for mixed dimensional poromechanical models},
  author={Beaude, Laurence and Chouly, Franz and Laaziri, Mohamed and Masson, Roland},
  journal={Computer Methods in Applied Mechanics and Engineering},
  volume={413},
  pages={116124},
  year={2023},
  publisher={Elsevier}
}

@article{da2015virtual,
  title={A virtual element method for elastic and inelastic problems on polytope meshes},
  author={Da Veiga, L Beir{\~a}o and Lovadina, Carlo and Mora, David},
  journal={Computer methods in applied mechanics and engineering},
  volume={295},
  pages={327--346},
  year={2015},
  publisher={Elsevier}
}

\end{document}